\documentclass[a4paper, 12pt,reqno]{amsart}

\usepackage{amsmath, amsthm, amssymb, amsfonts, fullpage, hyperref,marginnote}
\usepackage{mathtools} \mathtoolsset{showonlyrefs}
\usepackage{graphicx, subfigure, color}
\graphicspath{ {images/} }
\numberwithin{equation}{section}

\title[Boundaries of sine kernel universality]{\LARGE \normalfont B\MakeLowercase{oundaries of sine kernel universality for} G\MakeLowercase{aussian perturbations of} H\MakeLowercase{ermitian matrices}
}
\author[]{Tom Claeys$^*$}
\address{$^*$Institute de Recherche en Math\'ematique et Physique, Universit\'e catholique de Louvain, Chemin du Cyclotron
2, B-1348 Louvain-La-Neuve, Belgium}
\email{tom.claeys@uclouvain.be}

\author[]{Thorsten Neuschel$^\dag$}
\address{$^\dag$Institute de Recherche en Math\'ematique et Physique, Universit\'e catholique de Louvain, Chemin du Cyclotron
2, B-1348 Louvain-La-Neuve, Belgium}
\email{thorsten.neuschel@uclouvain.be}

\author[]{Martin Venker$^\ddag$}
\address{$^\ddag$Department of Mathematics, Ruhr University Bochum, 44780 Bochum,
Germany}
\email{martin.venker@ruhr-uni-bochum.de}

\newcommand{\R}{\mathbb{R}}
\newcommand{\C}{\mathbb{C}}

\DeclareMathOperator{\PE}{PE}

\newtheorem{thm}{Theorem}[section]
\newtheorem{cor}{Corollary}[section]
\newtheorem{lemma}{Lemma}[section]

\newtheorem{prop}{Proposition}[section]
\theoremstyle{remark}
\newtheorem{rmk}{Remark}[section]

\newcommand{\lb}{\left(}
\newcommand{\rb}{\right)}

\renewcommand{\O}{\mathcal O}
\newcommand{\lv}{\lvert}
\newcommand{\rv}{\rvert}
\renewcommand{\k}{\kappa}
\renewcommand{\t}{\tau}
\newcommand{\g}{\gamma}
\renewcommand{\epsilon}{\varepsilon}

\usepackage[normalem]{ulem}

\makeatletter
\renewcommand\section{\@startsection {section}{1}{\z@}%
                                   {-3.5ex \@plus -1ex \@minus -.2ex}%
                                   {2.3ex \@plus.2ex}%
                                   {\normalfont\Large\bfseries}}
\makeatother

\makeatletter
\renewcommand\subsection{\@startsection{subsection}{2}{\z@}%
                                     {-3.25ex\@plus -1ex \@minus -.2ex}%
                                     {1.5ex \@plus .2ex}%
                                     {\normalfont\large\bfseries}}
\makeatother

\begin{document}

\begin{abstract}
We explore the boundaries of sine kernel universality for the eigenvalues of Gaussian perturbations of large deterministic Hermitian matrices.  Equivalently, we study for deterministic initial data the time after which Dyson's Brownian motion exhibits sine kernel correlations. We explicitly describe this time span in terms of the limiting density and rigidity of the initial points.  Our main focus lies on cases where the initial density vanishes at an interior point of the support. We show that the time to reach universality becomes larger if the density vanishes faster or if the initial points show less rigidity.
\end{abstract}

\maketitle

\section{Introduction and main results}
It is well known that eigenvalues of large random matrices exhibit a highly universal behavior in the sense that the local 
limiting distributions depend only on few characteristics of the underlying matrix distribution. Typically, symmetries of the random 
matrices divide the models into universality classes. In this paper, we will deal with random Hermitian matrices of the form
\begin{align}
Y_n(t):=M_n+\sqrt{t}H_n,\label{def_matrix}
\end{align}
where $M_n$ is a deterministic $n\times n$ Hermitian matrix, $t>0$ and $H_n$ is an $n\times n$ random 
matrix sampled from the Gaussian Unitary Ensemble (GUE), i.e., from the distribution with density proportional to
\begin{align}
e^{-\frac n2\textup{Tr}(H_n^2)}\label{def_GUE}
\end{align}
on the space $\mathcal M_n$ of $n\times n$ complex Hermitian matrices with respect to the Lebesgue measure. Equivalently, the entries of $H_n$ 
on and above the diagonal are independent (with independent real and imaginary parts) with $H_{n,ii},\Re H_{n,ij},\Im H_{n,ij}, i<j$ being 
normally distributed with mean 0 and variance $1/(n(2-\delta_{ij}))$. 

The parameter $t$ will be interpreted as time since $\sqrt{t}H_n$ in \eqref{def_matrix} has the same distribution for fixed $t$ as $B_t/\sqrt n$, 
where $(B_t)_{t\geq0}$ is a Brownian motion on the space $\mathcal M_n$. The division by $\sqrt n$ is convenient for considering 
eigenvalues in the large $n$ limit as this rescaling will result in an almost surely compact spectrum. Thus $Y_n(t)$ has for given $t$ 
the same distribution as a rescaled Hermitian Brownian motion with initial configuration $M_n$. The corresponding 
eigenvalue process is called Dyson's Brownian motion \cite{Dyson}, see Figure \ref{figure1}.
\begin{figure}[h]
\centering
      \begin{minipage}[t]{0.45\linewidth}
			\centering
			 \includegraphics[trim=5cm  5cm 5cm 9cm,clip,width=\linewidth]{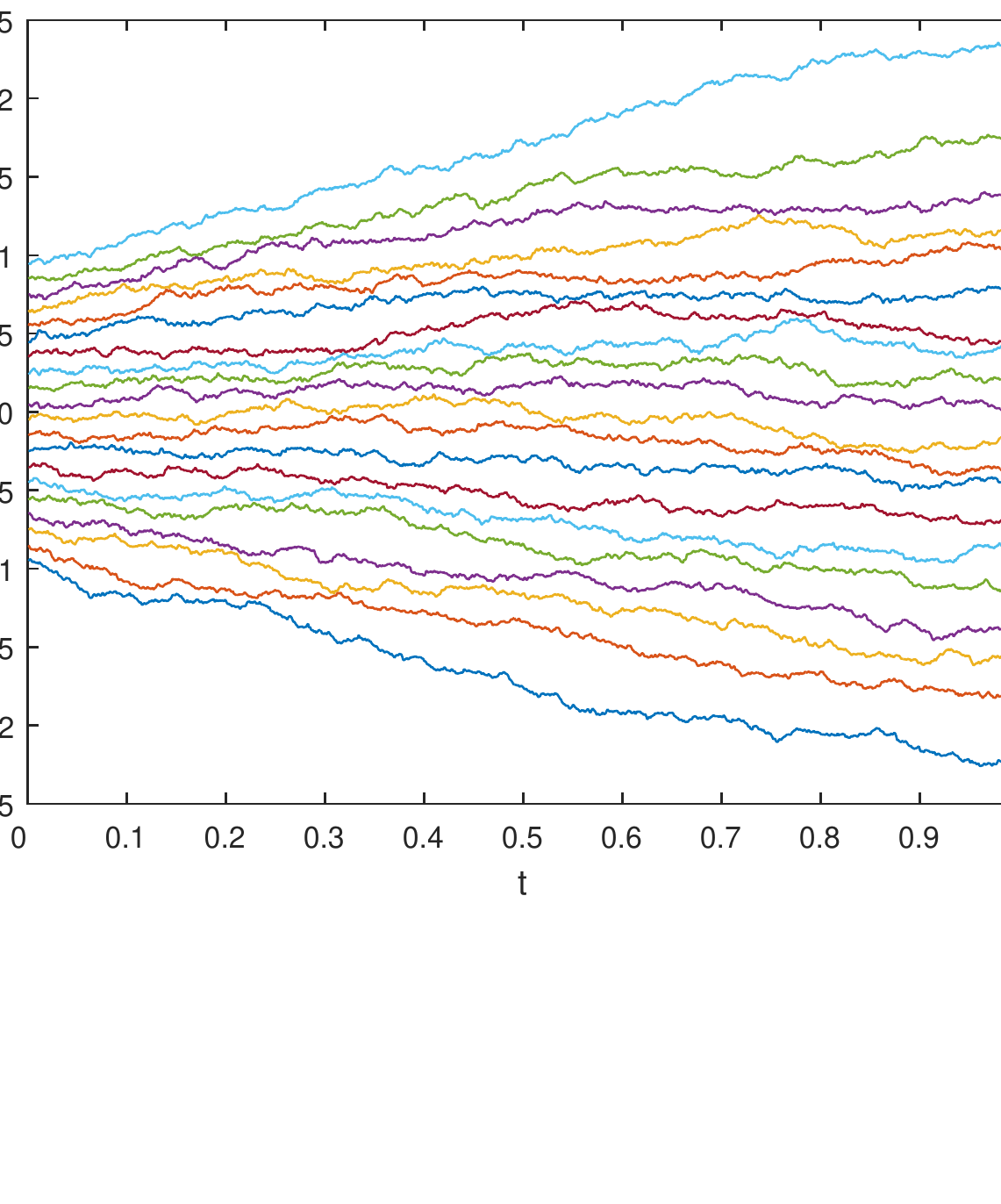}
			\end{minipage}%
			\hfill
			\begin{minipage}[t]{0.45\linewidth}
			\centering
\includegraphics[trim=5cm  5cm 5cm 9cm,clip,width=\linewidth]{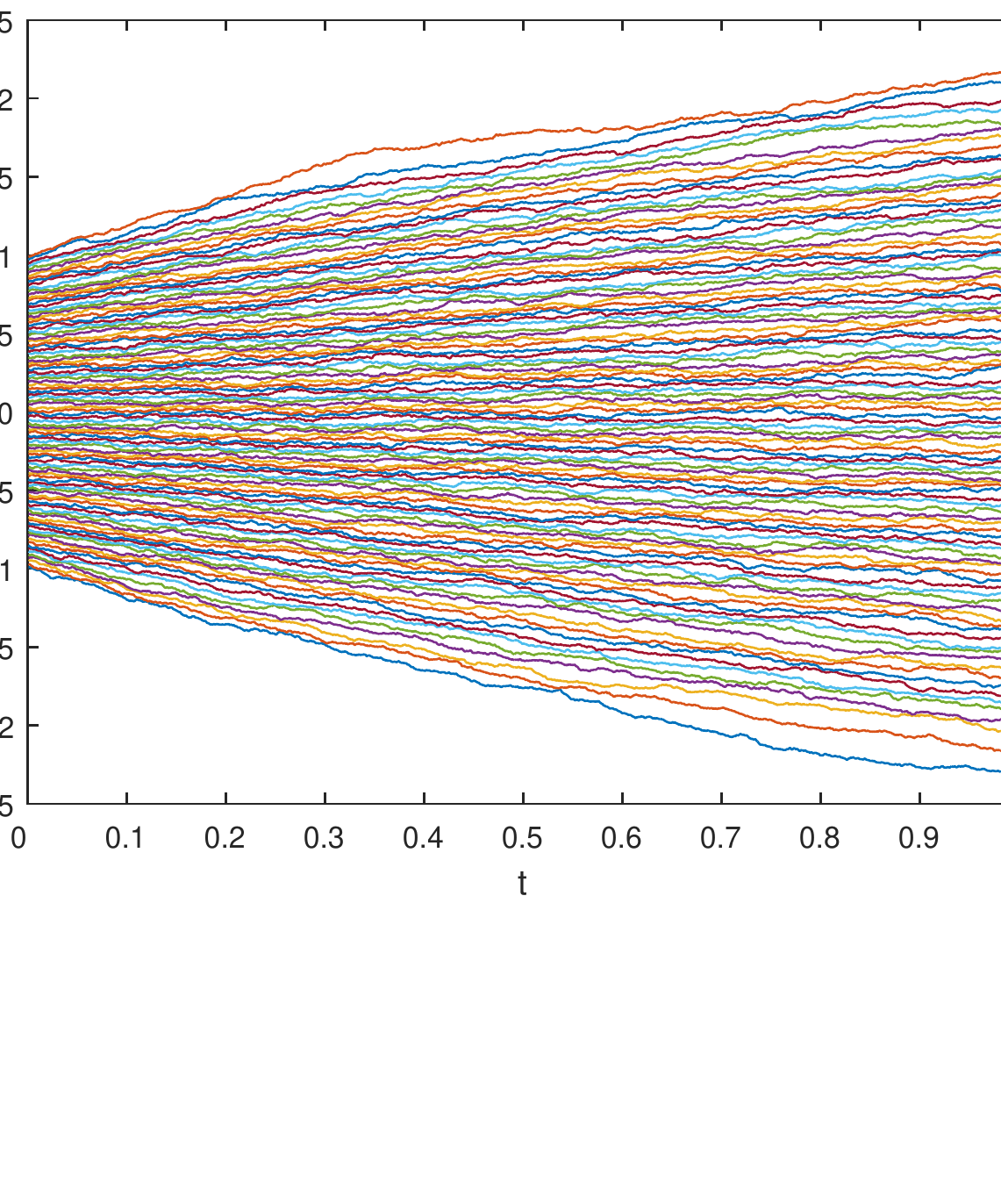} 
			\end{minipage} \vspace{-5em}\caption{Two samples of Dyson's Brownian motion for $t\in[0,1]$, with $n=20$ (left) and $n=100$ (right), for $M_n$ having equi-spaced eigenvalues on $[-1,1]$.}
			\label{figure1}
\end{figure}

Equivalently, $Y_n(t)$ may be seen as being sampled from a GUE with external source, i.e., a random Hermitian matrix $Y_n$ with density proportional to 
\[e^{-\frac{n}{2t}\textup{Tr}\left((Y_n-M_n)^2\right)}.\]
As 
$Y_n(t)$ is a random perturbation of the deterministic matrix $M_n$, it is an intriguing question to ask how large the time $t$ has to be in order 
to observe the universal local eigenvalue statistics well-known for many classes of random Hermitian matrices as the size $n$ tends to infinity. This question has been addressed in many recent papers, see \cite{bookErdosYau} for an overview.  From an intuitive point of view, if the eigenvalues of $M_n$ are sufficiently dense near a point $x^*$, then the time needed to reach universality is small; near points where the eigenvalues of $M_n$ are less dense or near points where there are large spacings between eigenvalues, the time needed to reach universality is larger. We aim at an explicit description of the mechanism behind the interaction between the initial configuration of eigenvalues and the time needed to reach local universality. In particular, we focus on the case where the limiting eigenvalue density vanishes at an interior point of its support, for which no previous results were available to the best of our knowledge. 
As usual, we 
will measure universality in terms of correlation functions of the eigenvalues.

In order to recapitulate the precise terms, let us consider the GUE \eqref{def_GUE} ($M_n=0$ and $t=1$ in \eqref{def_matrix}) as an example.  For $1\leq k\leq n$, the $k$-th correlation function 
$\rho^k_n$ of a density $P_n$ on $\R^n$ is the following multiple of the $k$-th marginal density
\begin{align}
\rho_n^k(x_1,\dots,x_k):=\frac{n!}{(n-k)!}\int_{\R^{n-k}}P_n(x_1,\dots,x_n)dx_{k+1}\dots dx_n.\label{correlation_functions}
\end{align}
The joint density $P_n$ of the eigenvalues of a GUE matrix is proportional to
\begin{align}\label{GUE_ev}
\prod_{i<j}(x_i-x_j)^2e^{-\frac n2\sum_{j=1}^nx_j^2}
\end{align}
and is determinantal, i.e., its correlation functions $\rho_n^k$ can be written in terms of a kernel $K_n$ as
\begin{align}
\rho_n^k(x_1,\dots,x_k)=\det\lb K_n(x_i,x_j)\rb_{1\leq i,j\leq k}.\label{determinantal}
\end{align}
As the matrix on the right-hand side of \eqref{determinantal} is of fixed order $k$, finding the large $n$ asymptotics of the $k$-th correlation function reduces to finding large $n$ asymptotics of the kernel 
$K_n$.
There is a classical formula for $K_n(x,y)$ in terms of Hermite polynomials but for us the following, also well-known, representation as a 
double contour integral is more useful,
\begin{align}\label{GUE_kernel_alternative}
K_n(x, y) = \frac{n}{(2\pi i)^2}\int_{x_0-i\infty}^{x_0+i\infty}dz\int_\gamma dw 
\frac{1}{z-w}\frac{e^{\frac{n}{2}\left( (z-x)^2+2\log z\right)}}{e^{\frac{n}{2} \left((w-y)^2+2\log w\right)}},
\end{align}
where $\g$ encircles the origin and $x_0$ is such that the two contours do not intersect.

In the so-called global regime, one has for $x\in\R$ that the limiting mean density of eigenvalues is almost surely given by the Wigner semicircle density,
\begin{align}\label{SC}
\lim_{n\to\infty}\frac 
1n\rho_n^1(x)=\lim_{n\to\infty}\frac{1}{n}K_n(x,x)=\frac1{2\pi}\sqrt{4-{x}^2}1_{[-2,2]}(x)=:\sigma(x).
\end{align}
The global limiting eigenvalue density is usually model-dependent and not 
universal. In contrast to this, the local correlations of eigenvalues show universal behavior, which for points $x^*\in(-2,2)$ in the bulk 
of the spectrum means 
\begin{align}\label{GUE_correlations}
\lim_{n\to\infty}\frac{1}{n\sigma(x^*)}K_n\lb x^*+\frac u{n\sigma(x^*)},x^*+\frac 
v{n\sigma(x^*)}\rb=\frac{\sin(\pi(u-v))}{\pi(u-v)}.
\end{align}
Here, the convergence is locally uniform for $u,v\in\R$ and the sine kernel at the right-hand side is understood as being 1 if $u=v$. The limit 
\eqref{GUE_correlations} is called local as it concerns correlations on the scale on which eigenvalues around the point $x^*$ have a 
mean spacing of order 1. 

The sine kernel has been found to appear for large classes of Hermitian random matrix models and is thus 
called universal. Given the large amount of research on universality of the sine kernel, we can only give a very partial list of 
references. For the results on bulk universality of the unitary invariant ensembles with density proportional to 
$\exp(-n\textup{Tr}V(M))$ on the matrix space, we mention \cite{Deiftetal} for the Riemann-Hilbert method, 
\cite{PS97,PS08} for an approach closer to mathematical physics and \cite{LevinLubinsky} for a comparative analytic approach. Two 
recent surveys are \cite{Lubinsky-Survey} and \cite{Kuijlaars-Survey}. For bulk universality of Wigner matrices, i.e., random 
matrices with as many independent entries as possible, we refer to \cite{TaoVu} and \cite{Erdosetal10,BEYY, bookErdosYau}.

Apart from these two by now classical situations, the sine kernel has also been shown to appear in a number of different models, among 
them sparse matrices \cite{HLY} and matrices with correlated entries (but without unitary invariance) \cite{AEK}. Other classes are more 
general particle systems with quadratic repulsion \cite{GV14} or several-matrix-models \cite{FG}. Different scalings like the 
unfolding have been considered in \cite{SV15}, giving rates of convergence and extending the uniformity in statements like 
\eqref{GUE_correlations}.

For the model \eqref{def_matrix}, sine kernel universality in the bulk has been shown for fixed $t>0$ under general conditions on $M_n$ \cite{Shcherbina,LeeSchnelliStetlerYau} and for $t$ converging to $0$ as $n\to 0$ at a sufficiently slow rate if $M_n$ is a random matrix \cite{Johansson2, Erdosetal10, LandonSosoeYau, LandonYau}. Near the edge of the spectrum, Airy kernel universality was shown for a large class of matrices $M_n$ in \cite{Shcherbina2} for $t>0$ fixed and for $t\to0$ in \cite{BEYedge,LandonYauedge}. For certain special choices of $M_n$, it is known that other limiting kernels can appear. If $M_n$ has only two distinct eigenvalues, the Pearcey kernel arises at a critical time \cite{BrezinHikami2, BrezinHikami, TracyWidom06, BleherKuijlaars, AdlerOrantinvanMoerbeke}, and generalizations of the Airy kernel can appear at the edge \cite{AdlerDelepinevanMoerbeke, AdlerFerrarivanMoerbeke, BertolaBuckinghamLeePierce1, BertolaBuckinghamLeePierce2}.
Some of these results have been obtained using a representation of the eigenvalue correlation kernel (see \eqref{eq:Kncontour3} below) in terms of multiple orthogonal polynomials. The asymptotic analysis of these multiple orthogonal polynomials can be performed using Riemann-Hilbert methods if the support of $\mu_n$ consists of a small number of points, or in very special situations like equi-spaced points \cite{ClaeysWang}, but so far not for general eigenvalue configurations $\mu_n$.

Let us now turn to our results about the model \eqref{def_matrix}. It is known since works by Br\'ezin and Hikami 
(cf.~\cite{BrezinHikami} and references therein) that the eigenvalue distribution $P_{n,t}$ of $Y_n(t)$ is determinantal in the sense of 
\eqref{determinantal} with a kernel also given as a double contour integral. To write down a formula for the kernel, let 
the 
deterministic eigenvalues of $M_n$ be $a_1^{(n)},\dots,a_n^{(n)}$ and denote the associated empirical spectral measure by $\mu_n$, 
i.e.,
\begin{align*}
\mu_n:=\frac1n\sum_{j=1}^n\delta_{a_j^{(n)}}.
\end{align*}
Moreover let 
\begin{equation}\label{def:g}
g_{\mu_n}(z):=\int\log (z-x) d\mu_n(x),
\end{equation}
where we take the principal branch of the logarithm.
 Then the correlation functions $\rho_{n,t}^k$ of $P_{n,t}$, defined as in \eqref{correlation_functions}, satisfy for $1\leq k\leq n$ 
\begin{align}
\rho_{n,t}^k(x_1,\dots,x_k)=\det\lb K_{n,t}(x_i,x_j)\rb_{1\leq i,j\leq k},\label{determinantal_K_{n,t}}
\end{align}
where the kernel $K_{n,t}$ is defined as
	\begin{align} 
	K_{n,t}(x, y) :=& e^{-\frac{n}{2t}(x^2-y^2)+(x-y)x_0\frac{n}{t}}\frac{n}{(2\pi i)^2t}\int_{x_0-i\infty}^{x_0+i\infty}dz\int_\gamma dw 
\frac{1}{z-w}\frac{e^{\frac{n}{2t}\left( (z-x)^2+2t g_{\mu_n}(z)\right)}}{e^{\frac{n}{2t} \left((w-y)^2+2t g_{\mu_n}(w)\right)}}\notag\\
	&+\frac{1}{\pi (x-y)}\sin\left((x-y)s\frac{n}{t}\right).\label{eq:Kncontour3}
	\end{align}
	Here, $\gamma$ is a contour encircling all the points $a_j^{(n)}$'s in the positive sense and such that $x_0+i\mathbb R$ and $\gamma$ intersect precisely at the two points $\t_1=x_0+is$ and $\t_2=x_0-is$, with $s>0$ and such that the line segment $(\t_1,\t_2)$ lies in the interior of $\gamma$. 
Explicit contour integral formulas for the correlation kernel $K_{n,t}$ go back to \cite{BrezinHikami}, see also \cite{Johansson01}. Our formula \eqref{eq:Kncontour3} is a variant of these existing formulas which is particularly convenient for asymptotic analysis in the bulk, as it is decomposed in a way that suggests convergence to the sine kernel.
For the convenience of the reader, we give a self-contained proof of \eqref{determinantal_K_{n,t}}--\eqref{eq:Kncontour3} in Appendix \ref{appendix}.

Throughout the paper we will make the following assumptions.
\vspace{8pt}

\noindent
\textbf{Assumption 1.} 
The probability measures  $\mu_n$ converge as $n\to\infty$ weakly to an absolutely continuous probability measure $\mu$ with compact 
support and with a density which is continuous as a function restricted to the support.

\vspace{8pt}

\noindent
\textbf{Assumption 2.} 
There is a compact subset of $\R$ which contains the points $a_j^{(n)}$ for all $j$ and all $n$.

\vspace{8pt}

We believe that Assumption 2 is purely technical and could be removed with some extra work.

If $H_n$ is a GUE matrix, then the limiting measure of $\sqrt{t}H_n$ will be the semicircle distribution 
\begin{align*}
d\sigma_t(x):=\frac{1}{2\pi t}\sqrt{4t-x^2}1_{[-2\sqrt{t},2\sqrt{t}]}(x)dx.
\end{align*} 
The eigenvalues of the model $Y_n(t)$ have for any $t$ a global limiting measure $\mu_t$ in the sense of \eqref{SC}. This measure is determined by $\mu$ and 
$\sigma_t$ and is called the additive free convolution of $\mu$ and $\sigma_t$ \cite{AndersonGuionnetZeitouni, HiaiPetz, Speicher, VoiculescuDykemaNica}, in symbols
\begin{align*}
\mu_t:=\mu\boxplus\sigma_t.
\end{align*}
The measure $\mu_t$ has a density for any $t>0$ which will be denoted by $\psi_t$ whereas the density of $\mu$ will be denoted by $\psi$.
As mentioned above, we will focus on bulk correlations. That is, we want to investigate the local correlations around points in 
the interior of the support of $\psi_t$. We aim at an explicit characterization of bulk points in terms of the initial limiting measure $\mu$ instead of an implicit one in terms of the measure $\mu_t$.

For that purpose, we define for any $x\in\mathbb R$
the path $t\mapsto x_t$ with
\begin{align}\label{H_map}
x_t:=x+t\int\frac{(x-u)\psi(u)du}{(x-u)^2+y_{t,\mu}(x)^2},
\end{align}
with $y_{t,\mu}$ given by
\begin{align} \label{def:y_t} 
y_{t,\mu}(x) := \inf \left\{ y > 0 \, \bigg| \, \int \frac{d\mu(s)}{(x-s)^2 + y^2} \leq \frac{1}{t} \right\}.
\end{align}
The definitions of this path and of the map $x\mapsto y_{t,\mu}(x)$ originate in work by Biane \cite{Biane} about the free additive convolution with a semicircle distribution.
It was shown there that the identity
\begin{equation}\label{eq:psitau}
\psi_{t}(x_{t})=\frac{y_{t,\mu}(x)}{\pi t}
\end{equation}
holds, and this implies in particular that $x\mapsto x_t$ defines a bijection between the support of the function $y_{t,\mu}$ and the support of the measure $\mu_t$. Note also that the support of $y_{t,\mu}$ increases with $t$ and that it contains the support of $\mu$ for any $t>0$.

Given $x^*\in\mathbb R$, there are two fundamentally different possibilities: either $\int\frac{d\mu(s)}{(s-x^*)^2}$ is finite, or it diverges. In the first case, it follows from \eqref{def:y_t} 
that $y_{t,\mu}(x^*)=0$ for $t\leq t_{cr}$ with the critical time $t_{cr}$ defined as
\begin{align}\label{t_cr}
t_{cr}=t_{cr}(x^*):=\lb\int\frac{d\mu(s)}{(s-x^*)^2}\rb^{-1}. 
\end{align} 
In the latter case, we set $t_{cr}(x^*):=0$, and then $y_{t,\mu}(x^*)>0$ for all $t>0$.
It is easy to see from \eqref{H_map} that $t\mapsto x_t$ is a linear path for $0<t<t_{cr}$. From \eqref{eq:psitau}, we see that $\psi_t(x_t^*)=0$ for $0<t\leq t_{cr}(x^*)$ and that $\psi_t(x_t^*)>0$ for $t>t_{cr}(x^*)$.

We note that $t_{cr}(x^*)=0$ if $\psi(x^*)>0$ and also if
$x^*$ is a zero of $\psi$ of algebraic order $0<\kappa\leq 1$, i.e.,
\begin{equation}\label{assumption_decay}
\psi(x)\sim C\lv x-x^*\rv^\k,\qquad x\to x^*,
\end{equation}
where $C>0$ is some constant and $\sim$ denotes leading order behavior.
On the other hand, if $x^*$ lies outside the support of $\mu$ and also if $x^*$
is a zero of $\psi$ of algebraic order $\kappa> 1$, we have $t_{cr}(x^*)>0$. We refer to \cite{ClaeysKuijlaarsLiechtyWang} for more details about the behavior of $\psi_t$ in this last case.

Our first result is merely a slight generalization of existing results, but it constitutes an important natural step for understanding and proving our next results.
It deals with fixed times $t$ (independent of $n$) bigger than $t_{cr}$.
\begin{thm}\label{thrm_t_fixed}
Let $\mu_n$ and $\mu$ satisfy Assumptions 1 and 2, let $x^*\in\mathbb R$ and let $t>t_{cr}(x^*)$ be fixed. Uniformly for $u,v$ in compacts 
of $\R$, we have
\begin{align*}
\lim_{n\to\infty}\frac{1}{n\psi_t(x^*_t)}K_{n,t}\lb x^*_t+\frac u{n\psi_t(x^*_t)},x^*_t+\frac 
v{n\psi_t(x^*_t)}\rb=\frac{\sin(\pi(u-v))}{\pi(u-v)},
\end{align*}
where $x_t^*$ is given by \eqref{H_map}.
\end{thm}
\begin{rmk}
The essence of the proof of Theorem \ref{thrm_t_fixed} goes back to the work of Johansson \cite{Johansson01} which itself has been inspired by 
the work of Brezin and Hikami \cite{BrezinHikami}. In \cite{Johansson01}, $M_n$ is a Wigner matrix (independent of $H_n$), i.e., $M_n$ is a random 
Hermitian matrix with as many independent entries as Hermitian symmetry allows. As the limiting measure for Wigner matrices with entries 
having mean $0$ and the same variance $s>0$ is always the semicircle distribution $\sigma_s$ and $\sigma_s\boxplus\sigma_t=\sigma_{s+t}$, Johansson's 
proof does not deal with a general $\psi$. It has been subsequently extended to sample covariance matrices \cite{BenArousPeche} and Wigner matrices under weak moment assumptions \cite{Johansson2}. A variant of Theorem \ref{thrm_t_fixed} that covers also real symmetric matrices has been shown in \cite{LeeSchnelliStetlerYau}. In that paper, $\psi$ is such that $\psi_t$ has a connected compact support with strictly positive density on the interior. The case of a general $\psi$ was studied by T. Shcherbina in \cite{Shcherbina}, where she proved bulk universality via the method of supersymmetry. Both \cite{LeeSchnelliStetlerYau,Shcherbina} characterize the bulk points as interior of the support of the density $\psi_t$, whereas in Theorem \ref{thrm_t_fixed} the map $t\mapsto x^*_t$ gives information on the origin of the point in terms of $\psi$ and $t$. Theorem \ref{thrm_t_fixed} also covers cases where 
$\psi(x^*)=0$ but $\psi_t(x^*_t)>0$. 
For instance, our result applies also when $\psi$ vanishes at an interior point of its support with exponent $\kappa\leq 1$, or with exponent $\kappa>1$ if in addition $t>t_{cr}$.
\end{rmk}

Theorem \ref{thrm_t_fixed} shows that for Dyson's Brownian motion the time to local universality in the bulk is at most $\O(1)$. 
However, Dyson envisioned in \cite{Dyson} that the universal correlations (the ``local thermodynamic equilibrium'') should be 
reached already on time scales slightly bigger than $\O(1/n)$ for a large class of families of measures $\mu_n$. This issue has been addressed in more recent years by Erd\H{o}s, Yau and many collaborators, and 
led to their celebrated proofs of universality of $\beta$- and Wigner ensembles in different symmetry classes (see \cite{bookErdosYau, LandonSosoeYau,Erdosetal10, LandonYau} and references therein). 
The main condition of Theorem \ref{thrm_t_fixed}, namely weak convergence of $\mu_n$ to $\mu$, is however not sufficient to have local universality on small time-scales, and the time to reach sine kernel universality depends in a more subtle way on the distribution of the points $a_k^{(n)}$. To make this precise, we define quantiles of $\mu$ as follows: we let $q_k^{(n)}\in\R$ be such that
\begin{equation}\label{def:quantiles}
\int_{-\infty}^{q_k^{(n)}}d\mu(s)=\frac{k-\frac{1}{2}}{n}\quad \mbox{for $k=1,\ldots, n$}.
\end{equation}
If $\mu$ is supported on a single interval, these numbers are uniquely defined since $\mu$ is absolutely continuous. If the support of $\mu$ is disconnected, it could happen that some of the values $q_k^{(n)}$ are not uniquely defined. In such cases, we have the freedom to take $q_k^{(n)}$ to be any value satisfying \eqref{def:quantiles}.

\vspace{8pt}

We now define the positive sequence $m_n/n$ as the maximal deviation, over $k=1,\ldots, n$, of a point $a_k^{(n)}$ from its corresponding quantile $q_k^{(n)}$,
in other words
\begin{equation}\label{assumption:rigidity}
m_n:= n \ \max_{1\leq k\leq n}\left|a_k^{(n)}-q_{k}^{(n)}\right|.
\end{equation}	
It should be noted that $m_n$ in general depends on the particular choice of those quantiles that are not uniquely defined by \eqref{def:quantiles} (in the case that the support is not connected). The sequence $m_n$ can be interpreted as a measure for the global rigidity of the eigenvalues $a_k^{(n)}$ with respect to the measure $\mu$. 
The simplest example to keep in mind, is the case
where $a_k^{(n)}=q_{k}^{(n)}$ for $1\leq k\leq n$, such that $m_n=0$.
In general, weak convergence of $\mu_n$ to $\mu$ does not imply a bound on the sequence $m_n$, and it can happen that $m_n$ grows as $n\to\infty$.

\begin{rmk}
It is straightforward to verify that \eqref{assumption:rigidity} implies for the Kolmogorov distance between $\mu_n$ and $\mu$ that there exists a constant $c>0$, depending on $\mu$ but not on $n$, such that
\begin{equation}\label{eq:FnF}
\frac{\tilde m_n}{n}:=\sup_{x\in\R}|F_n(x)-F(x)|\leq \frac{c(m_n+1)}{n},
\end{equation}
where $F_n$ and $F$ are the distribution functions of $\mu_n$ and $\mu$, respectively. Also, for any interval $I$, we have
\begin{equation}\label{eq:comparemunmu}
\left|\mu_n(I)- \mu(I)\right|\leq \frac{2\tilde m_n}{n}\leq \frac{2c(m_n+1)}{n}.
\end{equation}
These two facts will be crucial in the proofs of our results.
\end{rmk}

In our next result, we show that sine kernel universality is obtained near $x^*$ for times $t_n$ which decay slower than $(m_n+1)/n$ and slower than $(\log n)^{1+\rho}/n$ for some $\rho>0$, if $x^*$ is an interior point of the support of $\mu$ where its density is positive. Although this is not surprising in view of recent results in e.g.\ \cite{Erdosetal10, LandonSosoeYau} (see also Remark \ref{rmkErdosYau} below), we believe that it is of interest to state and prove this result under explicit conditions on the distribution $\mu_n$ via the sequence $m_n$.
\begin{thm}\label{thm_nv}
Let $\mu_n$ and $\mu$ satisfy Assumptions 1 and 2, and let $x^*$ belong to the interior of the support of $\mu$ and be such that 
$\psi(x^*)>0$. Let $t_n$ satisfy $t_n\to 0$, $\frac{nt_n}{(\log n)^{1+\rho}}\to\infty$, and $\frac{nt_n}{m_n+1}\to\infty$ as $n\to\infty$, for some $\rho>0$. Uniformly for $u,v$ in 
compacts of $\mathbb R$, we have
\begin{align*}
\lim_{n\to\infty}\frac{1}{n\psi_{t_n}(x^*_{t_n})}K_{n,t_n}\lb x^*_{t_n}+\frac u{n\psi_{t_n}(x^*_{t_n})},x^*_{t_n}+\frac 
v{n\psi_{t_n}(x^*_{t_n})}\rb=\frac{\sin(\pi(u-v))}{\pi(u-v)},
\end{align*}
with $x_{t_n}^*$ given by \eqref{H_map}.
\end{thm}
\begin{rmk}\label{rmkErdosYau}Results of a similar nature have been obtained in the study of bulk universality for Wigner random matrices. 
In \cite{LandonSosoeYau}, a version of this theorem using vague convergence is proved under general but more implicit assumptions on the initial eigenvalues. In that paper, the 
authors assume the so-called local law down to a scale $o(t_n)$. A similar result has also been obtained in \cite[Proposition 3.3]{Erdosetal10}, using similar saddle point methods as ours, but stated under more implicit conditions involving the Stieltjes transforms of $\mu_n$ and $\mu$. Note that Theorem \ref{thm_nv} slightly improves the lower bound on the time to universality from $\O(n^{-1+\epsilon})$ for some $\epsilon>0$ in \cite{LandonSosoeYau,Erdosetal10} to $\O(n^{-1}(\log n)^{1+\rho})$ for some $\rho>0$.
\end{rmk}
\begin{rmk}
Although we defined, for technical reasons, $m_n$ in \eqref{assumption:rigidity} as the global maximal deviation from the quantiles over all eigenvalues $a_k^{(n)}$, we believe that especially the deviations $\left|a_k^{(n)}-q_k^{(n)}\right|$ for those quantiles lying close to $x^*$ are important. 
\end{rmk}

The intuition behind the interplay between the behavior of $t_n$ and that of $m_n$ as $n\to\infty$, is that there might be gaps of mesoscopic size bigger than $\mathcal O(n^{-1})$ between eigenvalues $a_k^{(n)}$ if $m_n$ tends to infinity, and one cannot expect convergence to the sine kernel as long as such mesoscopic gaps are present. The larger such a gap, the longer it will take before it is removed by the process.
These heuristics will be confirmed and made 
precise in Theorem \ref{thrm_counterexample} below.

We now focus on the situation of a point $x^*$ in the interior of the support of $\psi$ with $\psi(x^*)=0$. 
We assume that $\psi$ vanishes at $x^*$ with exponent $\kappa<1$. Our next theorem shows that for initial configurations that are sufficiently close to quantiles, a time slightly larger than 
$\left(\frac{m_n+1}{n}\right)^{\frac{1-\k}{1+\k}}$ is enough to reach bulk universality, see Figure \ref{figure2} for an illustration.

\begin{thm}\label{thrm_v}
Let $\mu_n$ and $\mu$ satisfy Assumptions 1 and 2. Let $x^*\in\R$ be a point in the interior of the support of $\mu$ such that 
\eqref{assumption_decay} holds with $0<\k<1$. Let $t_n$ satisfy $t_n\to0$, $\frac{n t_n^{\frac{1+\k}{1-\k}}}{(\log n)^{1+\rho}} \to\infty$, and $\frac{n t_n^{\frac{1+\k}{1-\k}}}{m_n+1} \to\infty$ as $n\to\infty$, for some \(\rho>0\). 
Then locally uniformly in $u,v$ 
\begin{align*}
\lim_{n\to\infty}\frac{1}{n\psi_{t_n}(x^*_{t_n})}K_{n,t_n}\lb x^*_{t_n}+\frac u{n\psi_{t_n}(x^*_{t_n})},x^*_{t_n}+\frac 
v{n\psi_{t_n}(x^*_{t_n})}\rb=\frac{\sin(\pi(u-v))}{\pi(u-v)},
\end{align*}
with $x_{t_n}^*$ as in \eqref{H_map}.
\end{thm}
\begin{figure}[h]
\centering
      \begin{minipage}[t]{0.45\linewidth}
			\centering
			 \includegraphics[trim=5cm  5cm 5cm 9cm,clip,width=\linewidth]{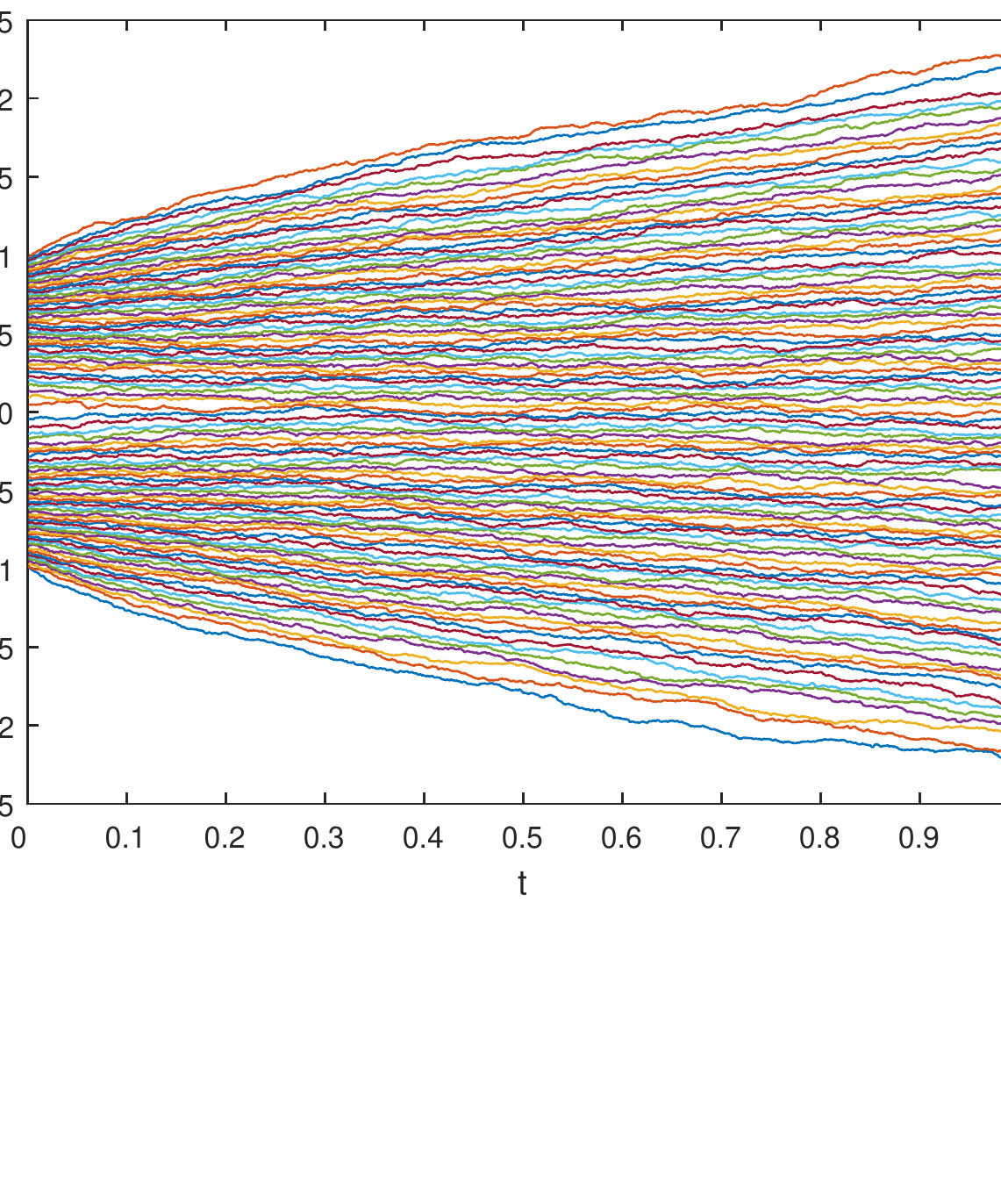}
			\end{minipage}%
			\hfill
			\begin{minipage}[t]{0.45\linewidth}
			\centering
\includegraphics[trim=5cm  5cm 5cm 9cm,clip,width=\linewidth]{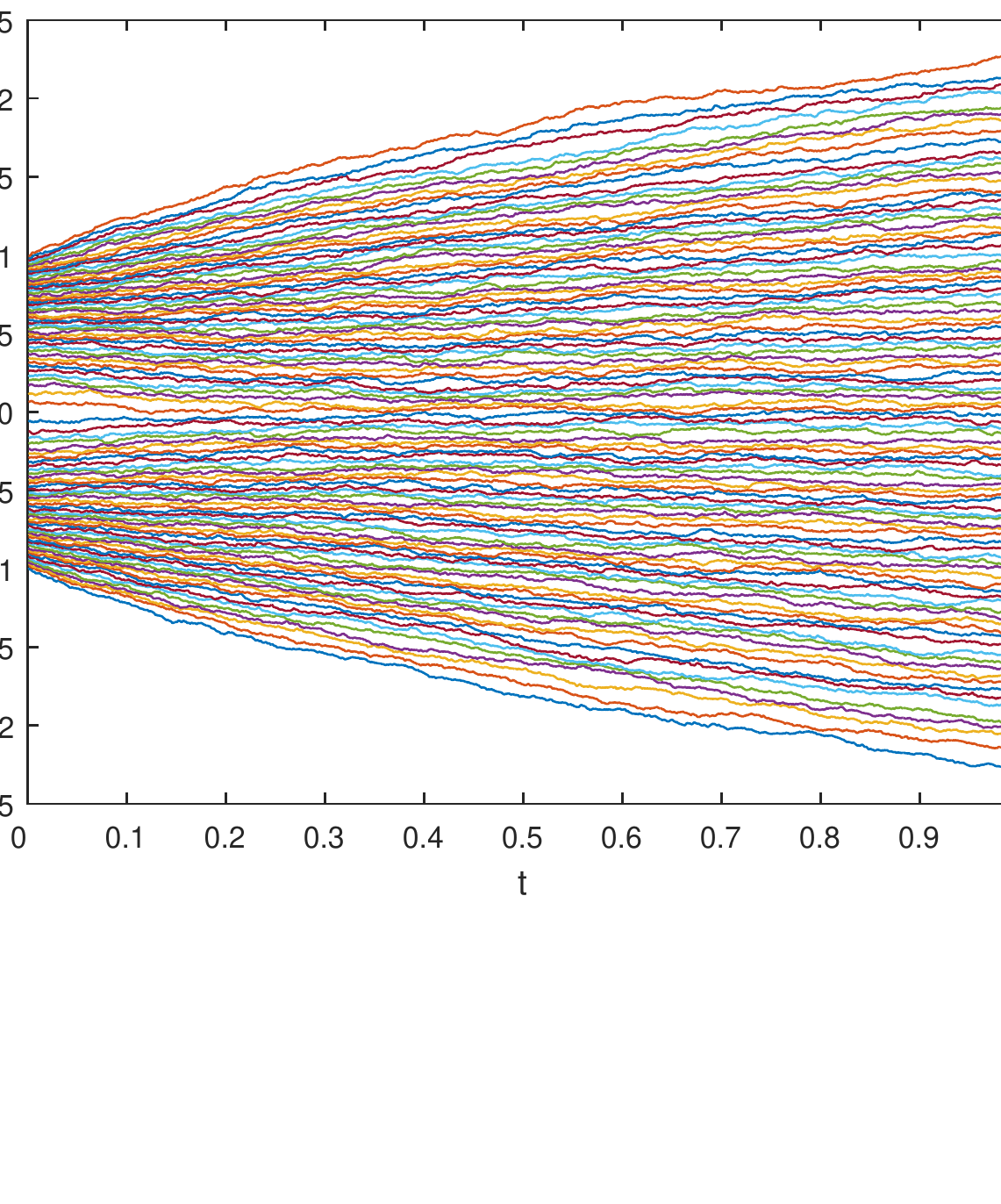} 
			\end{minipage} \vspace{-5em}\caption{Two samples of Dyson's Brownian motion for $t\in[0,1]$ and $n=100$, with $M_n$ having eigenvalues which are quantiles of the measures $d\mu(x)= \frac{3}{4}|x|^{1/2}dx$ (left) and $d\mu(x)= \frac{5}{6}|x|^{2/3}dx$ (right) on $[-1,1]$. For $t$ small, it is clearly visible that the paths near $0$ behave differently.}
			\label{figure2}
\end{figure} 
\begin{rmk}\label{rmkkappabig}
If $x^*$ is a point where the density of $\mu$ vanishes with exponent $\kappa>1$, we already mentioned that the density of the free additive convolution $\mu_t=\mu\boxplus\sigma_t$ has a zero for $t$ smaller than $t_{cr}>0$, see \cite{Biane} and \cite{ClaeysKuijlaarsLiechtyWang} for a detailed description if $\kappa$ is an even integer. Then, one cannot expect sine kernel universality to hold for $t\leq t_{cr}$, which means that the time to reach local universality is drastically bigger than in the case $\kappa<1$.
In the thresholding case $\k=1$ in 
\eqref{assumption_decay}, we expect that the time to reach bulk universality will tend to $0$ at a slow rate as $n\to\infty$, but we do not deal with this particular case in this paper.
\end{rmk}

\begin{rmk}
We emphasize that if we replace the conditions \(\frac{n t_n }{m_n+1}\to\infty\) in Theorem \ref{thm_nv} and \(\frac{n t_n^{\frac{1+\kappa}{1-\kappa}} }{m_n+1}\to\infty\) in Theorem \ref{thrm_v} by the slightly weaker conditions \(\frac{n t_n }{\tilde{m}_n}\to\infty\) and \(\frac{n t_n^{\frac{1+\kappa}{1-\kappa}} }{\tilde{m}_n}\to\infty\), respectively, then the statements of the theorems remain valid (which follows from our proofs). However, we decided to state the theorems in terms of more natural explicit conditions on the initial configurations.
\end{rmk}

In our last result, we show that the behavior of the sequence $m_n$ is crucial in Theorem \ref{thm_nv} and Theorem \ref{thrm_v}. To that end, we will consider initial configurations $\mu_n$ which have a gap of  size $\delta_n$ near a point $x^*$, with $n\delta_n^2\to\infty$. We then show that this gap propagates along the path \eqref{H_map} for times $t_n$ which are $o(\delta_n^2)$. This is illustrated in Figure \ref{figure3} below.

\begin{thm}\label{thrm_counterexample}
Let $\mu_n=\frac{1}{n}\sum_{k=1}^n a_k^{(n)}$ for some points $a_k^{(n)}$, $k=1,\ldots, n$, $n\in\mathbb N$.
Let $\delta_n$ be a sequence converging to $0$ in such a way that $n\delta_n^2\to\infty$, as $n\to\infty$. Suppose that $x^*\in\R$ is such that the intervals $[x^*-\delta_n, x^*+\delta_n]$ do not contain any of the starting points 
$a_1^{(n)},\ldots, a_n^{(n)}$.
If $\epsilon_n=o(\delta_n)$ and if $t_n=o(\delta_n^2)$, as $n\to\infty$, we have locally uniformly in $u,v$,
\[\lim_{n\rightarrow \infty}\epsilon_n K_{n,t_n}\left(x^*_{t_n}+\epsilon_n u, x^*_{t_n}+\epsilon_n v\right)=0.\]
\end{thm}
\begin{figure}[h]
\centering
      \begin{minipage}[t]{0.45\linewidth}
			\centering
			 \includegraphics[trim=5cm  5cm 5cm 9cm,clip,width=\linewidth]{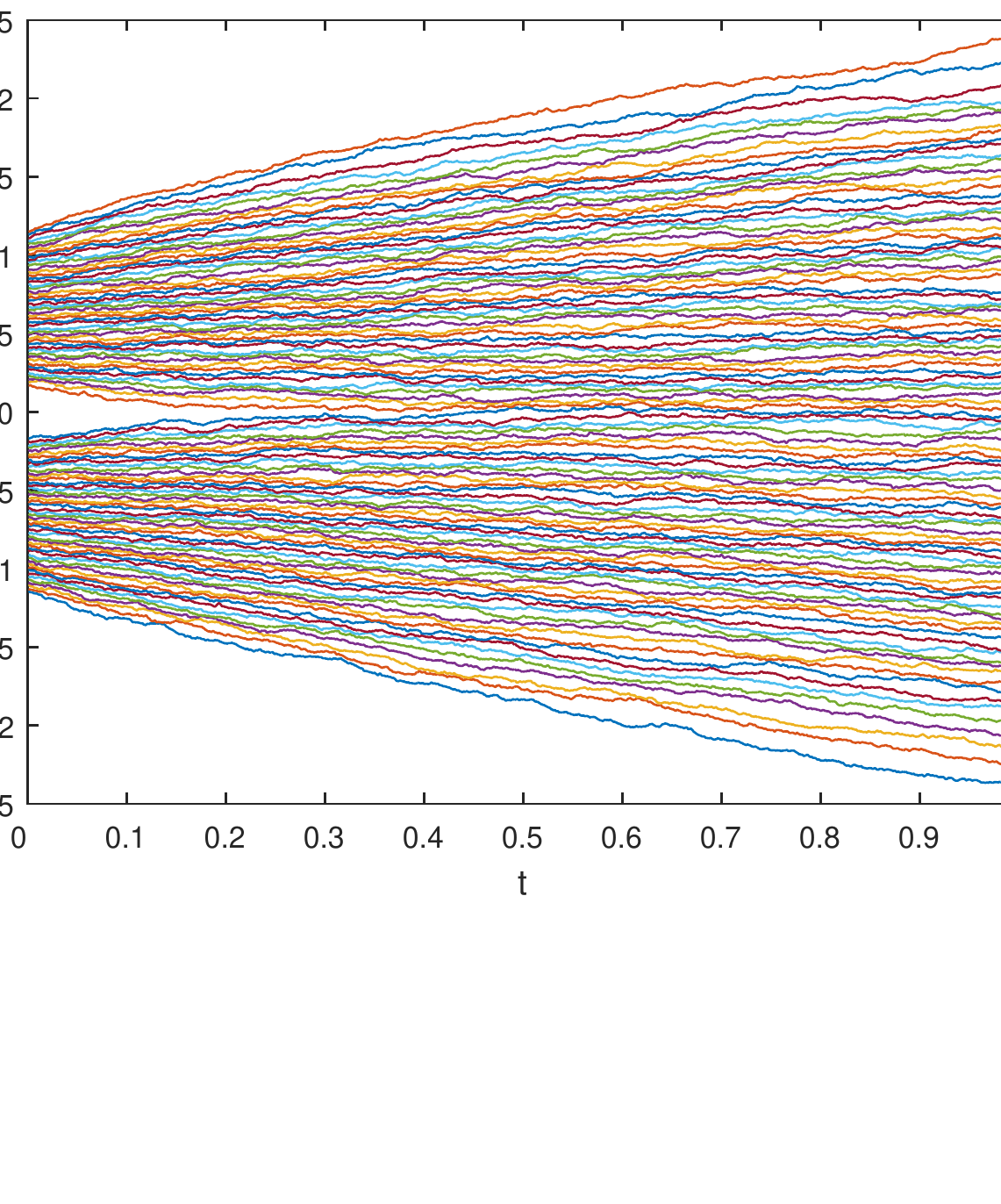}
			\end{minipage}%
			\hfill
			\begin{minipage}[t]{0.45\linewidth}
			\centering
\includegraphics[trim=5cm  5cm 5cm 9cm,clip,width=\linewidth]{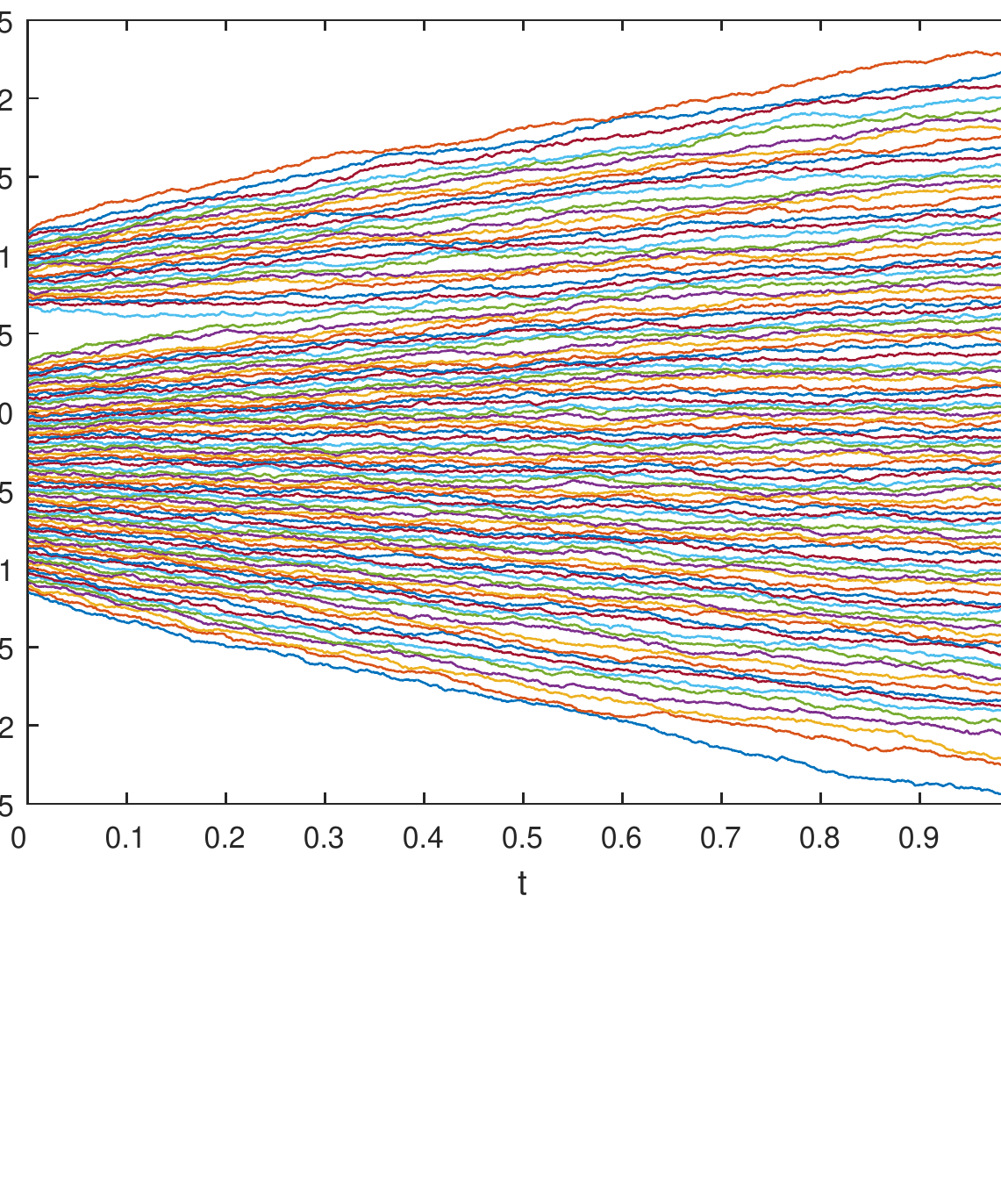} 
			\end{minipage} \vspace{-5em}
\caption{Two samples of Dyson's Brownian motion for $t\in[0,1]$, with $n=100$. The initial configurations are obtained by inserting a gap of size $0.34$ near $0$ (left) and near $1/2$ (right) into equi-spaced points on $[-1,1]$. The gap propagates for $t$ small.}
    \label{figure3}
\end{figure} 
As a consequence of the last theorem we present a corollary on gap probabilities.
\begin{cor}\label{cor}Let $x^*,\delta_n, \epsilon_n, t_n$ be as in Theorem \ref{thrm_counterexample}. Then
\begin{align*}
\lim_{n\to\infty}\textup{Prob}\lb Y_n(t_n) \text{ has no eigenvalues in} \left[x_{t_n}^*- \epsilon_n, x_{t_n}^*+ \epsilon_n\right]\rb=1.
\end{align*}
\end{cor}
\begin{proof}
It follows from the general theory of determinantal point processes that 
\[\textup{Prob}\lb Y_n(t_n) \text{ has no eigenvalues in } A\rb=\det(1-\left.K_{n,t_n}\right|_{A}),\]
where $\det(1-\left.K_{n,t_n}\right|_{A})$ is the Fredholm determinant associated to the integral operator with kernel $K_{n,t_n}$ acting on $L^2(A)$. For $A=\left[x_{t_n}^*- \epsilon_n, x_{t_n}^*+\epsilon_n\right]$, it follows from Theorem \ref{thrm_counterexample} that the operator $\left.K_{n,t_n}\right|_{A}$ converges in trace norm to $0$ (cf.~\cite{Simon}), hence the Fredholm determinant converges to $1$ as \(n\to\infty\).
\end{proof}
Theorem \ref{thrm_counterexample} and Corollary \ref{cor} apply in particular to cases in which the limiting density $\psi$ vanishes at $x^*$ with exponent $\kappa>1$. Then the typical distance between consecutive eigenvalues $a_k^{(n)}$ close to $x^*$ is $\mathcal O(n^{-\frac{1}{\kappa+1}})$, and the above results imply that each  gap of this size propagates for times $t_n=o(n^{-\frac{2}{\kappa+1}})$. In such cases, as explained in Remark \ref{rmkkappabig}, it can be expected that sine kernel universality is only reached for times beyond the critical time $t_{cr}$. Near the critical time, non-trivial limiting kernels are expected. We intend to come back to this in a future publication.

The proofs of our results are based on an asymptotic analysis of the double contour integral representation \eqref{eq:Kncontour3} for the kernel $K_{n,t}$.
Since the function $\frac{n}{2t}\left((z-x)^2+2tg_{\mu_n}(z)\right)$ appearing in the exponential of the double integral depends on $\mu_n$ and hence on $n$ in a rather complicated manner, we will need to carefully study the $n$-dependence of the saddle points and several associated quantities in order to be able to rigorously perform a saddle point analysis. On the other hand, to prove convergence to the sine kernel, we will not need to know the precise leading order behavior of the double integral, and an upper bound will be sufficient. This is a consequence of the convenient form of \eqref{eq:Kncontour3}, in which the sine kernel is explicitly present in the second term.

The paper is organized as follows. Section \ref{section_t_fixed} contains the proof of Theorem 
\ref{thrm_t_fixed}. Theorems \ref{thm_nv}, \ref{thrm_v} and \ref{thrm_counterexample} are proven in Sections \ref{section_thrm_nv}, \ref{section_thrm_v} and \ref{section_thrm_counterexample}, 
respectively. We conclude with Appendix \ref{appendix} giving a new self-contained proof of the determinantal relations \eqref{determinantal_K_{n,t}} for the kernel $K_{n,t}$ given by the double contour integral formula \eqref{eq:Kncontour3}.

\section{Proof of Theorem \ref{thrm_t_fixed}}\label{section_t_fixed}
We start by recalling some analytic facts about the free convolution $\mu\boxplus\sigma_t$.
Let us define the Stieltjes transform of $\mu$ by
\begin{align*}
 G_\mu(z):=\int\frac{d\mu(x)}{z-x}
\end{align*}
for $z$ in the upper half plane $\mathbb C^+:=\{z\in\C:\Im z>0\}$. Recall the definition of $y_{t,\mu}$ in \eqref{def:y_t} for $t>0$. Note that
\[\int \frac{d\mu(s)}{(x-s)^2 + y^2} = -\frac{1}{y}\Im G_\mu(x+iy). \]
The function $x \mapsto y_{t,\mu}(x)$ is continuous and we write
\begin{equation} \label{def:Omega} 
  \Omega_{t,\mu} = \{ x + iy \in \mathbb C^+ \mid y > y_{t,\mu}(x) \} 
\end{equation}
for the domain above the graph of $y_{t,\mu}$.
Biane \cite{Biane} showed that the function
\begin{equation}
H_{t,\mu}(z)=z+t G_\mu(z)\label{def:H}
\end{equation}
maps the region $\Omega_{t,\mu}$ conformally to the upper half plane.
The graph of $y_{t,\mu}$ (the boundary of $\Omega_{t,\mu}$) is mapped bijectively to the real line. For $x$ real, $G_\mu(x)$ is 
understood as the limit of $G_\mu(z)$ as $z$ approaches $x$ from the upper half plane. 
We write $F_{t,\mu}$ for the inverse function of $H_{t,\mu}$, mapping the real line back to the graph of $y_{t,\mu}$.

The Stieltjes transform and the density of the free convolution $\mu_t$ may be recovered from the formulas
\begin{equation} \label{eq:Gtau} 
  G_{\mu_{t}}(z) = G_{\mu}(F_{t,\mu}(z)) \quad \text{and} \quad  \psi_{t}(x) = - \frac{1}{\pi} \Im G_{\mu}(F_{t,\mu}(x)), \quad 
\text{for $x \in \mathbb R$}. 
\end{equation}
Given a reference point $x^*\in\mathbb R$, we define the time evolution
\begin{equation} \label{def:xtau}
x_t^{*}=H_{t,\mu}(x^{*}+iy_{t,\mu}(x^{*})),
\end{equation}
for points \(x^{*}\) on the real line. We have 
\begin{equation}\label{eq:psitau2}
\psi_{t}(x_{t}^{*})=\frac{y_{t,\mu}(x^{*})}{\pi t}.
\end{equation}
It is easy to see that the definition in \eqref{def:xtau} coincides with the one given in \eqref{H_map}. Moreover, using 
\eqref{eq:psitau} it is not difficult to see that with \eqref{t_cr}  we have $\psi_t(x^*_t)=0$ for any $t\leq t_{cr}$ and  
$\psi_{t}(x^*_t)>0$ for $t>t_{cr}$.

\begin{proof}[Proof of Theorem \ref{thrm_t_fixed}]
With the double contour integral formula \eqref{eq:Kncontour3} we can write the kernel as 
\[\frac{1}{c_t n}K_{n,t}\left({x_t^*}+\frac{u}{c_t n}, {x_t^*}+\frac{v}{c_t n}\right)=I_n(u,v)+A_n(u,v),\]
with \(c_t := \psi_t(x_t^{\ast})\) and
\begin{align}\label{I_n}&I_n(u,v):=\frac{e^{-\frac{1}{2nt c_t^2}(u^2-v^2)+\frac{1}{c_tt}(u-v)(x_0-x_t^*)}}{c_t(2\pi i)^2t} \int\limits_{x_0-i\infty}^{x_0+i\infty}dz \int\limits_{\gamma} dw 
~\frac{e^{-f_n(z,w)}}{z-w},\\
&A_n(u,v)=\frac{1}{\pi (u-v)} \sin\left(\frac{s}{c_t t} (u-v)\right).\label{A_n}
\end{align}
Here we define
\[f_n(z,w)=-\phi_n(z,u)+\phi_n(w,v),\]
\[\phi_n(z,u)=\frac{n}{2t}\left[\left(z-{x_t^*}-\frac{u}{c_t n}\right)^2+2t g_{\mu_n}(z)\right]\]
with $g_{\mu_n}$ as in \eqref{def:g}.
As in formula \eqref{eq:Kncontour3}, $s>0$ is the imaginary part of the point \(\t_1\) in the upper half plane which is the intersection of 
\(\gamma\) with the horizontal line \(x_0+i s\), \(s\in\mathbb{R}\). In order to prove Theorem \ref{thrm_t_fixed}, we will show for a suitable choice of the contour $\gamma$ and the point $x_0$ that $A_n(u,v)$ converges to the sine kernel as $n\to\infty$, whereas $I_n(u,v)$ converges to $0$. We will in particular choose $\gamma$ and $x_0$ in such a way that
\begin{equation}\label{eq:estimate s}s=\pi c_t t +o(1),~ n\rightarrow \infty,\end{equation}
with the $o(1)$ error term uniformly small for $u,v$ in any compact set. Then it is 
readily seen that we have
\begin{equation}\label{eq:limAn}\lim_{n\rightarrow \infty} A_{n}(u,v)=\frac{\sin\left(\pi (u-v)\right)}{\pi (u-v)},\end{equation}
uniformly for $u,v$ in any compact set, if we understand the right hand side as $1$ if $u=v$.
The next crucial step is to see that
\begin{equation}\label{eq:limIn}\lim_{n\rightarrow \infty}I_{n}(u,v)=0\end{equation}
uniformly in $u,v$, and this will be done by a two-dimensional saddle point analysis, in which delicate estimates are needed, especially because the phase function $f_n(z,w)$ can depend on $n$ in a complicated manner.

Computing the (two-dimensional) complex saddle points of the function \(f_n\) gives the equations
\[z-{x_t^*}-\frac{u}{c_t n}+t G_{\mu_n}(z)=0,\]
\[w-{x_t^*}-\frac{v}{c_t n}+t G_{\mu_n}(w)=0.\]
Expressing this in terms of the functions \(H_{t, \mu_n}(z)=z+t G_{\mu_n}(z)\) gives
\[H_{t, \mu_n}(z)={x_t^*}+\frac{u}{c_t n},\]
\[H_{t, \mu_n}(w)={x_t^*}+\frac{v}{c_t n}.\]
Now, as can be seen with help of the definition of \(y_{t, \mu}\), the function \(H_{t, \mu}(z)\) is conformal in a neighborhood of 
\(x^{*}+iy_{t, \mu}(x^{*}) = F_{t, \mu}({x_t^*})\), so the inverse \(F_{t, \mu}\) is conformal in a neighborhood of \({x_t^*}\). 
Moreover, it follows from the weak convergence of \(\mu_n\) to \(\mu\) (and eventually using Vitali's convergence theorem) that 
\(H_{t, \mu_n}\) converges uniformly to \(H_{t, \mu}\) in an $n$-dependent neighborhood \(U\) of the point \(x^{*}+iy_{t, \mu}(x^{*})\) as we stay 
at a positive distance from the supports of the measures. So there exists an index \(N\) such that for all \(n>N\) the functions 
\(H_{t, \mu_n}\) are conformal mappings from \(U\) onto a neighborhood of \({x_t^*}\). Hence, for large \(n\), we can consider the 
inverse functions \(F_{t, \mu_n}\) of \(H_{t, \mu_n}\) in a neighborhood of \({x_t^*}\) that is independent of $n$, so that we obtain a specific pair of 
solutions for the saddle point equations by
\[z_n = x_n+iy_n=F_{t, \mu_n}\left({x_t^*}+\frac{u}{c_t n}\right),\]
\[w_n= F_{t, \mu_n}\left({x_t^*}+\frac{v}{c_t n}\right)\]
and the corresponding complex conjugate solutions. In total we find four two-dimensional saddle points \((z_n,w_n), 
(\overline{z_n},w_n), (z_n,\overline{w_n})\) and \((\overline{z_n},\overline{w_n})\). In \eqref{I_n} we now choose specific 
contours of integration. The integral in the \(z\)-plane is taken along the vertical line through \(x_0=x_n= \Re F_{t, 
\mu_n}\left({x_t^*}+\frac{u}{c_t n}\right)\). The parts of this line in the lower and in the upper part of the complex plane give 
paths of descent for the phase function \(\phi_n(z,u)\). For instance, for \(\tau>0\) we have
\[\frac{d}{d\t}\Re \phi_n(x_n+i\t,u)=\frac{n\tau}{t}\left[t \int\frac{d\mu_n(a)}{(x_n-a)^2+\t^2}-1\right],\]
where 
\(\int\frac{d\mu_n(a)}{(x_n-a)^2+\t^2}\)
is strictly decreasing in \(\t\), and the right hand side of the above equation vanishes if and only if \(\t=y_n=\Im F_{t, \mu_n}\left({x_t^*}+\frac{u}{c_t n}\right)\). 
Moreover, we have 
\begin{align}\label{eq:derivative}
\beta_n:=-\frac{d^2}{d\t^2}\Big\vert_{\t=y_n}\Re \phi_n(x_n+i\t,u)=2ny_n^2 \int\frac{d\mu_n(a)}{\left((x_n-a)^2+y_n^2\right)^2} >0,
\end{align}
so that \(z_n\) is a simple saddle point for \(\phi_n(z,u)\). It is clear by the same argumentation that \(w_n\) is a simple saddle 
point for \(\phi_n(w,v)\). For the integral in the \(w\)-plane we construct a path \(\gamma=\gamma_n\) using the graph of the 
function \(y_{t,\mu_n}\). Since the supports of the measures \(\mu_n\) are all contained in a compact set independent of $n$, we can find a finite interval independent of $n$ (but depending on $t$), say \(J\), containing all the supports of the functions 
\(y_{t,\mu_n}\). Indeed, this follows easily from the fact that $y_{t,\mu_n}(x)=0$ for ${\rm dist}(x,\text{supp}(\mu_n))\geq \sqrt{t}$, which is a direct consequence of the definition \eqref{def:y_t}. The path \(\gamma_n\) starts at a real point to the right of \(J\) and follows the 
graph of \(y_{t,\mu_n}\), passes the saddle points \(z_n\) and \(w_n\) on the way, and finally returns for a last time back to the real axis at a point located to the left of \(J\). We complete the path \(\gamma_n\) just by going back using the complex conjugate path in the lower 
half plane. This establishes a path of descent for the phase function \(-\phi_n(w,v)\) of the integral in the \(w\)-plane passing 
through the saddle points \(w_n, \overline{w_n}\). We can verify this, for instance, by parametrizing the upper part of $\gamma_n$ by \(\tilde{\gamma}_n(\t)=\t+iy_{t,\mu_n}(\t)\) (using the opposite orientation for convenience) and computing
\[\frac{d}{d\t} \Re \phi_n(\tilde{\gamma}_n(\t),v)=\frac{n}{t}\left(H_{t,\mu_n}(\tilde{\gamma}_n(\t))-{x_t^*}-\frac{v}{c_t n}\right).\]
With these choices we get contours passing though each of the four saddle points along paths of descent. Now we will show that 
\begin{align*}
&\tilde{I}_n(u,v) =c_t (2\pi i)^2 t e^{\frac{1}{2nt c_t^2}(u^2-v^2)-\frac{1}{c_tt}(u-v)(x_0-x_t^*)}I_n(u,v)\\
&=\int\limits_{x_n-i\infty}^{x_n+i\infty}dz \int\limits_{\gamma_n} dw ~ 
\frac{e^{-f_n(z,w)}}{z-w}=\int\limits_{x_n-i\infty}^{x_n+i\infty}dz \int\limits_{\gamma_n} dw ~e^{\phi_n(z,u)-\phi_n(w,v)} 
\frac{1}{z-w}
\end{align*}
converges to zero as \(n\rightarrow \infty\). We remark in passing that we excluded the ``gauge factor'' $\exp(-\frac{1}{2nt c_t^2}(u^2-v^2)+\frac{1}{c_tt}(u-v)(x_0-x^*_t))$ from $I_n$ for notational convenience only. To show the vanishing of $\tilde{I}_n(u,v)$, we split the integration contours in several parts. As can be expected, the main contribution to the double integral will come from small neighbourhoods of the saddle points. For technical reasons, we define \begin{equation}\label{def:Ln}L_n =\frac{\log n}{\sqrt{n}},\end{equation}and split the integral into seven parts 
\(\tilde{I}_n(u,v)=\sum_{k=1}^7 \tilde{I}_n^{(k)}\) in the following way:
\[\tilde{I}_n^{(1)}=\int\limits_{z_n+iL_n}^{z_n+i\infty}dz \int\limits_{\gamma_n} dw ~e^{\phi_n(z,u)-\phi_n(w,v)} \frac{1}{z-w},\]
\[\tilde{I}_n^{(2)}=\int\limits_{\overline{z_n}-i\infty}^{\overline{z_n}-iL_n}dz \int\limits_{\gamma_n} dw 
~e^{\phi_n(z,u)-\phi_n(w,v)} \frac{1}{z-w},\]
\[\tilde{I}_n^{(3)}=\int\limits_{\overline{z_n}+iL_n}^{z_n-iL_n}dz \int\limits_{\gamma_n} dw ~e^{\phi_n(z,u)-\phi_n(w,v)} 
\frac{1}{z-w},\]
\[\tilde{I}_n^{(4)}=\int\limits_{z_n-iL_n}^{z_n+iL_n}dz \int\limits_{\gamma_n\cap U_{L_n}(w_n)} dw ~e^{\phi_n(z,u)-\phi_n(w,v)} 
\frac{1}{z-w},\]
where \(U_{L_n}(w_n)\) denotes a disk of radius \(L_n\) centered at \(w_n\),
\[\tilde{I}_n^{(5)}=\int\limits_{z_n-iL_n}^{z_n+iL_n}dz \int\limits_{\gamma_n\backslash U_{L_n}(w_n)} dw 
~e^{\phi_n(z,u)-\phi_n(w,v)} \frac{1}{z-w},\]
\[\tilde{I}_n^{(6)}=\int\limits_{\overline{z_n}-iL_n}^{\overline{z_n}+iL_n}dz \int\limits_{\gamma_n\cap U_{L_n}(\overline{w_n})} dw 
~e^{\phi_n(z,u)-\phi_n(w,v)} \frac{1}{z-w},\]
and
\[\tilde{I}_n^{(7)}=\int\limits_{\overline{z_n}-iL_n}^{\overline{z_n}+iL_n}dz \int\limits_{\gamma_n\backslash 
U_{L_n}(\overline{w_n})} dw ~e^{\phi_n(z,u)-\phi_n(w,v)} \frac{1}{z-w}.\]
To estimate the above integrals, we will need a rough bound for the length of the curve $\gamma_n$, which is obtained in the following lemma. We will prove a more general version as it will be needed later on, however here we only apply it for the case of a fixed $t$.
\begin{lemma}\label{lemma:length}
 Let $L(\g_n)$ denote the length of $\g_n$. Then for any positive bounded sequence $t_n$ we have
\begin{align*}
 L(\g_n)=\O\lb n^3\sqrt{t_n}\rb,\qquad n\to\infty.
\end{align*}
\end{lemma}
\begin{proof}
We assume that the intervals of positivity of the function \(y_{t_n,\mu_n}\) are contained inside the bounded interval \(J\). In order to estimate the length of \(\gamma_n\), which is constructed from the the graph of \(y_{t_n,\mu_n}\) as described above, it is necessary to estimate the length on these intervals of positivity which form the support of \(y_{t_n,\mu_n}\). We aim to find an upper bound for the number of intervals of monotonicity into which the support can be partitioned. In order to do so we will find an estimate for the number of vanishing points of the derivative of  \(y_{t_n,\mu_n}\) inside its support. If \(I\) is an interval on which \(y_{t_n,\mu_n}\) is monotonically increasing or decreasing, we have that the length of its graph on \(I\) is bounded by \(L(I)\sqrt{t_n}\), where $L(I)$ denotes the length of $I$ and we use the fact that \(y_{t_n,\mu_n}(x)\leq \sqrt{t_n}\) for all \(x\in\mathbb{R}\). On the intervals of positivity we have the equality
\[\sum_{k=1}^n \frac{1}{(x-a_{n}^{(k)})^2+y_{t_n,\mu_n}^2(x)}=\frac{n}{t_n},\]
which we can write as 
\begin{equation*}
\prod_{j=1}^n \left(y_{t_n,\mu_n}^2(x) +(x-a_n^{(j)})^2\right)=\frac{t_n}{n}\sum_{k=1}^n \prod_{j\neq k} \left(y_{t_n,\mu_n}^2(x) +(x-a_n^{(j)})^2\right).
\end{equation*}
Hence, the function \(w=w_n=y_{t_n,\mu_n}^2\) satisfies the algebraic equation
\begin{equation}\label{algeq_w}
\prod_{j=1}^n \left(w(x) +(x-a_n^{(j)})^2\right)=\frac{t_n}{n}\sum_{k=1}^n \prod_{j\neq k} \left(w(x) +(x-a_n^{(j)})^2\right).
\end{equation}
As the derivatives of $y_{t_n,\mu_n}$ and \(w\) vanish at the same points inside the support, it is sufficient to find an upper bound for the points of vanishing of the derivative of the algebraic function \(w\) defined by \eqref{algeq_w} on its entire associated Riemann surface. To this end, we rewrite equation \eqref{algeq_w} in the form
\begin{equation}\label{def_Q_n}Q_n(x,w)=x^{2n}+\sum_{k=0}^{2n-1}p_{k,n}(w) x^k =0,
\end{equation}
where \(p_{k,n}(w)\) are polynomials in \(w\) of degree at most \(n\). Let us assume that the derivative of \(w\) vanishes at some point \(x_0\) of its Riemann surface. Then by differentiation we see that the pair \((x_0,w(x_0))\) satisfies the equations
\[Q_n(x_0,w(x_0))=0,\quad \frac{\partial Q_n}{\partial x} (x_0,w(x_0))=0.\]
Hence, at the point \(w=w(x_0)\) the (univariate) polynomial \(Q_n(x,w(x_0))\) has a multiple root at \(x=x_0\). This means that \(w(x_0)\) is a root of the corresponding discriminant which, in view of \eqref{def_Q_n}, is of degree at most \(2n(2n-1)\). It follows that there is a set \(\{w_1,\ldots,w_{2n(2n-1)}\}\) of at most \(2n(2n-1)\) different values which \(w(x_0)\) can take. But for every possible value \(w(x_0)=w_k\) there are at most \(2n\) different solutions of \(Q_n(x,w(x_0))=0\) considered as a polynomial equation in \(x\). From this we obtain a set  \(\{x_1,\ldots,x_{4n^2(2n-1)}\}\) of at most \(4n^2(2n-1)\) different values that \(x_0\) can take, which means that there are at most \(4n^2(2n-1)\) points \(x_0\) such that \(w'(x_0)=0\). Hence, on the intervals of positivity of the function \(y_{t_n,\mu_n}\) its derivative can vanish at at most \(4n^2(2n-1)\) different points. But from this we infer that throughout the support of \(y_{t_n,\mu_n}\), there are at most \(4n^2(2n-1)+1\) many intervals of monotonicity. This gives for the length of \(\gamma_n\) 
\[L(\gamma_n)=\mathcal{O}\left(n^3 \sqrt{t_n}\right),\]
as \(n\to \infty.\)
\end{proof}

In what follows, all estimates hold for sufficiently large values of 
\(n\), and we will use constants \(\eta,\tilde{\eta},\hat\eta>0\) which are independent of \(n\) and also independent of $u,v$ (for $u,v$ in any compact set), but which can change their values at 
different occasions without being mentioned explicitly. 
Also, constants implied by asymptotic $\mathcal O(\cdot)$ or $o(\cdot)$ notations as $n\to\infty$ can be chosen independent of $u,v$ for $u,v$ in any compact set. Moreover, although terms of the form \(\sqrt{t}\) could be absorbed by constants we will write them explicitly in favor of later references.

\paragraph{\bf Estimation of \(\tilde{I}_n^{(1)}\). }
First, we observe that for $(z,w)$ on \([z_n+iL_n, z_n+i\infty]\times 
\gamma_n\) we have
\begin{equation}\label{def:dn}d_n^{-1}:=\frac{1}{\min|z-w|}\leq \frac{\eta}{L_n}.\end{equation}
Moreover, on \(\gamma_n\) the function \(-\phi_n(w,v)\) takes its maximum in \(w=w_n\), so that by Lemma \ref{lemma:length} we obtain
\[|\tilde{I}_n^{(1)}|\leq \frac{\tilde\eta n^3 \sqrt{t}}{d_n} 
e^{\Re\{\phi_n(z_n,u)-\phi_n(w_n,v)\}}\int\limits_{z_n+iL_n}^{z_n+i\infty}|dz|e^{\Re\{\phi_n(z,u)-\phi_n(z_n,u)\}}. \]
A computation using the definition of the saddle points shows
\begin{equation}\label{eq:znwnestimate}z_n-w_n = \mathcal{O}\left(\frac{1}{n}\right),\end{equation}
and also
\begin{equation}\label{eq:phinznwnestimate}-f_n(z_n,w_n)=\phi_n(z_n,u)-\phi_n(w_n,v) =\mathcal{O}(1),\end{equation}
as \(n\rightarrow \infty\). This gives
\begin{align*}|\tilde{I}_n^{(1)}|&\leq\frac{\tilde\eta n^3\sqrt{t}}{d_n}\int\limits_{z_n+iL_n}^{z_n+i\infty}|dz|e^{\frac{1}{2}\Re\{\phi_n(z,
u)-\phi_n(z_n,u)\}}e^{\frac{1}{2}\Re\{\phi_n(z,u)-\phi_n(z_n,u)\}}\\
&\leq\frac{\tilde\eta n^3\sqrt{t}}{d_n}e^{\frac{1}{2}\Re\{\phi_n(z_n+iL_n,u)-\phi_n(z_n,u)\}}\int\limits_{z_n+iL_n}^{z_n+i\infty}|dz|e^{\frac{1}{2}
\Re\{\phi_n(z,u)-\phi_n(z_n,u)\}}.
\end{align*}
A complex Taylor expansion for \(\phi_n(z,u)\) around \(z_n\) yields
\[\phi_n(z,u)=\phi_n(z_n,u)+\frac{1}{2}\phi_n^{''}(z_n,u)(z-z_n)^2+R_2(z)\]
with
\begin{equation}\label{eq:estimateR2}|R_2(z)|\leq \frac{\max_{|s-z_n|=r}|\phi_n(s,u)|}{r^3}\frac{|z-z_n|^3}{1-\frac{|z-z_n|}{r}}\leq n\hat{\eta}|z-z_n|^3\leq 
n\hat{\eta}L_n^3,\end{equation}
for suitably small \(r>0\) and \(|z-z_n|\leq L_n\). By \eqref{eq:derivative} and by the fact that $nL_n^3\to 0$ as $n\to\infty$, we get
\[\frac{\tilde\eta n^3\sqrt{t} }{d_n}e^{\frac{1}{2}\Re\{\phi_n(z_n+iL_n,u)-\phi_n(z_n,u)\}}=\mathcal{O}\left(\frac{n^3\sqrt{t}}{d_n} e^{-\frac{\beta_n}{2} L_n^2}\right),\quad n\rightarrow \infty.\]
Finally, we have
\[\int\limits_{z_n+iL_n}^{z_n+i\infty}|dz|e^{\frac{1}{2}\Re\{\phi_n(z,u)-\phi_n(z_n,u)\}}\leq\int\limits_{z_n+iL_n}^{z_n+i\infty}
|dz|e^{\Re\left\{-\left(z_n-{x_t^*}-\frac{u}{c_t n}\right)^2-2t g_{\mu_n}(z_n) +\left(z-{x_t^*}-\frac{u}{c_t n}\right)^2+2t 
g_{\mu_n}(z)\right\}}, \] 
and the right-hand side converges as $n\to\infty$ to 
\[\int\limits_{F_{t,\mu}({x_t^*})}^{F_{t,\mu}({x_t^*})+i\infty}|dz|e^{\Re\left\{-\left(F_{t,\mu}({x_t^*})-{x_t^*}\right)^2-2t 
g_{\mu}(F_{t,\mu}({x_t^*})) +\left(z-{x_t^*}\right)^2+2t g_{\mu}(z)\right\}} < \infty.\]
Together this gives by \eqref{def:Ln}, \eqref{def:dn} and the fact that $\lim_{n\to\infty}\beta_n/n$ exists and is positive,
\begin{equation}\label{eq:In1final}\tilde{I}_n^{(1)}=
\mathcal{O}\left(\frac{n^{3}\sqrt{t}}{d_n} e^{-\frac{\beta_n}{2} L_n^2}\right)=\mathcal{O}\left(n^{7/2}\frac{\sqrt{t}}{\log n} e^{-\eta(\log n)^2}\right),\quad n\rightarrow \infty,\end{equation}
for some constant $\eta>0$.

The integrals \(\tilde{I}_n^{(2)}\) and \(\tilde{I}_n^{(3)}\) can be treated in a very similar fashion, where the estimation of 
\(\tilde{I}_n^{(3)}\) is slightly simpler as the contours of integration stay bounded.

\paragraph{\bf Estimation of \(\tilde{I}_n^{(4)}\). }
Recall that $H_{t,\mu_n}$ maps the part of $\gamma_n$ in the upper half plane bijectively and conformally to a part of the real line. This implies the identity
\begin{equation}\label{eq:identityHy}\arctan y_{t,\mu_n}'(\Re w)=-\arg H_{t,\mu_n}'(w),\qquad w\in\gamma_n\cap 
U_{L_n}(w_n).\end{equation}
As $n\to\infty$, $\arg H_{t,\mu_n}'(w)\to \arg H_{t,\mu}'(x^*)\in (-\pi/2, \pi/2)$. This implies that $y_{t,\mu_n}'(\Re w)$ remains bounded for large $n$, and that the length of 
the contour $\gamma_n\cap 
U_{L_n}(w_n)$ is $\mathcal O(L_n)$ as $n\to\infty$.
Hence, we can use an arc-length parametrization $\gamma_n:[-\ell_n,\tilde \ell_n]\to \gamma_n\cap 
U_{L_n}(w_n)$, with $\gamma_n(0)=z_n$ and $\ell_n,\tilde \ell_n\leq \eta L_n$, to compute the integral $\tilde I_n^{(4)}$.
This easily leads to the upper bound
\begin{equation}\label{eq:estI4}|\tilde{I}_n^{(4)}|\leq  e^{\max\Re\left(\phi_n(z,u)-\phi_n(w,v)\right)} \int\limits_{z_n-iL_n}^{z_n+iL_n}|dz| \int_{-\ell_n}^{\tilde \ell_n} |ds| 
 \frac{1}{|z-\gamma_n(s)|},\end{equation}
 where the maximum is taken over $(z,w)\in [z_n-iL_n, z_n+iL_n]\times\left(\gamma_n\cap 
U_{L_n}(w_n)\right)$.
If we write $\gamma_n(L_n\sigma)=z_n+L_n\tilde\gamma_n(\sigma)$, $s=L_n\sigma$, and $z=z_n+iL_n\zeta$, then
\begin{equation}\label{eq:estI42}|\tilde{I}_n^{(4)}|\leq  L_n  e^{\max\Re\left(\phi_n(z,u)-\phi_n(w,v)\right)} \int\limits_{-1}^{1}|d\zeta| \int_{-\eta}^{\eta} |d\sigma| 
 \frac{1}{|i\zeta-\tilde\gamma_n(\sigma)|}.\end{equation}
The remaining integral remains bounded by the dominated convergence theorem and since the angle between $\gamma_n$ and $z_n+i\mathbb R$ cannot become small.
By construction of the integration contours as descent paths, we also have
$\max\Re\left(\phi_n(z,u)-\phi_n(w,v)\right)=\Re\left(\phi_n(z_n,u)-\phi_n(w_n,v)\right)$
and it follows that
\begin{equation}\label{eq:In4estimate}\tilde{I}_n^{(4)}=
\mathcal{O}\left(L_n\right),\quad n\rightarrow \infty.\end{equation}

By symmetry, the integral \(\tilde{I}_n^{(6)}\) can be estimated by the same arguments, so it remains to treat 
\(\tilde{I}_n^{(5)}\) and \(\tilde{I}_n^{(7)}\), which again are of the same type so we only have to deal with 
\(\tilde{I}_n^{(5)}\).

\paragraph{\bf Estimation of \(\tilde{I}_n^{(5)}\). }
 On the contour \([z_n-iL_n, z_n+iL_n]\times \left(\gamma_n\backslash U_{L_n}(w_n)\right) \) we have
\begin{equation}\label{def:dntilde}\tilde d_n^{-1}:=\frac{1}{\min|z-w|}\leq \frac{\eta}{L_n},\end{equation}
hence the double integral can be estimated by decoupled integrals,
\[|\tilde{I}_n^{(5)}|\leq \frac{1}{\tilde d_n} \int\limits_{z_n-iL_n}^{z_n+iL_n}|dz| 
e^{\Re\left\{\phi_n(z,u)-\phi_n(z_n,u)\right\}} \int\limits_{\gamma_n\backslash U_{L_n}(w_n)} |dw| 
~e^{-\Re\left\{\phi_n(w,v)-\phi_n(w_n,v)\right\}}. \]
By a Taylor expansion of \(\phi_n(z,u)\) around \(z_n\) and arguments already used above we have
\[\int\limits_{z_n-iL_n}^{z_n+iL_n}|dz| e^{\Re\left\{\phi_n(z,u)-\phi_n(z_n,u)\right\}} 
=\mathcal{O}\left(\frac{1}{\sqrt{n}}\right),\qquad n\to\infty.\]
To estimate the remaining integral we define 
\(\{w_n^{+},w_n^{-}\}=\gamma_n\cap\partial U_{L_n}(w_n)\). By Lemma \ref{lemma:length}, we 
obtain as $n\to\infty$,
\[|\tilde{I}_n^{(5)}|=\mathcal{O}\left(\frac{n^{3}\sqrt{t}}{\tilde d_n \sqrt{n}}\right) 
\left\{e^{-\Re\left\{\phi_n(w_n^{+},v)-\phi_n(w_n,v)\right\}}+e^{-\Re\left\{\phi_n(w_n^{-},v)-\phi_n(w_n,v)\right\}}\right\}.\]
We have by one last Taylor expansion, similarly as before, since $\Re\phi_n''(w_n,v)/n$ tends to a positive constant,
\begin{equation}\label{eq:estphinder}\Re\left\{\phi_n(w_n^{+},v)-\phi_n(w_n,v)\right\}\geq {{\eta}}n L_n^2\end{equation}
and an analogous estimate holds for \(\Re\left\{\phi_n(w_n^{-},v)-\phi_n(w_n,v)\right\}\). This gives
\begin{equation}\label{eq:In5estimate}|\tilde{I}_n^{(5)}|=\mathcal{O}\left(\frac{n^{3}\sqrt{t}}{\tilde d_n\sqrt{n}}e^{-{\eta}nL_n^2}\right)=\mathcal{O}\left(\frac{n^3\sqrt{t}e^{-{\eta} (\log n)^2}}{\log n}\right),\qquad n\to\infty.\end{equation}
Collecting all these estimates on \(\tilde{I}_n(u,v)\), we obtain \eqref{eq:limIn}. By taking into account that 
\[s=\Im z_n =  \Im F_{t,\mu_n}\left({x_t^*}+\frac{u}{c_t n}\right)\rightarrow \Im F_{t,\mu}\left({x_t^*}\right)= 
-\pi\psi_{t}({x_t^*})t,\]
we also obtain \eqref{eq:limAn} via \eqref{eq:estimate s}, and this completes the proof of the theorem.
\end{proof}

\section{Proof of Theorem \ref{thm_nv}}\label{section_thrm_nv}

For the proof of Theorem \ref{thm_nv}, we follow the same strategy as for the proof of Theorem \ref{thrm_t_fixed} in Section \ref{section_t_fixed}, but there are quite some technical issues to be dealt with differently, because of the fact that $t=t_n\to 0$. The most essential difference is that the saddle points $z_n$ and $w_n$ approach the real line as $n\to\infty$. We need to control the speed of convergence in order to obtain a saddle point approximation.

\subsection{Auxiliary results}

\begin{lemma}\label{lemma:comparisonG}
Let \(x^*\) be an interior point of the support of \(\mu\), and suppose that \(\delta>0\) is such that the interval \([x^{\ast}-\delta, x^{\ast}+\delta]\) belongs to the interior of the support. Let $m_n$ be defined by \eqref{assumption:rigidity}. Then there exists a constant $C_1>0$ and an index \(N\in\mathbb{N}\) such that for all \(n>N\), \(\epsilon>0\) and \(x\in [x^{\ast}-\delta, x^{\ast}+\delta]\) we have the inequality
\begin{equation}\label{eq: comparison G}
\left|G_{\mu_n}(x+i\epsilon)-G_{\mu}(x+i\epsilon)\right|\leq \frac{\tilde m_n\pi}{n\epsilon} \leq \frac{C_1(m_n+1)}{n\epsilon}.
\end{equation}
\end{lemma}
\begin{proof}
Let $F_n$ and $F$ denote the distribution functions of $\mu_n$ and $\mu$, respectively. We define
\begin{align*}
g(s):=\frac{1}{x+i\epsilon-s}
\end{align*} 
and choose $a<b$ such that the supports of $\mu$ and $\mu_n$ lie in the interior of $[a,b]$.
Then by integration by parts and the absolute continuity of $g$ we have
\begin{align*}
G_{\mu_n}(x+i\epsilon)-G_{\mu}(x+i\epsilon)&=\int_a^b g\,d(F_n-F)=-\int_a^b(F_n-F)\,dg\\
&=-\int_a^b(F_n-F)(s)g'(s)ds.
\end{align*}
By \eqref{eq:FnF}, we have $\sup_{s\in[a,b]}\lv F_n(s)-F(s)\rv\leq \frac{\tilde m_n}{n}\leq \frac{c(m_n+1)}{n}$, and hence
\begin{align*}
&\lv G_{\mu_n}(x+i\epsilon)-G_{\mu}(x+i\epsilon)\rv\leq \frac{\tilde m_n}{n}\int_a^b 
\frac1{(x-s)^2+\epsilon^2}ds\\
&=\frac{\tilde m_n}{n\epsilon}\lb\arctan\lb\frac{b-x}{\epsilon}\rb-\arctan\lb\frac{a-x}{\epsilon}\rb\rb\leq\frac{\tilde m_n\pi}{n\epsilon}\leq \frac{C_1(m_n+1)}{n\epsilon}.
\end{align*}
\end{proof}

In the following, integral expressions of the form \(\int_{\alpha}^{\beta} f(s) d\mu_n (s)\) are always to be understood as \(\int\limits_{[\alpha,\beta]} f(s) d\mu_n (s)\).

\begin{lemma}\label{lemma:boundy}
Under the conditions of Theorem \ref{thm_nv}, for any $\epsilon>0$, there exist $\delta>0$ and $n_0>0$ such that
\begin{equation}
y_{t_n,\mu_n}(x)\geq (1-\epsilon)t_n\psi(x),
\end{equation}
for any $n>n_0$, $x\in[x^*-\delta, x^*+\delta]$.
\end{lemma}
\begin{proof}
For any $\alpha_n>0$ and $x\in\mathbb R$, we have
the inequality
\begin{equation}
\frac{1}{n}\sum_{k=1}^n\frac{1}{\alpha_n^2+(a_{k}^{(n)}-x)^2}\geq \frac{1}{2\alpha_n^2} \int_{x-\alpha_n}^{x+\alpha_n}d\mu_n(s).
\end{equation}
We define $\alpha_n=(1-\epsilon)\psi(x^*) t_n$, so that by \eqref{eq:comparemunmu} we have 
\[\int_{x-\alpha_n}^{x+\alpha_n}d\mu_n(s)\geq \int_{x-\alpha_n}^{x+\alpha_n}d\mu(s)-\frac{2\tilde m_n}{n},\]
and hence
\begin{equation}
\frac{1}{n}\sum_{k=1}^n\frac{1}{\alpha_n^2+(a_{k}^{(n)}-x)^2}\geq \frac{1}{2\alpha_n^2} \left(\int_{x-\alpha_n}^{x+\alpha_n}d\mu(s) -\frac{2\tilde m_n}{n}\right).
\end{equation}

If $x$ is sufficiently close to $x^*$ and $\alpha_n\to 0$ as $n\to\infty$, there exists $n_0$ such that for $n>n_0$ we have
\[\int_{x-\alpha_n}^{x+\alpha_n}d\mu(s)\geq \left(1-\frac{\epsilon}{4}\right)\int_{x^*-\alpha_n}^{x^*+\alpha_n}d\mu(s)\geq 2\left(1-\frac{\epsilon}{2}\right)\psi(x^*)\alpha_n.\]
We then have
\begin{equation}\label{eq:ineqalphalemma}
\frac{1}{n}\sum_{k=1}^n\frac{1}{\alpha_n^2+(a_{k}^{(n)}-x)^2}\geq \psi(x^*)\left(1-\frac{\epsilon}{2}\right)\alpha_n^{-1}-\frac{\tilde m_n}{n\alpha_n^2}\geq (1-\epsilon)\psi(x^*)\alpha_n^{-1}=\frac{1}{t_n},
\end{equation}
if $n\alpha_n/\tilde m_n\to\infty$, which is true since $nt_n/(m_n+1) \to\infty$ as $n\to\infty$.
By definition of $y_{t_n,\mu_n}$, we have that $\alpha_n\leq y_{t_n,\mu_n}(x)$, and this proves the result.
\end{proof}

\begin{lemma}\label{lemma limyn}
Under the conditions of Theorem \ref{thm_nv}, we have
\begin{equation}
\lim_{n\to\infty}\frac{y_{t_n,\mu_n}(x)}{\pi t_n}=\psi(x),
\end{equation}
uniformly for $x\in[x^*-\delta, x^*+\delta]$ with $\delta$ sufficiently small.
\end{lemma}
\begin{proof}
Combining \eqref{eq:Gtau} and \eqref{eq:psitau}, we obtain
\[\frac{y_{t_n,\mu_n}(x)}{\pi t_n}=-\frac{1}{\pi}\Im G_{\mu_n}(x+iy_{t_n,\mu_n}(x))\]
for real $x$. Using Lemma \ref{lemma:boundy} and Lemma \ref{lemma:comparisonG}, we can conclude that
\[\frac{y_{t_n,\mu_n}(x)}{\pi t_n}=-\frac{1}{\pi}\Im G_{\mu}(x+iy_{t_n,\mu_n}(x))+\mathcal O\left(\frac{\tilde{m}_n}{nt_n}\right)\to \psi(x)\]
as $n\to\infty$ with $nt_n/\tilde{m}_n\to\infty$, uniformly for $x$ sufficiently close to $x^*$.
\end{proof}

\begin{lemma}\label{lemma:znwn2}
Under the conditions of Theorem \ref{thm_nv}, for $u, v$ in a compact set, there exists a constant $M>0$ such that $z_n=z_n(u,t_n)$ and $w_n=w_n(v,t_n)$ satisfy
\begin{equation}
|z_n-w_n|\leq \frac{M}{n},
\end{equation}
for sufficiently large $n$.
\end{lemma}
\begin{proof}
We have 
\[z_n=F_{t_n,\mu_n}\left(x_{t_n}^*+\frac{u}{c_{t_n}n}\right),\qquad w_n=F_{t_n,\mu_n}\left(x_{t_n}^*+\frac{v}{c_{t_n}n}\right).\]
To prove the result, it is sufficient to show that
$|F_{t_n,\mu_n}'(x)|$ remains uniformly bounded for $x\in \mathbb R$ and $|x-x_{t_n}^*|<\frac{K}{n}$ for a sufficiently large constant $K>0$.
We have \[F_{t_n,\mu_n}'(x)=\frac{1}{H_{t_n,\mu_n}'(F_{t_n,\mu_n}(x))}.\]
Using the fact that $\int\frac{d\mu_n(s)}{|z-s|^2}=\frac{1}{t_n}$ for $z=F_{t_n,\mu_n}(x)$ on positive parts of the graph of $y_{t_n,\mu_n}$, we have
\begin{align}
|H_{t_n,\mu_n}'(z)|&=\left|1-t_n\int \frac{d\mu_n(s)}{(z-s)^2}\right|\label{H'_bound0}\\
&=t_n\left|\int \frac{d\mu_n(s)}{|z-s|^2}-\int \frac{d\mu_n(s)}{(z-s)^2}\right|\notag\\
&=t_n\left|\int \frac{|z-s|^2-(\overline{z}-s)^2}{|z-s|^4}d\mu_n(s)\right|.\notag
\end{align}
The modulus of the latter integral can be bounded below by the absolute value of its real part, and this gives
\begin{equation}\label{H'_bound}
|H_{t_n,\mu_n}'(z)|\geq 2t_n (\Im z)^2 \int \frac{d\mu_n(s)}{|z-s|^4}\geq 2t_n (\Im z)^2 \int_{\Re z_n-t_n}^{\Re z_n+t_n} \frac{d\mu_n(s)}{|z-s|^4}.
\end{equation}
By Lemma \ref{lemma:boundy}, this can be further estimated by
\begin{equation}\label{H'_bound2}
|H_{t_n,\mu_n}'(z)|\geq  2\tilde C t_n^{-1} \int_{\Re z_n-t_n}^{\Re z_n+t_n} d\mu_n(s),
\end{equation}
for some $\tilde C>0$,
and the right hand side of the latter expression is bounded below by a positive constant since $\int_{\Re z_n-t_n}^{\Re z_n+t_n} d\mu_n(s)\sim  2\psi(x^*)t_n$ as $n\to\infty$, which follows from a straightforward argument using \eqref{eq:comparemunmu} and the fact that $nt_n/\tilde m_n\to\infty$.
\end{proof}

\subsection{Asymptotics of \texorpdfstring{$A_n$}{An}}

It is now straightforward to give the asymptotics of $A_n$. We have
\[A_n(u,v)=\frac{1}{\pi (u-v)} \sin\left(\frac{s}{c_{t_n} t_n} (u-v)\right),\]
where $s$ is the imaginary part of the intersection of 
$\gamma$ with the vertical line through $x_0=x_n=\Re F_{t_n,\mu_n}\left(x_{t_n}^*+\frac{u}{c_{t_n} n}\right)$. In other words, if we make the same choice of integration contours in \eqref{I_n} as before, then \[s=\Im F_{t_n,\mu_n}\left(x_{t_n}^*+\frac{u}{c_{t_n} n}\right)=y_{t_n,\mu_n}\left(x_n\right).\]
Since $x_n\to x^*$, $c_{t_n}\to \psi(x^*)$ as $n\to\infty$ and $t_n\to 0$, it follows from Lemma \ref{lemma limyn} that
$\frac{s}{c_{t_n} t_n}\to \pi$, and consequently
\begin{equation}
\lim_{n\to\infty}A_n(u,v)=\frac{\sin(\pi (u-v))}{\pi (u-v)}.
\end{equation}

\subsection{Estimation of \texorpdfstring{$I_n$}{In}}
Here, we follow the estimates done in Section \ref{section_t_fixed} for the proof of Theorem \ref{thrm_t_fixed}. As already specified in the analysis of $A_n$, we take the same integration contours in \eqref{I_n} as before, by following the graph of $y_{t_n,\mu_n}$ and its complex conjugate for the contour $\gamma_n$, and by taking $x_0=x_n=\Re F_{t_n,\mu_n}(x_{t_n}^*+\frac{u}{c_{t_n} n})$.
At first sight, the fact that the saddle points $z_n$ and $w_n$ of the phase function $\phi_n$ approach the real line may appear problematic, as one may expect that the contributing neighborhoods of $z_n, \overline{z_n}$ and $w_n, \overline{w_n}$ will overlap.
However, this is not the case, essentially because the second derivative of the phase function $\phi_n$ blows up rapidly, which means that the contributing neighborhoods to the integral $I_n$ become small as well.
More concretely, in view of \eqref{eq:derivative}, we observe that
\begin{align*}
\beta_n=-\frac{d^2}{d\t^2}&\Big\vert_{\t=y_n}\Re \phi_n(x_n+i\t,u)=2ny_n^2 \int\frac{d\mu_n(a)}{\left((x_n-a)^2+y_n^2\right)^2}\\
&\geq  2n y_n^2 \frac{1}{(t_n^2+y_n^2)^2}\int_{x_n-t_n}^{x_n+t_n}d\mu_n(s).
\end{align*}
Using Lemma \ref{lemma limyn} and the fact that $n\int_{x_n-t_n}^{x_n+t_n}d\mu_n(s)> \epsilon n t_n$ for sufficiently small $\epsilon>0$ under the conditions of Theorem \ref{thm_nv} (as before, this follows from \eqref{eq:comparemunmu} and the fact that $nt_n/(m_n+1)\to\infty$), we obtain for some $K>0$ that
\begin{equation}\label{eq: secondder2}
\beta_n\geq \frac{Kn}{t_n}.
\end{equation} 
With this in mind, we can proceed along the same lines as in the proof of Theorem \ref{thrm_t_fixed}, splitting the integral $\tilde I_n$ in seven parts $\tilde I_n^{(n)},\ldots, \tilde I_n^{(7)}$. Since the imaginary parts of the saddle points are proportional to $t_n$ as $n\to\infty$, we need to take $L_n$ such that $L_n/t_n\to 0$, to avoid overlapping neighborhoods of the saddle points. A convenient choice will appear to be, instead of \eqref{def:Ln}, 
\begin{equation}
\label{def:Ln2}L_n=\sqrt{\frac{t_n}{n}}(\log n)^{\frac{1+\rho}{2}}.
\end{equation}

\paragraph{\bf Estimation of $\tilde I_n^{(4)}$.}
The estimates \eqref{eq:estI4}, \eqref{eq:estI42}, and \eqref{eq:In4estimate} remain valid using the same calculations as in Section \ref{section_t_fixed}, provided that we can show that the length of $\gamma_n\cap U_{L_n}(w_n)$ is $O(L_n)$ as $n\to\infty$, and that the angle between $\gamma_n$ and $z_n+i\mathbb R$ stays away from $0$ for large $n$.
To see this, note that for $w\in\gamma_n\cap U_{L_n}(w_n)$,
\begin{multline*}
H_{t_n,\mu_n}'(w)=1-t_n\int\frac{d\mu_n(s)}{(w-s)^2}=t_n\left(\int\frac{d\mu_n(s)}{|w-s|^2}- \int\frac{d\mu_n(s)}{(w-s)^2}\right)\\
=t_n\int \frac{|w-s|^2- (\overline w-s)^2}{|w-s|^4} d\mu_n(s).
\end{multline*}
Hence, 
\begin{multline*}
\Re H_{t_n,\mu_n}'(w)=2t y_{t_n,\mu_n}(\Re w)^2\int\frac{d\mu_n(s)}{|w-s|^4}\geq 2t_n y_{t_n,\mu_n}(\Re w)^2\int_{\Re w - y_{t_n,\mu_n}(\Re w)}^{\Re w + y_{t_n,\mu_n}(\Re w)}\frac{d\mu_n(s)}{|w-s|^4}\\
\geq \eta \frac{t_n}{y_{t_n,\mu_n}(\Re w)^2} \int_{\Re w - y_{t_n,\mu_n}(\Re w)}^{\Re w + y_{t_n,\mu_n}(\Re w)}d\mu_n(s).
\end{multline*}
By Lemma \ref{lemma limyn}, \eqref{eq:comparemunmu}, and the fact that $nt_n/(m_n+1)\to\infty$, this is bounded below by a positive constant.
It is also easily seen that $|H_{t_n,\mu_n}'(w)|\leq 2$, and it follows that the argument of $H_{t_n,\mu_n}'(w)$ remains bounded away from $\pm \pi/2$.
By \eqref{eq:identityHy}, this implies that $y_{t_n,\mu_n}'(\Re w)$ is bounded for $w\in\gamma_n\cap U_{L_n}(w_n)$, and then it easily follows that the length of $\gamma_n\cap U_{L_n}(w_n)$ is $\mathcal O(L_n)$ as $n\to\infty$, and that the angle between $\gamma_n$ and $z_n+i\mathbb R$ does not approach $0$.

\paragraph{\bf Estimation of $I_n^{(1)}$.}
Using the fact that $y_{t_n,\mu_n}'$ is bounded near $x^*$ (which we showed in the above paragraph), we see easily that the inequality \eqref{def:dn}, where the minimum is over $(z,w)\in [z_n+iL_n, z_n+i\infty)\times \gamma_n$, still holds.

By Lemma \ref{lemma:znwn2}, we still have \eqref{eq:znwnestimate}.
Using the fact that
\[n\left|g_{\mu_n}(z_n)-g_{\mu_n}(w_n)\right|=n\left|\int_{w_n}^{z_n}G_{\mu_n}(s)ds\right|\leq n\int_{w_n}^{z_n}\left|G_{\mu_n}(s)\right|\ |ds|,\]
and combining this with Lemma \ref{lemma:boundy} and Lemma \ref{lemma:comparisonG}, we obtain that \eqref{eq:phinznwnestimate} also holds. In the remaining part of the analysis of $I_n^{(1)}$, we follow similar estimates as in the proof of Theorem \ref{thrm_t_fixed}, but we need to estimate the error term $R_2(z)$ in the Taylor expansion in a different way. Instead of \eqref{eq:estimateR2}, we have for $|z-z_n|\leq L_n$
\begin{equation}
|R_2(z)|=\left|\int_{z_n}^{z}d\xi_1\int_{z_n}^{\xi_1}d\xi_2\int_{z_n}^{\xi_2}d\xi_3 \, \phi_n'''(\xi_3;u)\right|\leq \max|\phi_n'''(\xi;u)|  L_n^3,
\label{eq:estimatephi3a}\end{equation}
where we integrate over line segments, and where the maximum is over the line segment between $z_n$ and $z$.
Next, given $\xi$ on this line segment, we let $\gamma$ be the circle of radius $\frac{1}{2}\Im\xi\geq \frac{1}{2}(\Im z_n - L_n)$ around $\xi$, and we use Cauchy's theorem to estimate the maximum,
\begin{equation}\max|\phi_n'''(\xi;u)|=\left|\frac{1}{2\pi}\int_{\gamma} \frac{\phi_n''(\zeta;u)}{(\zeta-\xi)^2}d\xi\right|\leq \frac{2}{\Im\xi} \max_{\zeta\in\gamma}|\phi_n''(\zeta;u)|\leq \frac{2n\max_{\zeta\in\gamma}|1+t_nG_{\mu_n}'(\zeta)|}{t_n(\Im z_n-L_n)}.\label{eq:estimatephi3b}\end{equation}
By Lemma \ref{lemma:boundy}, it follows that $\Im z_n$ is bounded below by $\tilde{c} t_n$ for a sufficiently small constant $\tilde{c}>0$. 
Since  $L_n=o(t_n)$ as $n\to\infty$, we have
\[\Im z_n-L_n\geq \frac{\tilde{c}}{2}t_n,\qquad \Im\zeta\geq \frac{1}{2}\Im\xi\geq \frac{1}{2}(\Im z_n - L_n)\geq \frac{\tilde{c}}{4}t_n,\]
for $n$ sufficiently large.
Now we can use Lemma \ref{lemma limyn} to conclude that there exists a sufficiently small constant $\widehat c>0$, uniform in $\zeta$,  such that $y_{\widehat c t_n, \mu_n}(\Re\zeta)\leq \frac{\tilde{c}}{4}t_n\leq  \Im\zeta$. But for such a constant $\widehat c$, we have
\[\left|G_{\mu_n}'(\zeta)\right|\leq \int \frac{1}{|\zeta-s|^2} d\mu_n(s)\leq \frac{1}{\widehat c t_n},\]
by definition of $y_{t_n,\mu_n}$.
Substituting this in \eqref{eq:estimatephi3a} and \eqref{eq:estimatephi3b}, we obtain
\[|R_2(z)|\leq \frac{4nL_n^3}{\tilde{c} t_n^2} \left(1+\frac{1}{\widehat c}\right)=o(\beta_n L_n^2),\qquad n\to\infty.\]

Then, proceeding as in the proof of 
Theorem \ref{thrm_t_fixed}, we get
\begin{equation}
\label{eq:estimateI1-2}
\tilde I_n^{(1)}=\mathcal O\left(\frac{n^3\sqrt{t_n}}{d_n}e^{-\frac{\beta_n}{4} L_n^2}\right)=\mathcal O\left(n^{4}e^{-\eta (\log n)^{1+\rho}}\right)=o(1),\qquad n\to\infty.
\end{equation}
The same bound applies to $\tilde I_n^{(2)}$ and $\tilde I_n^{(3)}$.

\paragraph{\bf Estimation of $\tilde I_n^{(5)}$.}
For $\tilde I_n^{(5)}$ and $\tilde I_n^{(7)}$, we first note that \eqref{def:dntilde} still holds. Using a Taylor expansion and error estimate similar to the one for $\tilde I_n^{(1)}$, one verifies that
\eqref{eq:estphinder} improves to\[\Re\left\{\phi_n(w_n^{+},v)-\phi_n(w_n,v)\right\}\geq {{\eta}}\frac{n}{t_n} L_n^2\]
because $\frac{t_n}{n}\Re\phi_n''(w_n,v)$ is bounded in absolute value und bounded away from zero as $n\to\infty$. This then leads to, instead of \eqref{eq:In5estimate},
\begin{equation}\label{eq:In5estimate2}|\tilde{I}_n^{(5)}|=\mathcal{O}\left(\frac{n^{3}t_n}{\tilde d_n\sqrt{n}}e^{-{\eta}\frac{n}{t_n}L_n^2}\right)=\mathcal{O}\left(n^3 e^{-{\eta} (\log n)^{1+\rho}}\right),\qquad n\to\infty.\end{equation}
Combining the above estimates, we get that 
$\lim_{n\to\infty}\tilde I_n=0$, which completes the proof of Theorem \ref{thm_nv}.

It remains to consider the factor $\exp(-\frac{1}{2nt_n c_{t_n}^2}(u^2-v^2)+\frac{1}{c_{t_n}t_n}(u-v)(x_0-x^*_{t_n}))$ that was neglected when passing from $I_n$ to $\tilde{I}_n$. We will show that $x_0-x^*_{t_n}=x_n-x^*_{t_n}=\O(t_n)$, which clearly suffices. Expanding $F_{t_n,\mu_n}$ around $H_{t_n,\mu_n}(x^*+iy_{t_n,\mu_n}(x^*))$, we get by the convergence of $c_{t_n}$ and the uniform boundedness of $F'_{t_n,\mu_n}$ around $x^*_{t_n}$ (see the proof of Lemma \ref{lemma:znwn2}) 
\begin{align*}
F_{t_n,\mu_n}\lb x_{t_n}^*+\frac u{c_{t_n}n}\rb=x^*+\O\lb\frac1n+\left\lv H_{t_n,\mu_n}(x^*+iy_{t_n,\mu_n}(x^*))-H_{t_n,\mu}(x^*+iy_{t_n,\mu}(x^*))\right\rv\rb.
\end{align*}
Equation \eqref{H'_bound0} and $t_n\int\frac{d\mu_n(s)}{|z-s|^2}=1$ show that $H'_{t_n,\mu_n}$ is uniformly bounded. Since by Lemma \ref{lemma limyn} $y_{t_n,\mu_n}(x^*)-y_{t_n,\mu}(x^*)=\O(t_n)$, we have $H_{t_n,\mu_n}(x^*+iy_{t_n,\mu_n}(x^*))-H_{t_n,\mu_n}(x^*+iy_{t_n,\mu}(x^*))=\O(t_n)$. Furthermore, Lemmas \ref{lemma:comparisonG} and \ref{lemma limyn} show that $H_{t_n,\mu_n}(x^*+iy_{t_n,\mu}(x^*))-H_{t_n,\mu}(x^*+iy_{t_n,\mu}(x^*))=\O(1/n)$. It remains to bound $\lv x^*-x_{t_n}^*\rv$. Since $H_{t_n,\mu}$ maps $x^*+iy_{t_n,\mu}(x^*)$ to the real line, we have $x^*-x_{t_n}^*=t_n\Re G_\mu(x^*+iy_{t_n,\mu}(x^*))$. Now, since $y_{t_n,\mu}(x^*)\to0$, $\Re G_\mu(x^*+iy_{t_n,\mu}(x^*))$ converges to the Hilbert transform
\begin{align*}
\int \frac{d\mu(s)}{x^*-s}
\end{align*}
of $\mu$ at $x^*$, the integral being understood as a principal value integral.
This proves $x_n-x^*_{t_n}=\O(t_n)$.

\section{Proof of Theorem \ref{thrm_v}}\label{section_thrm_v}
Throughout this section we assume Assumptions 1 and 2, $t_n$ is a sequence such that \(t_n \to 0\), \(\frac{n t_n^{\frac{1+\k}{1-\k}}}{(\log n)^{1+\rho}}\to \infty\), and \(\frac{n t_n^{\frac{1+\k}{1-\k}}}{m_n+1}\to \infty\), as \(n\to \infty\), and \(x^{\ast}\) is an interior point of the support of \(\mu\) such that \(\psi(x)\sim C\vert x-x^{\ast}\vert^{\k}\), as \(x\to x^{\ast}\), with \(0<\k<1\). 
\subsection{Auxiliary results}
For the large $n$ asymptotics of $A_n$, we need asymptotic equivalence of $y_{t_n,\mu}(x^*)$ and $y_{t_n,\mu_n}(x_n)$, where $x_n=\Re z_n$. First, the asymptotic equivalence is shown for  $y_{t_n,\mu}(x^*)$ and $y_{t_n,\mu_n}(x^*)$ (Lemma \ref{Lemma4.5}). By a Taylor expansion of $y_{t_n,\mu_n}(x^*)$ around $x^*$, this is transferred to $y_{t_n,\mu_n}(x_n)$ (Lemma \ref{Lemma4.7}). The necessary estimates on the derivative $y_{t_n,\mu_n}'$  and the difference $\lv x_n-x^*\rv$ are contained in Lemmas \ref{Lemma4.4} and \ref{Lemma4.6}, respectively. 
\begin{lemma}\label{Lemma4.1}
	For every \(\epsilon>0\) there exist \(\delta>0 \) and \(n_0\in \mathbb{N}\) such that for every \(n>n_0\) and \(x\in [x^{\ast}-\delta, x^{\ast}+\delta] \) we have
	\[y_{t_n,\mu_n}(x)\geq (1-\epsilon) \left(\frac{C}{\k +1}\right)^{\frac{1}{1-\k}} t_n^{\frac{1}{1-\k}}.\]
\end{lemma}
\begin{proof}
We define the sequence \(\alpha_n=\left(\frac{(1-2\tilde{\epsilon})C t_n}{\k +1}\right)^{\frac{1}{1-\k}}\) for some \(\tilde{\epsilon}\in \left(0, \frac{1}{2}\right). \) Then we have \(n\alpha_n^{\k+1}\to \infty\) as \(n\to \infty\). By \eqref{eq:comparemunmu} we have for any real \(x\) that 
\[\int_{x-\alpha_n}^{x+\alpha_n} d\mu_n(s) \geq \int_{x-\alpha_n}^{x+\alpha_n} d\mu(s)-\frac{2\tilde m_n}{n}.\]
Moreover, there exists a suitably small \(\delta>0\) such that for \(x\in[x^{\ast}-\delta, x^{\ast}+\delta]\) and \(n>n_0\)
\[\int_{x-\alpha_n}^{x+\alpha_n} d\mu(s) \geq \left(1-\frac{\tilde{\epsilon}}{4}\right)\int_{x^{\ast}-\alpha_n}^{x^{\ast}+\alpha_n} d\mu(s) \geq \frac{2(1-\tilde{\epsilon}) C \alpha_n^{\k+1}}{\k +1}. \] 
This gives
\begin{align*}
\frac{1}{n} \sum_{k=1}^n \frac{1}{\alpha_n^2 + (a_k^{(n)}-x)^2} &\geq \frac{1}{2\alpha_n^2 }\int_{x-\alpha_n}^{x+\alpha_n} d\mu_n(s) \geq \frac{1}{2\alpha_n^2 } \left\{\frac{2(1-\tilde{\epsilon}) C \alpha_n^{\k+1}}{\k +1} -\frac{2\tilde m_n}{n} \right\}\\
&\geq \frac{(1-2\tilde{\epsilon})C}{\k+1} \alpha_n^{\k+1} = \frac{1}{t_n}.
\end{align*}
Here we used the fact that \(\frac{n t_n^{\frac{1+\k}{1-\k}}}{m_n+1}\to \infty\) and hence $\frac{\tilde m_n}{n}=o(\alpha_n^{\kappa+1})$ as $n\to\infty$.
It follows that \(y_{t_n,\mu_n}(x) \geq \alpha_n\) for \(n>n_0\) and \(x\in[x^{\ast}-\delta, x^{\ast}+\delta]\). The statement now easily follows by choosing a suitably small \(\tilde{\epsilon}\).
\end{proof}

\begin{lemma} \label{Lemma4.2} Let \(\alpha\in \mathbb{R}\) and \( \beta >0\). Then there exist constants \(K,K' >0\) such that 
	\[\vert \Im G_\mu (x^{\ast}+ \alpha+  i\beta ) \vert \leq K \vert \alpha \vert^{\k} + K' \beta^{\k},\]
	as \(\alpha + i \beta \to 0\).
	
\end{lemma}

\begin{proof}
	Let \(\mu\) be supported on \([a,b]\). Then we obtain for some constant \(K_1>0\)
	\[\vert \Im G_\mu (x^{\ast}+ \alpha+  i\beta ) \vert \leq K_1 \beta  \int_a ^{b} \frac{\vert s-x^{\ast}\vert^{\k}}{(x^{\ast} -s +\alpha)^2 + \beta^2} ds.\]
	Assuming without loss of generality \(\alpha >0 \) and splitting up the integral into three parts we obtain
	\begin{align*}\vert \Im G_\mu (x^{\ast}+ \alpha+  i\beta ) \vert &\leq K_1  \beta \Bigg\{\int_a ^{x^{\ast}} \frac{\vert s-x^{\ast}\vert^{\k}}{(x^{\ast} -s)^2 + \beta^2} ds+ \int_{x^{\ast}} ^{x^{\ast}+\alpha} \frac{\vert s-x^{\ast}\vert^{\k}}{(x^{\ast}+\alpha -s)^2 + \beta^2} ds \\
	&+ \int_{x^{\ast}+\alpha} ^{b} \frac{\vert s-x^{\ast}\vert^{\k}}{(x^{\ast}+\alpha -s)^2 + \beta^2} ds\Bigg\}.
	\end{align*}
For the first integral we have for some constant \(K_2 >0\)
	\[\int_a ^{x^{\ast}} \frac{\vert s-x^{\ast}\vert^{\k}}{(x^{\ast} -s)^2 + \beta^2} ds =\beta^{\k -1} \int_{\frac{a-x^{\ast}}{\beta}} ^{0} \frac{\vert s\vert^{\k}}{s^2 +1}\leq K_2 \beta^{\k -1}.\]
For the second integral we have for some constant \(K_3 >0\)
	\[\int_{x^{\ast}} ^{x^{\ast}+\alpha} \frac{\vert s-x^{\ast}\vert^{\k}}{(x^{\ast}+\alpha -s)^2 + \beta^2} ds = \alpha^{\k-1}\int_{0} ^{1} \frac{s^{\k}}{(1-s)^2+\left(\frac{\beta}{\alpha}\right)^2}ds \leq K_3 \frac{\alpha^{\k}}{\beta}.\]
Finally, for the third integral we obtain for some constant \(K_4>0\)
	\[\int_{x^{\ast}+\alpha} ^{b} \frac{\vert s-x^{\ast}\vert^{\k}}{(x^{\ast}+\alpha -s)^2 + \beta^2} ds\leq \int_{x^{\ast}+\alpha} ^{b} \frac{\vert s-x^{\ast}-\alpha\vert^{\k} + \alpha^{\k}}{(x^{\ast}+\alpha -s)^2 + \beta^2} ds \leq K_4 \left(\frac{\alpha^{\k}}{\beta}+ \beta^{\k-1}\right),\]
which proves the statement.
	
\end{proof}

\begin{lemma} \label{Lemma4.3}
For every constant \(K>0\) we have
\[\limsup_{n\to\infty} \sup_{x\in\left[x^{\ast}-Kt_n^{\frac{1}{1-\k}},~ x^{\ast}+Kt_n^{\frac{1}{1-\k}} \right]} \frac{y_{t_n,\mu_n} (x)}{t_n^{\frac{1}{1-\k}}} <\infty.\]
In other words, we have \(y_{t_n,\mu_n} (x) = \mathcal{O}(t_n^{\frac{1}{1-\k}})\), if \(x-x^{\ast}=\mathcal{O}(t_n^{\frac{1}{1-\k}})\), as \(n\to\infty\).
\end{lemma}

\begin{proof}
By Lemma \ref{lemma:comparisonG} and Lemma \ref{Lemma4.1} we have
\[\frac{y_{t_n,\mu_n}(x)}{\pi t_n} = -\frac{1}{\pi} \Im G_{\mu_n}(x+iy_{t_n,\mu_n}(x)) = -\frac{1}{\pi} \Im G_{\mu}(x+iy_{t_n,\mu_n}(x)) + \mathcal{O}\left(\frac{m_n +1}{n t_n^{\frac{1}{1-\k}}}\right),\]
as \(n\to\infty\). Now, using Lemma \ref{Lemma4.2} we obtain
\[y_{t_n,\mu_n}(x) = \mathcal{O}\left( t_n\left[(x-x^{\ast})^\k + y_{t_n,\mu_n}(x)^{\k}\right]+\frac{m_n+1}{nt_n^{\frac{\k}{1-\k}}} \right)=\mathcal{O}\left(t_n t_n^{\frac{\k}{1-\k}}+t_n y_{t_n,\mu_n}(x)^{\k} \right) ,\]
as \(n\to\infty\). Dividing this by \(y_{t_n,\mu_n}(x)^{\k}\) gives 
\[y_{t_n,\mu_n}(x)^{1-\k} =\mathcal{O}\left(t_n\right),\]
as \(n\to\infty\), from which the statement follows.
\end{proof}

\begin{lemma}\label{Lemma4.4}
For every constant \(K>0\) we have
\[\limsup_{n\to\infty} \sup_{x\in\left[x^{\ast}-Kt_n^{\frac{1}{1-\k}},~ x^{\ast}+Kt_n^{\frac{1}{1-\k}} \right]} \vert y_{t_n,\mu_n}'(x) \vert  <\infty.\]
In other words, we have \(y_{t_n,\mu_n}' (x) = \mathcal{O}(1)\), if \(x-x^{\ast}=\mathcal{O}(t_n^{\frac{1}{1-\k}})\), as \(n\to\infty\).
\end{lemma}

\begin{proof} If \(x-x^{\ast}=\mathcal{O}(t_n^{\frac{1}{1-\k}})\), as \(n\to\infty\), we have for positive constants \(C_j\) \((j=1,\ldots,5)\) and \(n\) large enough
\begin{align*} &\vert H'_{t_n,\mu_n} (x+iy_{t_n,\mu_n}(x)) \vert \geq \Re H'_{t_n,\mu_n} (x+iy_{t_n,\mu_n}(x)) = 2t_n y_{t_n,\mu_n}(x)^2 \int \frac{d\mu_n(s)}{\vert x-s+iy_{t_n,\mu_n}(x) \vert^4}\\
& \geq 2t_n y_{t_n,\mu_n}(x)^2 \int_{x-y_{t_n,\mu_n}(x)}^{x+y_{t_n,\mu_n}(x)}\frac{d\mu_n(s)}{\left((x-s)^2+y_{t_n,\mu_n}(x)^2\right)^2}\geq \frac{t_n}{2y_{t_n,\mu_n}(x)^2} \int_{x-y_{t_n,\mu_n}(x)}^{x+y_{t_n,\mu_n}(x)} d\mu_n(s)\\
& \geq \frac{t_n}{2y_{t_n,\mu_n}(x)^2} \left\{ \int_{x-y_{t_n,\mu_n}(x)}^{x+y_{t_n,\mu_n}(x)} d\mu(s)-\frac{C_1\tilde m_n}{n}\right\}\\
&\geq \frac{t_n}{2y_{t_n,\mu_n}(x)^2} \left\{ \frac{1}{2} \int_{x^{\ast}-y_{t_n,\mu_n}(x)}^{x^{\ast}+y_{t_n,\mu_n}(x)} d\mu(s)-\frac{C_1\tilde m_n}{n}\right\}\\
& \geq \frac{t_n}{2y_{t_n,\mu_n}(x)^2} \left\{ C_2 y_{t_n,\mu_n}(x)^{\k+1}-\frac{C_1\tilde m_n}{n}\right\}= C_3 t_n y_{t_n,\mu_n}(x)^{\k-1} -\frac{C_4 t_n \tilde m_n}{n y_{t_n,\mu_n}(x)^2} \geq C_5>0,
\end{align*}
where in the last estimates we used the Lemmas \ref{Lemma4.1} and \ref{Lemma4.3}. Moreover, we have 
\[\vert H'_{t_n,\mu_n} (x+iy_{t_n,\mu_n}(x)) \vert \leq 2\]
and
\[\vert y'_{t_n,\mu_n}(x)\vert = \vert \tan \left(-\arg H'_{t_n,\mu_n} (x+iy_{t_n,\mu_n}(x))\right)\vert. \]
From the above computation we know that \(\arg H'_{t_n,\mu_n} (x+iy_{t_n,\mu_n}(x)) \in \left[-\frac{\pi}{2}+\epsilon, \frac{\pi}{2}-\epsilon\right]\) for some small \(\epsilon >0\), so that the statement follows.
\end{proof}

\begin{lemma}\label{Lemma4.5}
We have 
\[y_{t_n,\mu}(x^{\ast})\sim \left(\frac{C\pi t_n}{\sin\left(\pi \frac{1+\k}{2}\right)}\right)^{\frac{1}{1-\k}}, \quad n\to\infty,\]
and
\[y_{t_n,\mu_n}(x^{\ast})\sim \left(\frac{C\pi t_n}{\sin\left(\pi \frac{1+\k}{2}\right)}\right)^{\frac{1}{1-\k}}, \quad n\to\infty.\]
\end{lemma}

\begin{proof} We have for a sequence \(\delta_n\to 0\) and \(\frac{t_n}{\delta_n^2}\to 0\) as \(n\to\infty\)
\[1=t_n \int\frac{d\mu(s)}{(s-x^{\ast})^2+y_{t_n,\mu}(x^{\ast})^2}= t_n \int_{x^{\ast}-\delta_n}^{x^{\ast}+\delta_n}\frac{d\mu(s)}{(s-x^{\ast})^2+y_{t_n,\mu}(x^{\ast})^2}+o(1),\]
as \(n\to\infty\). Moreover, we have
\begin{align*}t_n \int_{x^{\ast}-\delta_n}^{x^{\ast}+\delta_n}\frac{d\mu(s)}{(s-x^{\ast})^2+y_{t_n,\mu}(x^{\ast})^2}&\sim C t_n \int_{-\delta_n}^{\delta_n} \frac{\vert s\vert^{\k}}{s^2+y_{t_n,\mu}(x^{\ast})^2}ds\\
&= C t_n y_{t_n,\mu}(x^{\ast})^{\k-1} \int_{-\delta_n/y_{t_n,\mu}(x^{\ast})}^{\delta_n/y_{t_n,\mu}(x^{\ast})}\frac{\vert s\vert^{\k}}{s^2+1}ds.
\end{align*}
Since \(\frac{\delta_n}{y_{t_n,\mu}(x^{\ast})}\to\infty\), as \(n\to\infty\), and using the identity 
\[\int_{-\infty}^{\infty}\frac{\vert s\vert^{\k}}{s^2+1}ds=\frac{\pi}{\sin\left(\pi \frac{\k+1}{2}\right)}\]
the first part of the statement follows. For the second part we consider

\begin{align*} 1&=t_n \int\frac{d\mu_n(s)}{(s-x^{\ast})^2+y_{t_n,\mu_n}(x^{\ast})^2} = -\frac{t_n}{y_{t_n,\mu_n}(x^{\ast})}\Im G_{\mu_n}\left(x^{\ast}+i y_{t_n,\mu_n}(x^{\ast})^2\right)\\
&=t_n \int\frac{d\mu(s)}{(s-x^{\ast})^2+y_{t_n,\mu_n}(x^{\ast})^2}\\
&\qquad \qquad-\frac{t_n}{y_{t_n,\mu_n}(x^{\ast})} \left\{\Im G_{\mu_n}\left(x^{\ast}+i y_{t_n,\mu_n}(x^{\ast})\right)-\Im G_{\mu}\left(x^{\ast}+i y_{t_n,\mu_n}(x^{\ast})\right)\right\}.
\end{align*}
By Lemma \ref{lemma:comparisonG} and Lemma \ref{Lemma4.1} we have
\[\left\vert \frac{t_n}{y_{t_n,\mu_n}(x^{\ast})} \left\{\Im G_{\mu_n}\left(x^{\ast}+i y_{t_n,\mu_n}(x^{\ast})\right)-\Im G_{\mu}\left(x^{\ast}+i y_{t_n,\mu_n}(x^{\ast})\right)\right\}\right\vert = \mathcal{O}\left(\frac{\tilde m_n}{n t_n^{\frac{1+\k}{1-\k}}}\right)=o(1),\]
as \(n\to\infty\). Hence, for a sequence \(\delta_n\to 0\) and \(\frac{t_n}{\delta_n^2}\to 0\), as \(n\to\infty\), we have
\[1=t_n \int_{x^{\ast}-\delta_n}^{x^{\ast}+\delta_n} \frac{d\mu(s)}{(x^{\ast}-s)^2+y_{t_n,\mu_n}(x^{\ast})^2}+o(1),\]
as \(n\to\infty\). From here we can proceed as in the first part of the proof to obtain the second statement.
\end{proof}

For the next Lemmas we recall that by definition of the saddle points we have for real \(u,v\):

\[z_n= F_{t_n,\mu_n}\left(H_{t_n,\mu}\left(x^{\ast}+y_{t_n,\mu}(x^\ast)\right)+\frac{u}{nc_{t_n}}\right),\]
\[w_n= F_{t_n,\mu_n}\left(H_{t_n,\mu}\left(x^{\ast}+y_{t_n,\mu}(x^\ast)\right)+\frac{v}{nc_{t_n}}\right),\]
\[x_n=\Re z_n,\quad \Im z_n =y_{t_n,\mu_n}(x_n),\]
with \[c_{t_n}=\psi_{t_n}(x_{t_n}^*)=\frac{y_{t_n,\mu}(x^*)}{\pi t_n}.\]

\begin{lemma} \label{Lemma4.6}
If \(u\) and \(v\) belong to a compact subset, then there exist positive constants \(K,K'\) and \(n_0\in\mathbb{N}\) independent of \(u\) and \(v\), such that for \(n>n_0\) we have
\[\vert x_n-x^{\ast} \vert \leq \frac{K \tilde m_n}{n t_n^{\frac{\kappa}{1-\kappa}}}\]
and
\[\vert z_n-w_n \vert \leq \frac{K' \tilde m_n}{n t_n^{\frac{\kappa}{1-\kappa}}}.\]
Thus, we have \(x_n-x^{\ast}= \mathcal{O}\left(\frac{m_n+1}{n t_n^{\frac{\kappa}{1-\kappa}}}\right)\) and \(z_n-w_n= \mathcal{O}\left(\frac{m_n+1}{n t_n^{\frac{\kappa}{1-\kappa}}}\right)\), as \(n\to \infty\), uniformly with respect to \(u,v\) on compact subsets.
\end{lemma}

\begin{proof} 
We split the proof in four parts.
\begin{itemize}
\item[(i)] We first show 
\[\vert y_{t_n,\mu}(x^{\ast})^2 - y_{t_n,\mu_n}(x^{\ast})^2 \vert = \mathcal{O}\left(\frac{\tilde m_n t_n}{n}\right),\quad n\to \infty.\]
To see that, we observe that we have as \(n\to \infty\) (see the proof of Lemma \ref{Lemma4.5})
\[\frac{1}{t_n}= \int \frac{d\mu (s)}{(x^{\ast}-s)^2+y_{t_n,\mu}(x^{\ast})^2}= \int \frac{d\mu (s)}{(x^{\ast}-s)^2+y_{t_n,\mu_n}(x^{\ast})^2}+\mathcal{O}\left(\frac{\tilde m_n}{n t_n^{\frac{2}{1-\k}}}\right), \quad n\to \infty.\]
This gives
\begin{align*}&\int \frac{1}{(x^{\ast}-s)^2+y_{t_n,\mu}(x^{\ast})^2} -\frac{1}{(x^{\ast}-s)^2+y_{t_n,\mu_n}(x^{\ast})^2} d\mu (s)\\
= &\vert y_{t_n,\mu}(x^{\ast})^2 - y_{t_n,\mu_n}(x^{\ast})^2 \vert \int \frac{d\mu (s)}{\left((x^{\ast}-s)^2+y_{t_n,\mu}(x^{\ast})^2\right)\left((x^{\ast}-s)^2+y_{t_n,\mu_n}(x^{\ast})^2\right)}\\
=& \mathcal{O}\left(\frac{\tilde m_n}{n t_n^{\frac{2}{1-\k}}}\right), \quad n\to \infty.
\end{align*}
Using Lemma \ref{Lemma4.5}, by an elementary estimation we have for \(n\) large enough and a positive constant \(c_1\)
\[\int \frac{d\mu (s)}{\left((x^{\ast}-s)^2+y_{t_n,\mu}(x^{\ast})^2\right)\left((x^{\ast}-s)^2+y_{t_n,\mu_n}(x^{\ast})^2\right)} \geq c_1 \frac{1}{t_n t_n^{\frac{2}{1-\k}}},\]
which gives 
\[\vert y_{t_n,\mu}(x^{\ast})^2 - y_{t_n,\mu_n}(x^{\ast})^2 \vert = \mathcal{O}\left(\frac{\tilde m_n t_n}{n}\right),\quad n\to \infty.\]

\item[(ii)] Next we show 
\[H_{t_n,\mu}(x^{\ast}+i y_{t_n,\mu}(x^{\ast}))+ \frac{u}{nc_{t_n}} - H_{t_n,\mu_n}(x^{\ast}+i y_{t_n,\mu_n}(x^{\ast})) =\mathcal{O}\left(\frac{\tilde m_n}{nt_n^{\frac{\kappa}{1-\kappa}}}\right),\quad n\to \infty. \]
To this end, we write
\begin{align*} &H_{t_n,\mu}(x^{\ast}+i y_{t_n,\mu}(x^{\ast}))+ \frac{u}{nc_{t_n}} - H_{t_n,\mu_n}(x^{\ast}+i y_{t_n,\mu_n}(x^{\ast})) \\
&= \frac{u}{nc_{t_n}} + t_n \Re \left\{G_{\mu} (x^{\ast}+y_{t_n,\mu}(x^{\ast}))-G_{\mu_n} (x^{\ast}+y_{t_n,\mu_n}(x^{\ast}))\right\}\\
&= \frac{u}{nc_{t_n}} + t_n\Re \left\{G_{\mu} (x^{\ast}+y_{t_n,\mu_n}(x^{\ast}))-G_{\mu_n} (x^{\ast}+y_{t_n,\mu_n}(x^{\ast}))\right\}\\
&\quad + t_n\Re \left\{G_{\mu} (x^{\ast}+y_{t_n,\mu}(x^{\ast}))-G_{\mu} (x^{\ast}+y_{t_n,\mu_n}(x^{\ast}))\right\}.
\end{align*}
Now, by Lemmas \ref{lemma:comparisonG} and \ref{Lemma4.5} we have 
\[t_n\Re \left\{G_{\mu} (x^{\ast}+y_{t_n,\mu_n}(x^{\ast}))-G_{\mu_n} (x^{\ast}+y_{t_n,\mu_n}(x^{\ast}))\right\}=\mathcal{O}\left(\frac{\tilde m_n}{nt_n^{\frac{\kappa}{1-\kappa}}}\right), \quad n\to \infty.\]
Moreover, we obtain
\begin{align*}&t_n\Re \left\{G_{\mu} (x^{\ast}+y_{t_n,\mu}(x^{\ast}))-G_{\mu} (x^{\ast}+y_{t_n,\mu_n}(x^{\ast}))\right\}\\
&= t_n\left(y_{t_n,\mu}(x^{\ast})^2 - y_{t_n,\mu_n}(x^{\ast})^2\right) \int \frac{(x^{\ast}-s) d\mu (s)}{\left((x^{\ast}-s)^2+y_{t_n,\mu}(x^{\ast})^2\right)\left((x^{\ast}-s)^2+y_{t_n,\mu_n}(x^{\ast})^2\right)},
\end{align*}
and the last integral can be seen using straightforward estimates to be $\mathcal O\left(\frac{1}{t_n t_n^{\frac{1}{1-\kappa}}}\right)$, so that it tends to $0$ as $n\to\infty$. Thus, using (i) we obtain (ii).

\item[(iii)] By the proof of Lemma \ref{Lemma4.4} we have for some constant \(c_2>0\)
\[\Re H_{t_n,\mu_n}'(x+iy_{t_n,\mu_n}(x)) \geq c_2,\quad \quad \vert \Im H_{t_n,\mu_n}'(x+iy_{t_n,\mu_n}(x)) \vert \leq 2,\]
for \(x-x^{\ast}= \mathcal{O}(t_n^{\frac{1}{1-\k}}),\) as \(n \to \infty\). This gives for some constant \(c_3>0\)
\[\Re \left\{\frac{d}{dx} H_{t_n,\mu_n}(x+iy_{t_n,\mu_n}(x)) \right\} \geq c_3 \]
for \(n\) large enough, if \(x-x^{\ast}= \mathcal{O}(t_n^{\frac{1}{1-\k}}),\) as \(n \to \infty\).

\item[(iv)] Now, we choose the constant \(K>0\) such that for all \(u\) in a given compact set we have 

\[\vert H_{t_n,\mu}(x^{\ast}+i y_{t_n,\mu}(x^{\ast}))+ \frac{u}{nc_{t_n}} - H_{t_n,\mu_n}(x^{\ast}+i y_{t_n,\mu_n}(x^{\ast})) \vert \leq \frac{K \tilde m_n}{n t_n^{\frac{\kappa}{1-\kappa}}}.\]

If we define \(\delta_n = \frac{K \tilde m_n+1}{c_3 n t_n^{\frac{\kappa}{1-\kappa}}}\), then using (iii) we know that by \(x\mapsto H_{t_n,\mu_n}(x+i y_{t_n,\mu_n}(x))\) the interval \((x^{\ast}-\delta_n , x^{\ast}+\delta_n)\) is mapped bijectively onto an interval containing 
\[ \left[H_{t_n,\mu_n}(x^{\ast}+i y_{t_n,\mu_n}(x^{\ast}))- \frac{K \tilde m_n}{nt_n^{\frac{\kappa}{1-\kappa}}}, H_{t_n,\mu_n}(x^{\ast}+i y_{t_n,\mu_n}(x^{\ast}))+ \frac{K \tilde m_n}{nt_n^{\frac{\kappa}{1-\kappa}}} \right].\]
As this interval contains the point \(H_{t_n,\mu}(x^{\ast}+i y_{t_n,\mu}(x^{\ast}))+ \frac{u}{nc_{t_n}}\), we can conclude that for its preimage we have
\[\vert x_n - x^{\ast}\vert \leq \delta_n = \mathcal{O}\left(\frac{\tilde m_n}{nc_{t_n}}\right),\]
as \(n\to\infty\) uniformly with respect to \(u\).
From this and Lemma \ref{Lemma4.4} we additionally obtain
\[\vert y_{t_n,\mu_n}(x_n)-y_{t_n,\mu_n}(x^{\ast})\vert =\mathcal{O}\left(\frac{\tilde m_n}{nc_{t_n}}\right), \quad n\to\infty,\]
which gives us
\[z_n - (x^{\ast}+y_{t_n\mu_n}(x^{\ast}))= \mathcal{O}\left(\frac{\tilde m_n}{nc_{t_n}}\right), \quad n\to\infty.\]
From this it easily follows that
\[z_n-w_n=\mathcal{O}\left(\frac{\tilde m_n}{nc_{t_n}}\right)=\mathcal{O}\left(\frac{\tilde m_n}{nt_n^{\frac{\kappa}{1-\kappa}}}\right), \quad n\to\infty,\]
uniformly with respect to \(u\) and \(v\) in compact subsets.
\end{itemize}
\end{proof}

\begin{lemma}\label{Lemma4.7} We have uniformly in \(u\) on compact subsets
\[y_{t_n,\mu_n}(x_n)\sim \left(\frac{C\pi t_n}{\sin\left(\pi \frac{1+\k}{2}\right)}\right)^{\frac{1}{1-\k}}, \quad n\to\infty.\]
\end{lemma}
\begin{proof} From (iv) in the proof of Lemma \ref{Lemma4.6} we know that
\[\vert y_{t_n,\mu_n}^2(x_n)-y_{t_n,\mu_n}^2(x^{\ast})\vert =\mathcal{O}\left(\frac{\tilde m_n t_n}{n}\right), \quad n\to\infty.\]
Using Lemma \ref{Lemma4.5} and recalling that we have \(c_{t_n}=\frac{y_{t_n,\mu}(x^{\ast})}{\pi t_n}\), we obtain
\[\vert y_{t_n,\mu_n}(x_n)-y_{t_n,\mu_n}(x^{\ast})\vert =\mathcal{O}\left(\frac{\tilde m_n t_n}{n t_{n}^{\frac{1}{1-\k}}}\right), \quad n\to\infty.\]
Dividing this by \(y_{t_n,\mu_n}(x^{\ast})\) und using Lemma \ref{Lemma4.5} we obtain
\[\frac{y_{t_n,\mu_n}(x_n)}{y_{t_n,\mu_n}(x^{\ast})} = 1+ o(1),\]
as \(n\to\infty\), uniformly in \(u\) on compact subsets, from which the statement follows.

\end{proof}

\begin{lemma}\label{Lemma4.8} We have uniformly with respect to \(u\) and \(v\) on compact subsets
\[\lim_{n\to\infty} \phi_n(z_n,u) - \phi_n(w_n,v) -\frac{v-u}{t_n c_{t_n}} \left(z_n - x_{t_n}^{\ast}\right) = 0.\]
\end{lemma}

\begin{proof}
We first recall that by definition the phase function is given by
\[\phi_n(z,u)=\frac{n}{2t_n}\left[\left(z-x_{t_n}^{\ast}-\frac{u}{nc_{t_n}}\right)^2 +2t_n \int \log (z-s) d\mu_n(s)\right].\]
Hence, we obtain

\begin{align*} & \phi_n(z_n,u) - \phi_n(w_n,v) =  \phi_n(z_n,u) - \phi_n(z_n,v) +  \phi_n(z_n,v) - \phi_n(w_n,v)\\
&=\frac{n}{2t_n}\left[\left(z_n-x_{t_n}^{\ast}-\frac{u}{nc_{t_n}}\right)^2-\left(z_n-x_{t_n}^{\ast}-\frac{v}{nc_{t_n}}\right)^2\right]+\int_{w_n}^{z_n}\phi_n'(s,v) ds\\
&=\frac{v-u}{t_n c_{t_n}} \left(z_n - x_{t_n}^{\ast}\right) - \frac{(v-u)(u+v)}{2t_nc_{t_n}^2 n}+\frac{n}{t_n}\int_{w_n}^{z_n}\left(H_{t_n,\mu_n}(s)-H_{t_n,\mu_n}(w_n)\right) ds,
\end{align*}
where the integral is performed over the part of the graph of \(y_{t_n,\mu_n}\) from \(w_n\) to \(z_n\).
Now we have 
\[\frac{(v-u)(u+v)}{2t_nc_{t_n}^2 n} \to 0,\quad n\to\infty,\]
uniformly in \(u\) and \(v\), and 

\begin{align*} &\left\vert \frac{n}{t_n}\int_{w_n}^{z_n}\left(H_{t_n,\mu_n}(s)-H_{t_n,\mu_n}(w_n)\right) ds \right\vert= \left\vert\frac{n}{t_n}\int_{w_n}^{z_n}\int_{w_n}^s H_{t_n,\mu_n}'(x) dx ds\right\vert,
\end{align*}
which by Lemma \ref{Lemma4.4} can be bounded above, so that we obtain
\[\left\vert \frac{n}{t_n}\int_{w_n}^{z_n}\left(H_{t_n,\mu_n}(s)-H_{t_n,\mu_n}(w_n)\right) ds \right\vert = \mathcal{O}\left(\frac{n}{t_n} \left(z_n - w_n\right)^2\max \vert H_{t_n,\mu_n}'(z)\vert\right),\]
as \(n\to\infty\), where the maximum is taken over the part of the graph of \(y_{t_n,\mu_n}\) from \(w_n\) to \(z_n\). By the boundedness of \(H_{t_n,\mu_n}'(z)\) and using Lemma \ref{Lemma4.7} this can be bounded further, and obtain
\[\left\vert \frac{n}{t_n}\int_{w_n}^{z_n}\left(H_{t_n,\mu_n}(s)-H_{t_n,\mu_n}(w_n)\right) ds \right\vert =\mathcal{O}\left(\frac{1}{t_n c_{t_n}^2 n}\right)= o(1),\quad n\to\infty, \]
uniformly in \(u\) and \(v\) on compact sets, from which the statement follows.
\end{proof}

\subsection{Asymptotics of \texorpdfstring{$A_n$}{An}}

In order to prove that
\[\lim_{n\to\infty}A_{n}(u,v) = \frac{\sin\left(\pi (u-v)\right)}{\pi (u-v)}\]
uniformly with respect to \(u\) and \(v\) on compact subsets, we have to prove
\[\lim_{n\to\infty}\frac{s}{t_n c_{t_n}}= \pi,\]
with \(s=y_{t_n,\mu_n}(x_n)\). However, recalling that \(c_{t_n} = \frac{y_{t_n,\mu}(x^{\ast})}{\pi t_n}\), this immediately follows by a combination of Lemma \ref{Lemma4.5} and Lemma \ref{Lemma4.7}.

\subsection{Estimation of \texorpdfstring{$I_n$}{In}}

Using the notations and assumptions in the statement of Theorem 1.3 we recall that the integral \(I_n\) is given by

\[I_n(u,v)= \frac{1}{t_n c_{t_n}}e^{-\frac{u^2-v^2}{2nt_n c_{t_n}^2}+\frac{u-v}{t_n c_{t_n}}\left(x_n-x_{t_n}^{\ast}\right)} \frac{1}{(2\pi i)^2} \int_{x_n-i\infty}^{x_n+i\infty} dz \int_{\gamma_n} dw \frac{e^{\phi_n(z,u)-\phi_n(w,v)}}{z-w}.\]

As done previously, we split up the integral into seven parts, of which we explicitly have to deal with 
\[\tilde{I}_n^{(1)}=\int_{z_n+iL_n}^{z_n+i\infty} dz \int_{\gamma_n} dw \frac{e^{\phi_n(z,u)-\phi_n(w,v)}}{z-w},\]
\[\tilde{I}_n^{(4)}=\int_{z_n-L_n}^{z_n+iL_n} dz \int_{\gamma_n\cap U_{L_n}(w_n)} dw \frac{e^{\phi_n(z,u)-\phi_n(w,v)}}{z-w},\]
\[\tilde{I}_n^{(5)}=\int_{z_n-L_n}^{z_n+iL_n} dz \int_{\gamma_n \backslash U_{L_n}(w_n)} dw \frac{e^{\phi_n(z,u)-\phi_n(w,v)}}{z-w},\]
where \(L_n\) is defined by
\[L_n = \sqrt{\frac{t_n}{n}} \left(\log n\right)^{\frac{1+\rho}{2}}\]
for some \(\rho >0\). Using the assumption \(\frac{n t_n^{\frac{1+\k}{1-\k}}}{\left(\log n\right)^{1+\rho}} \to \infty\), as \(n\to\infty\), we have
\[\lim_{n\to\infty} \frac{L_n}{t_n c_{t_n}}=\lim_{n\to\infty} \frac{L_n}{t_n^{\frac{1}{1-\k}}}=0.\]
Moreover, we observe (see the proof of Lemma \ref{Lemma4.4}) that the sequence
\[\frac{t_n}{n}\beta_n = \frac{t_n}{n}\Re{\phi_n''(z_n,u)}= 2t_ny_{t_n,\mu_n}(x_n)^2 \int \frac{d\mu_n (s)}{\vert z_n-s\vert^4}\]
is bounded in absolute value and stays away from zero as \(n\to\infty\).
Following the strategy of the proofs of Theorem 1.1 and Theorem 1.2 we will have to show
\[\lim_{n\to\infty} \frac{1}{t_n c_{t_n}}e^{\frac{u-v}{t_n c_{t_n}}\left(x_n-x_{t_n}^{\ast}\right)} \tilde{I}_n^{(j)} =0,\quad j=1,4,5,\]
locally uniformly in \(u\) and \(v\).

\paragraph{\bf Estimation of $\tilde I_n^{(1)}$.} 

First we observe that by Lemma \ref{Lemma4.4} we have for some constant \(K_1 >0\)
\[d_n^{-1}= \frac{1}{\min \vert z-w\vert} \leq \frac{K_1}{L_n},\]
where the minimum is taken on the contour \([z_n+iL_n,z_n+i\infty]\times \gamma_n\).
This gives
\begin{align*} &\left\vert \frac{1}{t_n c_{t_n}}e^{\frac{u-v}{t_n c_{t_n}}\left(x_n-x_{t_n}^{\ast}\right)} \tilde{I}_n^{(1)}\right\vert\\
&\leq \frac{K_1}{t_nc_{t_n}L_n} L(\gamma_n) e^{\frac{u-v}{t_n c_{t_n}}\left(x_n-x_{t_n}^{\ast}\right)} \int_{z_n+iL_n}^{z_n+i\infty} e^{\Re \phi_n(z,u)}\vert dz\vert \int_{\gamma_n} e^{-\Re \phi_n(w,v)}\vert dw\vert.
\end{align*}
The length of \(\gamma_n\) can be estimated by Lemma \ref{lemma:length} and we obtain further, using Lemma \ref{Lemma4.8}, for some constant \(K_2>0\)

\begin{align*} &\left\vert \frac{1}{t_n c_{t_n}}e^{\frac{u-v}{t_n c_{t_n}}\left(x_n-x_{t_n}^{\ast}\right)} \tilde{I}_n^{(1)}\right\vert\\
&\leq \frac{K_2}{t_nc_{t_n}L_n} \sqrt{t_n}n^3 e^{\frac{u-v}{t_n c_{t_n}}\left(x_n-x_{t_n}^{\ast}\right)+\Re\left\{\phi_n(z_n,u)-\phi_n(w_n,v)\right\}} \int_{z_n+iL_n}^{z_n+i\infty} e^{\Re \left\{\phi_n(z,u)-\phi_n(z_n,u)\right\}}\vert dz\vert\\
&\leq \frac{K_2}{t_nc_{t_n}L_n} \sqrt{t_n}n^3 e^{\frac{1}{2} \Re\left\{\phi_n(z_n+iL_n,u)-\phi_n(z_n,u)\right\}}\int_{z_n+iL_n}^{z_n+i\infty} e^{\frac{1}{2}\Re \left\{\phi_n(z,u)-\phi_n(z_n,u)\right\}}\vert dz\vert.
\end{align*}

It follows from elementary considerations that the integral 
\[\int_{z_n+iL_n}^{z_n+i\infty} e^{\frac{1}{2}\Re \left\{\phi_n(z,u)-\phi_n(z_n,u)\right\}}\vert dz\vert\]
remains bounded as \(n\to\infty\), so that it is sufficient to show 
\[\Re\left\{\phi_n(z_n+iL_n,u)-\phi_n(z_n,u)\right\} \leq -K_3 \left(\log n\right)^{1+\rho}\]
for some constant \(K_3>0\) and \(n\) large enough. A complex Taylor expansion for \(\phi_n(z,u)\) around \(z_n\) yields
\[\phi_n(z,u)=\phi_n(z_n,u)+\frac{1}{2}\phi_n^{''}(z_n,u)(z-z_n)^2+R_2(z).\]
We will estimate \(R_2(z)\) for \(\vert z-z_n \vert\leq L_n\) in a similar way than in the estimation of $\tilde I_n^{(1)}$ in the proof of Theorem 1.2. To this end, we proceed as in \eqref{eq:estimatephi3a} to obtain
\[|R_2(z)|\leq \max|\phi_n'''(\xi;u)|  L_n^3\]
where the maximum is over the line segment between $z_n$ and $z$.  Given $\xi$ on this line segment, we let $\gamma$ be the circle of radius $\frac{1}{2}\Im\xi\geq \frac{1}{2}(\Im z_n - L_n)$ around $\xi$, and we again use Cauchy's theorem to estimate the maximum,
\[|\phi_n'''(\xi;u)|=\left|\frac{1}{2\pi}\int_{\gamma} \frac{\phi_n''(\zeta;u)}{(\zeta-\xi)^2}d\xi\right|\leq \frac{2}{\Im\xi} \max_{\zeta\in\gamma}|\phi_n''(\zeta;u)|\leq \frac{2n}{t_n(\Im z_n-L_n)} \max_{\zeta\in\gamma}|1+t_nG_{\mu_n}'(\zeta)|.\]
Now, for any \(\xi\) on the line segment between $z_n$ and $z$, we have for \(\zeta\in\gamma\) that 
\[\vert \Re\zeta -\Re z_n \vert \leq L_n+ \frac{3}{2}(\Im z_n +L_n),\]
so that by Lemma \ref{Lemma4.6} we have 
\[\vert \Re\zeta -x^{\ast} \vert \leq K_4 t_{n}^{\frac{1}{1-\k}},\]
where the constant \(K_4>0\) can be chosen uniformly in \(\zeta\) and \(\xi\) (and \(u\)). Hence, by Lemma \ref{Lemma4.3} there is a constant \(K_5>0\) such that for large \(n\) we have
\[y_{t_n,\mu_n}\left(\Re \zeta\right)\leq K_5 t_{n}^{\frac{1}{1-\k}} \]
uniformly in \(\zeta, \xi\) and \(u\). But then we find a small constant \(\hat{c}>0\) such that we have for large \(n\) uniformly
\[y_{\hat{c}t_n,\mu_n}\left(\Re \zeta\right)\leq \Im \zeta.\]
This gives 
\[\vert G_{\mu_n}'(\zeta) \vert  \leq \frac{1}{\hat{c}t_n},\]
for large \(n\) uniformly. This yields for \( \xi \in U_{L_n}(z_n)\)
\[|\phi_n'''(\xi;u)| \leq \frac{K_6 n}{t_n^{1+\frac{1}{1-\k}}}\]
for a constant \(K_6>0\), so that for \(\vert z-z_n \vert\leq L_n\) we obtain
\[|R_2(z)|\leq \frac{K_6 n}{t_n^{1+\frac{1}{1-\k}}}L_n^3 = o\left((\log n)^{1+\rho}\right),\]
as \(n\to\infty\). From this we obtain 
\[\Re\left\{\phi_n(z_n+iL_n,u)-\phi_n(z_n,u)\right\} \leq -K_3 \left(\log n\right)^{1+\rho}\]
for some constant \(K_3>0\) and \(n\) large enough, which is sufficient to show
\[\lim_{n\to\infty}\frac{1}{t_n c_{t_n}}e^{\frac{u-v}{t_n c_{t_n}}\left(x_n-x_{t_n}^{\ast}\right)} \tilde{I}_n^{(1)}=0,\]
uniformly in \(u\) and \(v\) on compact subsets.

\paragraph{\bf Estimation of $\tilde I_n^{(4)}$.}

From Lemma \ref{Lemma4.4} we know that the length of \(\gamma_n\cap U_{L_n}(w_n)\) is \(\mathcal{O} \left(L_n\right)\) as \(n\to\infty\), and that the angle between \(\gamma_n\) and the vertical line \(z_n +i \mathbb{R}\) cannot become small for large \(n\). Hence, we can proceed as in the estimation of $\tilde I_n^{(4)}$ in the proofs of Theorems 1.1 and 1.2, which means that we obtain

\[\left\vert \frac{1}{t_n c_{t_n}}e^{\frac{u-v}{t_n c_{t_n}}\left(x_n-x_{t_n}^{\ast}\right)} \tilde{I}_n^{(4)}\right\vert\leq \frac{L_n}{t_n c_{t_n}}e^{\frac{u-v}{t_n c_{t_n}}\left(x_n-x_{t_n}^{\ast}\right)+\Re\left\{\phi_n(z_n,u)-\phi_n(w_n,v)\right\}}. \]
Using Lemma \ref{Lemma4.8} the expression on the right-hand side vanishes as \(n\to\infty\) locally uniformly in \(u\) and \(v\).

\paragraph{\bf Estimation of $\tilde I_n^{(5)}$.}

By Lemma \ref{Lemma4.6} we have \(z_n -w_n = o \left(L_n\right)\) as \(n\to\infty\), locally uniformly in \(u\) and \(v\), which means that we again obtain
\[\tilde{d_n}^{-1}= \frac{1}{\min \vert z-w\vert} \leq \frac{K_1}{L_n}\]
for some constant \(K_1>0\), where the minimum is taken over the contour \([z_n-iL_n, z_n+iL_n]\times \left(\gamma_n\backslash U_{L_n}(w_n)\right)\).
Hence, we obtain
\begin{align*}&\left\vert \frac{1}{t_n c_{t_n}}e^{\frac{u-v}{t_n c_{t_n}}\left(x_n-x_{t_n}^{\ast}\right)} \tilde{I}_n^{(5)}\right\vert\\
&\leq \frac{K_1}{t_n c_{t_n} L_n} e^{\frac{u-v}{t_n c_{t_n}}\left(x_n-x_{t_n}^{\ast}\right)+\Re\left\{\phi_n(z_n,u)-\phi_n(w_n,v)\right\}} \\
&\times \int_{z_n-iL_n}^{z_n+iL_n}e^{\Re\left\{\phi_n (z,u)-\phi_n(z_n,u)\right\}} \vert dz\vert \int_{\gamma_n\backslash U_{L_n}(w_n)} e^{\Re\left\{\phi_n (w_n,v)-\phi_n(w,v)\right\}} \vert dz\vert,
\end{align*}
which, by Lemma \ref{Lemma4.8} can be estimated further to
\[\left\vert \frac{1}{t_n c_{t_n}}e^{\frac{u-v}{t_n c_{t_n}}\left(x_n-x_{t_n}^{\ast}\right)} \tilde{I}_n^{(5)}\right\vert\leq \frac{K_2}{t_n c_{t_n}} \int_{\gamma_n\backslash U_{L_n}(w_n)} e^{\Re\left\{\phi_n (w_n,v)-\phi_n(w,v)\right\}} \vert dz\vert\]
for some constant \(K_2>0\). Now defining \(\{w_n^{+},w_n^{-}\}=\gamma_n\cap \partial U_{L_n}(w_n)\) the last integral can be estimated above by
\[L(\gamma_n) \left\{e^{\Re\left\{\phi_n (w_n,v)-\phi_n(w_n^{+},v)\right\}}+e^{\Re\left\{\phi_n (w_n,v)-\phi_n(w_n^{-},v)\right\}}\right\}.\]
The length of \(\gamma_n\) can be estimated by Lemma \ref{lemma:length}, whereas the exponential terms can be bounded the same way as in the estimation of $\tilde I_n^{(1)}$ above. This finally leads to
\[\left\vert \frac{1}{t_n c_{t_n}}e^{\frac{u-v}{t_n c_{t_n}}\left(x_n-x_{t_n}^{\ast}\right)} \tilde{I}_n^{(5)}\right\vert=\mathcal{O}\left(n^{7/2} e^{-K_3\left(\log n\right)^{1+\rho}}\right),\]
as \(n\to\infty\), uniformly in \(u\) and \(v\) on compact subsets, for some constant \(K_3>0\). As previously seen in the proofs of Theorems 1.1 and 1.2, this completes the proof of Theorem 1.3.

\section{Proof of Theorem \ref{thrm_counterexample}}\label{section_thrm_counterexample}
We start with some elementary lemmas.
\begin{lemma}\label{lemma:xtau}Let $\mu$ be a probability measure  on the real line, $x^*\in\mathbb R$, $t>0$, and let $x_t^*$ be 
defined by \eqref{def:xtau}, 
	Then, \begin{equation}
	|x_t^*-x^*|\leq \sqrt{t}.
	\end{equation}
\end{lemma}
\begin{proof}
	If $z\in \overline{\Omega_{t,\mu}}$, it follows from \eqref{def:y_t} and the Cauchy-Schwarz inequality that the Stieltjes 
transform $G_{\mu}(z)$ satisfies the bound
	\begin{equation}\label{eq:estimateG}
	|G_{\mu}(z)|\leq \int\frac{d\mu(s)}{|z-s|}\leq\left(\int\frac{d\mu(s)}{|z-s|^2}\right)^{1/2}\leq \frac{1}{\sqrt t}.
	\end{equation}
	It follows from \eqref{def:H} that
	\begin{equation}\label{eq:estimateH}\sup_{z\in \overline{\Omega_{t,\mu}}} \left|H_{t,\mu}(z)-z \right|\leq \sqrt{t}.
	\end{equation}
	Applying this to $z=x^*+iy_{t,\mu}(x^*)$ and using \eqref{def:xtau}, we get 
	\[|H_{t,\mu}(x^*+iy_{t,\mu}(x^*))-x^*-iy_{t,\mu}(x^*)|\leq \sqrt{t},\]
	and the result now follows easily.
\end{proof}

\begin{lemma}\label{lemma:zn}
	Let $\mu,\nu$ be probability measures on the real line, let $x^*\in\mathbb R$, and let $x_t^*$ be defined by 
\eqref{def:xtau}. 
	We have for any $\epsilon\in\mathbb R, t>0$ that \begin{equation}
	|\Re \left\{F_{t,\nu}\left(x_{t}^*+\epsilon\right)\right\}-x^*|\leq 2\sqrt{t}+|\epsilon|.
	\end{equation}
\end{lemma}
\begin{proof}
	It follows from \eqref{eq:estimateH} applied to $\nu$ and $z=F_{t,\nu}(x_t^*+\epsilon)$ that
	\[\left|F_{t,\nu}\left(x_t^*+\epsilon\right) -x_t^*-\epsilon\right|\leq \sqrt{t}.\]
	By Lemma \ref{lemma:xtau}, it follows that
	\[\left|F_{t,\nu}\left(x_{t}^*+\epsilon\right)-x^{\ast}\right|\leq 2\sqrt{t}+|\epsilon|.\]
\end{proof}

\begin{lemma}\label{lemma:znwn}
	Let $x^*\in\mathbb R$, let $\delta_n, \epsilon_n, t_n$ be sequences of positive numbers converging to $0$, as $n\to\infty$, and let 
$\mu_n$ be a sequence of probability measures. 
	Define, for $u,v\in\mathbb R$,
	\begin{equation}\label{eq: def znwn}
	z_n=F_{t_n,\mu_n}\left(x_{t_n}^*+\epsilon_n u\right)
	,\qquad w_n=
	F_{t_n,\mu_n}\left(x_{t_n}^*+\epsilon_n v\right).
	\end{equation}
	If $[x^*-\delta_n, x^*+\delta_n]$ does not intersect with the support of $\mu_n$, if $\epsilon_n=o(\delta_n)$, and if $t_n=o(\delta_n^2)$ as $n\to\infty$, we 
have
	\begin{equation}\label{eq:estimateznwn}
	\Im z_n=\Im w_n=0,\qquad |z_n-w_n|\leq 2\epsilon_n|u-v|
	\end{equation}
	for $n$ sufficiently large, locally uniformly in \(u,v\).
\end{lemma}
\begin{proof}Because of the conditions imposed on the sequences $\delta_n, \epsilon_n, t_n$, we have by Lemma \ref{lemma:zn} that $|\Re 
z_n-x^*|=o\left(\delta_n\right)$ as $n\to\infty$, and it follows that $[\Re z_n-\delta_n/2, \Re z_n+\delta_n/2]$ does not 
intersect with the support of $\mu_n$ for $n$ sufficiently large. Similar estimates hold for $w_n$.
	Furthermore, if $x\in\mathbb R$ is such that ${\rm dist}(x,{\rm supp}({\mu}_n))\geq \sqrt{t_n},$ 
	we have
	\[y_{t_n,{\mu}_n}(x)=0,\]
	and this implies that $z_n, w_n\in\mathbb R$.
	Next, it is easy to see that
	\begin{equation}\left|G_{\mu_n}'(z_n)\right|\leq \frac{4}{\delta_n^2},\qquad \left|G_{\mu_n}'(w_n)\right|\leq 
\frac{4}{\delta_n^2}\label{eq: estimates Gn}\end{equation}
	for $n$ sufficiently large. We also have
	\[|G_{\mu_n}(z_n)-G_{\mu_n}(w_n)|=|z_n-w_n|\left\vert \int \frac{d\mu_n(s)}{(z_n-s)(w_n-s)}\right\vert.\]
	The latter can be estimated using the Cauchy-Schwarz inequality by
	\[\leq |z_n-w_n| \sqrt{\int \frac{d\mu_n(s)}{(z_n-s)^2} \int \frac{d\mu_n(s)}{(w_n-s)^2}}=  
|z_n-w_n|\sqrt{G_{\mu_n}'(z_n) G_{\mu_n}'(w_n)},  \]
	giving
	\begin{equation}\label{nosine_eq2}
	|G_{{\mu}_n}(z_n)-G_{{\mu}_n}(w_n)|\leq 4|z_n-w_n| \delta_n^{-2},
	\end{equation}
	for $n$ sufficiently large.
	By the definition of \(z_n\) and \(w_n\) we have
	\[z_n-x_{t_n}^{\ast}-\epsilon_n u+t_n G_{{\mu}_n}(z_n)=0, \]
	\[w_n-x_{t_n}^{\ast}-\epsilon_n v+t_n G_{{\mu}_n}(w_n)=0,\]
	so that 
	\[|z_n-w_n|\leq \epsilon_n|u-v|+t_n |G_{{\mu}_n}(z_n)-G_{{\mu}_n}(w_n)|\leq \epsilon_n|u-v|+4t_n \delta_n^{-2} 
|z_n-w_n|, \]
	which implies \eqref{eq:estimateznwn} for $n$ large, since $t_n\delta_n^{-2}\to 0$ as $n\to\infty$.
\end{proof}

\vspace{1cm}

For the correlation kernel $K_{n,t}$, recall \eqref{eq:Kncontour3}, which gives
\[\epsilon_n K_{n,t_n}\left(x_{t_n}^{\ast}+\epsilon_n u, x_{t_n}^{\ast}+\epsilon_n v\right)= A_n(u,v)+ I_n(u,v),\]
with
\[I_n(u,v)=e^{-\frac{n\epsilon_n^2}{2t_n}(u^2-v^2)+\frac{n\epsilon_n}{t_n}(u-v)(\Re z_n-x^*)}\frac{n\epsilon_n}{(2\pi i)^2 t_n}\int\limits_{\Re z_n - i\infty}^{\Re z_n + i\infty} dz \int\limits_{\gamma_n} dw\, 
e^{\phi_n(z,u)-\phi_n(w,v)}\frac{1}{z-w},\]
where 
\[\phi_n(z,u)=\frac{n}{2 t_n}\left\{(z-x_{t_n}^{\ast} -\epsilon_n u)^2+2t_n g_{\mu_n}(z)\right\},\]
\(\gamma_n\) consists of the graph \(y_{t_n,{\mu}_n}\) and its complex conjugate (positively oriented), and 
\[A_n(u,v)=\frac{1}{\pi (u-v)}\sin\left(\frac{n\epsilon_n s}{t_n} (u-v)\right),\]
where \(s=\Im\left\{F_{t_n, {\mu}_n}\left(x_{t_n}^{\ast}+\epsilon_n u\right)\right\}=\Im z_n\) and \(x_0=\Re z_n\).
Under the conditions of Lemma \ref{lemma:znwn}, $z_n$ is real for large $n$ so that\[A_n(u,v)=0.\]

Similarly as in the proof of Lemma \ref{Lemma4.8}, we have
\[\phi_n(z,u)-\phi_n(z,v)=\frac{n\epsilon_n^2}{2t_n}(u^2-v^2)-\frac{n\epsilon_n}{t_n}(u-v)(z-x^*),\]
and it follows that we can write $I_n(u,v)$ as 
\[I_n(u,v)=\frac{n\epsilon_n}{(2\pi i)^2 t_n}\int\limits_{z_n - i\infty}^{z_n + i\infty} dz \int\limits_{\gamma_n} dw\, 
e^{\phi_n(z,v)-\phi_n(w,v) -\frac{n\epsilon_n}{t_n}(u-v)(z-z_n) }\frac{1}{z-w}.\]
Thus, in order to see that 
\[\lim_{n\rightarrow \infty}\epsilon_n K_{n,t_n}\left(x_{t_n}^{\ast}+\epsilon_n u, x_{t_n}^{\ast}+\epsilon_n v\right)=0,\]
it is sufficient to show
\[\lim_{n\rightarrow \infty} I_n(u,v)=0.\]
The critical points of $\phi_n(z,v)$ and $\phi_n(w,v)$ are defined by the equations
\[H_{t_n,\mu_n}(z)=x_{t_n}^*+\epsilon_n v,\qquad H_{t_n,\mu_n}(w)=x_{t_n}^*+\epsilon_n v,\] which means that they are both precisely 
the real point $w_n$ defined in \eqref{eq: def znwn}. If $\epsilon_n=o(\delta_n)$  and $t_n=o(\delta_n^2)$ as $n\to\infty$, it 
follows from Lemma \ref{lemma:zn} that $w_n$ lies in an interval $[x^*-\delta_n/4, x^*+\delta_n/4]$ for $n$ sufficiently large.

It is crucial to observe that the parts of $\gamma_n$ lying on the real line do not contribute to the integral $I_n(u,v)$, as they cancel out against their complex conjugate.
As a consequence of this observation and by symmetry with respect to complex conjugation, in order to show that the large \(n\) limit of \(I_n(u,v)\) is zero, it is sufficient to show that
\begin{equation}\label{nosine_eq5}
\lim_{n\to\infty}\frac{n\epsilon_n}{t_n} \left\vert \int\limits_{z_n}^{z_n+i\infty} dz \int\limits_{\gamma_n^{\pm}} dw\, 
e^{\phi_n(z,v)-\phi_n(w,v) -\frac{n\epsilon_n}{t_n}(u-v)(z-z_n)} \frac{1}{z-w}\right\vert =0,
\end{equation}
where \(\gamma_n^{\pm}\) denotes the part of \(\gamma_n\) which lies to the right ($+$) or left ($-$) of $w_n$ and which is not 
located on the real axis.  For \(w\in\gamma_n^{\pm}\) we have \(|\Re w-\Re w_n| \geq \delta_n/2\) for $n$ sufficiently large by 
construction.

Using the inequality \(|z-w|\geq \delta_n/2\) for \((z,w)\in(z_n+i\mathbb R)\times\gamma_n^{\pm}\), we have
\begin{equation}\label{estimateI123}\frac{n\epsilon_n}{t_n}\left\vert \int\limits_{z_n}^{z_n+i\infty} dz \int\limits_{\gamma_n^{\pm}} dw\, 
e^{\phi_n(z,v)-\phi_n(w,v) -\frac{n\epsilon_n}{t_n}(u-v)(z-z_n)} \frac{1}{z-w}\right\vert \leq \frac{2n\epsilon_n}{t_n\delta_n} I_n^{(1)} \times I_n^{(2)},
\end{equation}
with
\begin{align}
&I_n^{(1)}=\int\limits_{z_n}^{z_n+i\infty} e^{\Re \left(\phi_n (z,v)-\phi_n(z_n,v)\right)} \vert dz\vert,\\
&I_n^{(2)}= \int\limits_{\gamma_n^{\pm}} e^{\Re(\phi_n(w_n,v)-\phi_n (w,v))}  \vert dw \vert.
\end{align}
We set 
	\begin{equation}
	L_n=\frac{\sqrt{t_n}}{\log n}
	\end{equation} and we will now estimate these integrals.

\begin{lemma}\label{lemma: I3}
	Let us assume the conditions of Lemma \ref{lemma:znwn}, and additionally let \(n\delta_n^2 \to +\infty\), as \(n\to\infty\). Then, locally uniformly in \(v\), 
we have
	\begin{equation}\label{eq: nosine6}I_n^{(2)}\leq |\gamma_n^\pm|e^{-\frac{n}{4t_n}L_n^2}
	\end{equation}
	for $n$ sufficiently large, where $|\gamma_n^\pm|$ denotes the length of the curve $\gamma_n^\pm$.
\end{lemma}
\begin{proof}We give the proof for $\gamma_n^+$, the case of $\gamma_n^-$ is similar.
	We easily
	obtain the bound
	\[I_n^{(2)}\leq 
	e^{\max_{w\in\gamma_n^+}\Re\left( \phi_n(w_n,v)-\phi_n (w,v) \right)} |\gamma_n^+|.\]
	
	Next, we note that for $x\in [w_n-L_n, w_n+L_n]$, ${\rm dist}(x,{\rm supp}({\mu}_n))\geq \sqrt{t_n},$ for $n$ large,
	which means that
	\[y_{t_n,{\mu}_n}(x)=0.\]
	Hence, $[w_n-L_n, w_n+L_n]$ is part of $\gamma_n$, but not of $\gamma_n^+$. Since \(\gamma_n\) is a path of descent 
for \(-\phi_n(w,v)\), we obtain that $-\Re\phi_n(w,v)$ attains its maximum on $\gamma_n^+$ at a real point $w\geq w_n+L_n$, and we have 
	\[e^{\Re\left( \phi_n(w_n,v)-\phi_n (w,v) \right)}\leq e^{\Re\left( \phi_n(w_n,v)-\phi_n (w_n+L_n,v) \right)}\] for 
$w\in\gamma_n^+$.

	It follows that
	\begin{equation}\label{eq:In3estimateintermediate}I_n^{(3)}\leq |\gamma_n^+|
	e^{\Re\left( \phi_n(w_n,v)-\phi_n (w_n+L_n,v) \right)}.\end{equation}

	Moreover, we have
	\[\phi_n''(z,v)=\frac{n}{t_n}+n G_{{\mu}_n}'(z).\]
	By \eqref{eq: estimates Gn}, we have 
	\[\phi_n''(z_n,v)\sim \frac{n}{t_n},\qquad 
	\phi_n''(w_n,v)\sim \frac{n}{t_n},\]
	as \(n\rightarrow \infty\).
	For \(|z-x^*|\leq \delta_n/2\) we have
	\[|g_{{\mu}_n}(z)|\leq\log \frac{2}{\delta_n},\] for $n$ sufficiently large, with the branches of the function 
	\[g_{{\mu}_n}(z)=\int \log (z-s) d{\mu}_n(s)\]
	defined such that \(g_{{\mu}_n}\) is analytic in the disk centered at $x^{\ast}$ with radius \(\delta_n\).
	If \(|w-w_n|\leq L_n\), we have
	\[\phi_n(w,v)= \phi_n(w_n,v)+\frac{1}{2}\phi_{n}''(w_n,v)(w-w_n)^2+R_2(w),\]
	where 
	\[|R_2(w)|\leq \frac{64\max_{|w_n-s|=\delta_n/4}|\phi_n(s,v)|}{\delta_n^3}\frac{|w-w_n|^3}{1-\frac{4|w-w_n|}{\delta_n}}\leq 
\frac{64\max_{|x^*-s|=\delta_n/2}|\phi_n(s,v)||w-w_n|^3}{\delta_n^3(1-\frac{4|w-w_n|}{\delta_n})},\]
	which, by use of Lemma \ref{lemma:xtau}, can be estimated for large $n$ by
	\[\mathcal 
O\left(\frac{n\log\frac{2}{\delta_n}+\frac{n\delta_n^2}{t_n}}{\delta_n^3}L_n^3\right)=o\left(\frac{nL_n^2}{t_n}\right),\qquad 
n\to\infty,\]
where in the last equality we used the assumption \(n\delta_n^2 \to\infty\), as \(n\to\infty\).
	Using this in \eqref{eq:In3estimateintermediate}, we obtain \eqref{eq: nosine6}.
\end{proof}     

\begin{lemma}\label{lemma: I2}
	Let us assume the conditions of Lemma \ref{lemma:znwn}. Then, locally uniformly in \(u,v\), we have
	\begin{equation}
	I_n^{(1)}=\mathcal O\left(\sqrt{t_n}\right),
	\end{equation}
	as \(n\to\infty\).
\end{lemma}
\begin{proof}
	
	We have
	\begin{align*}
	I_n^{(1)}&=\int\limits_{z_n}^{z_n+i\infty} e^{-\Re \left(\phi_n (z_n,v)-\phi_n(z,v)\right)} \vert dz\vert\\
	&=\int\limits_{z_n+}^{z_n+i\infty} e^{-\frac{n}{2t_n}\Re\left\{\left(z_n-x_{t_n}^{\ast}-\epsilon_n v\right)^2-\left(z-x_{t_n}^{\ast}-\epsilon_n v\right)^2+2t_n g_{{\mu}_n}(z_n)-2t_n g_{{\mu}_n}(z) \right\}} \vert 
dz\vert\\
	&\leq\int\limits_{0}^{\infty} e^{-\frac{n}{2t_n}\Re\left\{\left(z_n-x_{t_n}^{\ast}-\epsilon_n v\right)^2 
-\left(z_n+i\xi-x_{t_n}^{\ast}-\epsilon_n v\right)^2\right\}+n \Re\left\{g_{{\mu}_n}(z_n+i\xi)-g_{{\mu}_n}(z_n)\right\}} 
 d\xi\\
	&\leq \int\limits_{0}^{\infty}e^{-\frac{n}{2t_n}\xi^2+n\int\log\left|1+\frac{i\xi}{z_n-s}\right| d\mu_n(s)} d\xi.
	\end{align*}
	Using the fact that $|z_n-s|\geq \sqrt{t_n}$ and
	\[\log\left|1+ix\right|=\frac{1}{2}\log(1+x^2)\leq \frac{x^2}{2},\qquad x\in\mathbb R,\]
	we obtain
	\[I_n^{(1)}\leq \sqrt{t_n}\int_{0}^{\infty}e^{n\left(-\frac{x^2}{2}+\frac{1}{2}\log(1+x^2)\right)}dx\leq 
\sqrt{t_n}\int_{0}^{\infty}e^{\left(-\frac{x^2}{2}+\frac{1}{2}\log(1+x^2)\right)}dx,\]
	and this integral is convergent.
\end{proof}    


Combining \eqref{estimateI123} with Lemma \ref{lemma: I3} and Lemma \ref{lemma: I2}, we obtain
\[ I_n(u,v)=\mathcal O\left(\frac{n\epsilon_n|\gamma_n^{\pm}|}{\sqrt{t_n}\delta_n}e^{-\frac{n}{4t_n}L_n^2}\right),\]     
as $n\to\infty$. Using also Lemma \ref{lemma:length}, the definition of \(L_n\), and the fact that $\epsilon_n=o(\delta_n)$, as $n\to\infty$, we get
\[ I_n(u,v)=\mathcal O\left(\frac{n^4\epsilon_n}{\delta_n}e^{-\frac{n}{4t_n}L_n^2}\right)=\mathcal O\left(n^4e^{-\frac{n}{4\log^2 n}}\right),\] 
as $n\to\infty$, and this yields \eqref{nosine_eq5}.

\appendix
\section{Exact expression for the correlation kernel \texorpdfstring{$K_{n,t}$}{}}\label{appendix}
It is the purpose of this Appendix to derive the explicit representation of the correlation kernel \(K_{n,t}\), on which we based our proofs. Let us first assume that $a_1^{(n)},\ldots, a_n^{(n)}$  are not deterministic but random points following a 
polynomial ensemble of the form
\begin{align}
\frac{1}{Z_n} \Delta_n(a^{(n)}) \det\left[ f_{k-1}\left(a_j^{(n)}\right) \right]_{j,k=1}^n\label{PE}
\end{align}
	for certain functions $f_0, \ldots, f_{n-1}$, where \(Z_n >0\) is a normalization constant, and \(\Delta_n(a^{(n)})\) denotes the Vandermonde determinant
\[\Delta_n(a^{(n)}) = \prod_{j<k} \left(a_k^{(n)}-a_j^{(n)}\right).\]	
	
Then a correlation kernel $\tilde{K}_{n,t}$ for the eigenvalues of \(Y_n(t)\), $y_1,\ldots, y_n$, say, can be 
expressed in the following form, which was obtained in \cite{ClaeysKuijlaarsWang} (up to rescaling, see 
\cite{ClaeysKuijlaarsLiechtyWang} for the scaled version)
	\begin{equation} \label{eq:Kn}
	\tilde{K}_{n,t}(x,y) = \frac{n}{2\pi i t} \int_{x_0- i \infty}^{x_0 + i \infty} dz \int_{-\infty}^{\infty} dw K_n^{\rm PE}(z, w) 
e^{\frac{n}{2t} ((z-x)^2 - (w-y)^2)},
	\end{equation}
	for any choice of $x_0$,
	where $K_n^{\rm PE}$ is the correlation kernel of the polynomial ensemble \eqref{PE}, which takes the form
	\[ K_n^{\PE}(x,y) = \sum_{j=0}^{n-1} p_j(x) q_j(y) \]
	with $p_j$ a polynomial of degree $j$ and with $q_j$ in the
	linear span of $f_0, \ldots, f_{n-1}$, such that
	\[\int_{\mathbb R}p_j(x)q_k(x)dx=\delta_{j,k}.\]
	It should be noted that such a system of $p_k$'s
	and $q_k$'s is not unique in general.
	
	The case of deterministic $a_1^{(n)},\ldots, a_n^{(n)}$ can be seen as a degenerate case of a polynomial ensemble (see Section 5.4 of 
\cite{ClaeysKuijlaarsWang}), with functions $f_k$ replaced by Dirac distributions
	\[f_k(x)=\delta(x-a_k^{(n)}).\]
For pairwise distinct points \(a_k^{(n)}\) the system of $p_k$'s and $q_k$'s can formally be taken as follows,
	\[q_k(x)=\delta(x-a_k^{(n)}),\qquad p_k(x)=\prod_{j\neq k}\frac{x-a_j^{(n)}}{a_k^{(n)}-a_j^{(n)}}.\]
	Note that $p_k$ is the Lagrange interpolating polynomial, vanishing at the points $a_j^{(n)}$, $j\neq k$, and equal to $1$ 
at $a_k^{(n)}$.

	Substituting this in \eqref{eq:Kn}, we obtain an expression for the kernel $\tilde{K}_{n,t}$. In \cite{ClaeysKuijlaarsWang}, only 
genuine non-degenerate polynomial ensembles were considered, so we need to adapt the proof.
	\begin{prop}
		Let \(n\in\mathbb{N}\), \(1\leq k \leq n\) and $t>0$. Then the function 
		\begin{equation} \label{eq:Kn2}
		\tilde{K}_{n,t}(x,y) := \frac{n}{2\pi it}\sum_{k=1}^n\int_{x_0-i\infty}^{x_0+i\infty}dz \prod_{j\neq 
k}\frac{z-a_j^{(n)}}{a_k^{(n)}-a_j^{(n)}}e^{\frac{n}{2t} ((z-x)^2 - (a_k^{(n)}-y)^2)}
		\end{equation}
		satisfies 
		\begin{align*}
\rho_{n,t}^k(x_1,\dots,x_k)=\det\lb \tilde{K}_{n,t}(x_i,x_j)\rb_{1\leq i,j\leq k}, \quad x_1,\ldots,x_k \in \mathbb{R},
\end{align*}
where \(\rho_{n,t}^k\) denotes the \(k\)-th correlation function of \(P_{n,t}\) (cf. \eqref{correlation_functions} and \eqref{determinantal_K_{n,t}}).
	\end{prop}
	
	\begin{proof} Let us first assume that the points \(a_1^{(n)},\ldots,a_n^{(n)}\) are pairwise distinct. Then it is known (see \cite[p.15]{ClaeysKuijlaarsWang} and references therein) that the eigenvalues $y_1,\ldots, 
y_n$ of \(M_n+\sqrt{t}H_n\) have a joint probability density function of the form
		\begin{equation}\label{eq:jpdfA}\frac{1}{C_n} \frac{\Delta_n (y)}{\Delta_n (a^{(n)})}\det \left(f_{k-1}(y_j)\right)_{j,k=1}^n
		\end{equation}
		with a normalization constant \(C_n>0\) and
		\[f_{k}(y)=e^{-\frac{n}{2t}(y-a_{k+1}^{(n)})^2}, ~~~k=0,\ldots,n-1,\]
		where \(a_{1}^{(n)},\ldots, a_{n}^{(n)}\) are the deterministic eigenvalues of \(M_n\).
		We show that a kernel of this determinantal ensemble is given by
		\[\tilde{K}_{n,t}(x,y):=\sum_{j=1}^n \hat{p}_j (x)\hat{q}_j(y),\]
		where we define 
		\[\hat{p}_j(x)=\frac{n}{2\pi i t}\int\limits_{x_0 -i \infty}^{x_0 +i \infty}\ell_{j,n}(z) e^{\frac{n}{2t} (z-x)^2} 
dz,\]
		with 
		\[\ell_{j,n}(z)=\prod_{\nu \neq j}\frac{z-a_\nu^{(n)}}{a_j^{(n)}-a_\nu^{(n)}},\]
		and
		\[\hat{q}_j(y)=f_{j-1}(y)=e^{-\frac{n}{2t}(y-a_{j}^{(n)})^2}.\]
		In order to prove this, we will show that the following biorthogonality relation holds
		\[\int\limits_{-\infty}^{\infty}\hat{p}_j(x)\hat{q}_k(x) dx=\delta_{j,k},~~~j,k\in\{1,\ldots,n\}.\]
		To this end, we make use of the (inverse) Weierstrass transformation of a function \(\varphi\) (see 
\cite[p.16]{ClaeysKuijlaarsWang} and references) given by
		\[\mathcal{W}\varphi (y)=\frac{1}{\sqrt{2\pi}}\int\limits_{-\infty}^\infty \varphi(t) e^{-\frac{1}{2}(t-y)^2}dt,\]
		\[\mathcal{W}^{-1}\Phi (x)=\frac{1}{\sqrt{2\pi}i}\int\limits_{-i\infty}^{i \infty}\Phi(s) 
e^{\frac{1}{2}(s-x)^2}ds.\]
		For a polynomial \(P\) and \(k=1,\ldots,n\) we have
		\begin{equation}\label{WP}\int\limits_{-\infty}^{\infty} P(x) e^{-\frac{n}{2t} 
(x-a_k^{(n)})^2}dx=\sqrt{\frac{t}{n}}\int\limits_{-\infty}^{\infty}P\left(\sqrt{\frac{t}{n}}x\right)e^{-\frac{1}{2}\left(x-\sqrt{\frac{n}{
t}}a_k^{(n)}\right)^2}dx =\sqrt{\frac{2\pi t}{n}} \mathcal{W} \tilde{P} \left(\sqrt{\frac{n}{t}} a_k^{(n)}\right),
		\end{equation}
		where \(\mathcal{W}\) acts on \(\tilde{P}\) defined by
		\[\tilde{P}(x)=P\left(\sqrt{\frac{t}{n}}x\right).\] 
		Moreover, letting \(x_0=0\) and after a linear change of variables in the definition of \(\hat{p}_j\) we obtain
		\begin{equation}\label{Wl} \hat{p}_j (x)=\frac{1}{2\pi i} \sqrt{\frac{n}{t}} \int\limits_{-i \infty}^{ i 
\infty}\ell_{j,n}\left(\sqrt{\frac{t}{n}} s\right) e^{\frac{1}{2} \left(s-\sqrt{\frac{n}{t}}x\right)^2} ds = \sqrt{\frac{n}{2\pi t}} 
\mathcal{W}^{-1}\tilde{\ell}_{j,n}\left(\sqrt{\frac{n}{t}} x\right),
		\end{equation}
		where \(\mathcal{W}^{-1}\) acts on \(\tilde{\ell}_{j,n}\) defined by
		\[\tilde{\ell}_{j,n}(x)=\ell_{j,n}\left(\sqrt{\frac{t}{n}}x\right).\]
		Using \eqref{Wl} and the definition of \(\hat{q}_k\) we get
		\[\int\limits_{-\infty}^{\infty}\hat{p}_j(x)\hat{q}_k(x) dx=\int\limits_{-\infty}^{\infty}\sqrt{\frac{n}{2\pi t}} 
\mathcal{W}^{-1}\tilde{\ell}_{j,n}\left(\sqrt{\frac{n}{t}} x\right) e^{-\frac{n}{2t}(x-a_{k}^{(n)})^2} dx.\]
		Now, applying \eqref{WP} to the polynomial
		\[P(x)=\sqrt{\frac{n}{2\pi t}} \mathcal{W}^{-1}\tilde{\ell}_{j,n}\left(\sqrt{\frac{n}{t}} x\right)\]
		gives
		\[\int\limits_{-\infty}^{\infty}\hat{p}_j(x)\hat{q}_k(x) dx=\sqrt{\frac{2\pi t}{n}} \mathcal{W} \tilde{P} 
\left(\sqrt{\frac{n}{t}} a_k^{(n)}\right)=\ell_{j,n}(a_k^{(n)})=\delta_{j,k},~~j,k\in\{1,\ldots,n\}.\]
Finally, we observe that the joint probability function in \eqref{eq:jpdfA} as well as the kernel \(\tilde{K}_{n,t}\) both depend continuously on the initial points \(a_1^{(n)},\ldots,a_n^{(n)}\), so by a continuity argument the statement follows for arbitrary initial configurations. 
	\end{proof}
	We can also write the sum as a contour integral: if $\gamma$ is a contour encircling all $a_j^{(n)}$'s in the positive sense 
and if $x_0$ is such that $\gamma$ and $x_0+i\mathbb R$ do not intersect, we can write $\tilde K_{n,t}$ as a double  contour integral,
	\begin{equation} \label{eq:Kncontour}
	\tilde K_{n,t}(x,y) = \frac{n}{(2\pi i)^2t}\int_{x_0-i\infty}^{x_0+i\infty}dz\int_\gamma dw 
\frac{1}{z-w}\frac{\prod_{j=1}^n(z-a_j^{(n)})}{\prod_{j=1}^n(w-a_j^{(n)})}\frac{e^{\frac{n}{2t} (z-x)^2}}{e^{\frac{n}{2t} 
(w-y)^2)}},
	\end{equation}
	which follows from a residue calculation.
	If $\gamma$ and $x_0$ are such that $x_0+i\mathbb R$ has exactly two intersection points $\t_1=x_0+i s_1$ and $\t_2=x_0+i s_2$ 
with $\gamma$, with $\Im \t_1>0>\Im \t_2$ and the line segment $[\t_1,\t_2]$ is fully contained in $\gamma$, we have
	\begin{multline} \label{eq:Kncontour2}
	\tilde K_{n,t}(x,y) = \frac{n}{(2\pi i)^2t}\int_{x_0-i\infty}^{x_0+i\infty}dz\int_\gamma dw 
\frac{1}{z-w}\frac{\prod_{j=1}^n(z-a_j^{(n)})}{\prod_{j=1}^n(w-a_j^{(n)})}\frac{e^{\frac{n}{2t} (z-x)^2}}{e^{\frac{n}{2t} 
(w-y)^2)}}\\
	+\frac{n}{2\pi it}\int_{\t_2}^{\t_1}dz\frac{e^{\frac{n}{2t} (z-x)^2}}{e^{\frac{n}{2t} (z-y)^2)}}.
	\end{multline}
	If we take $\t_1$ and $\t_2$ as complex conjugates, i.e., $s_1=-s_2=s>0$, we obtain upon evaluating the second term,
	\begin{multline} \label{eq:Kncontour3_appendix}
\tilde K_{n,t}(x,y) = \frac{n}{(2\pi i)^2t}\int_{x_0-i\infty}^{x_0+i\infty}dz\int_\gamma dw 
\frac{1}{z-w}\frac{\prod_{j=1}^n(z-a_j^{(n)})}{\prod_{j=1}^n(w-a_j^{(n)})}\frac{e^{\frac{n}{2t} (z-x)^2}}{e^{\frac{n}{2t} 
(w-y)^2)}}\\
	+\frac{1}{\pi (x-y)}\sin\left((x-y)s\frac{n}{t}\right)e^{\frac{n}{2t}(x^2-y^2)}e^{-(x-y)x_0\frac{n}{t}}.
	\end{multline}

	Let $g$ be defined as
	\begin{equation}\label{defg}
	g_\mu(z)=\int\log(z-a)d\mu(a),
	\end{equation}
	for any compactly supported probability measure $\mu$ on $\mathbb R$.
	The function $g_\mu$ depends in general on the choice of the logarithm, if the support of $\mu$ is contained in an interval 
$I\subset\mathbb R$, then $e^{ng_\mu(z)}$ is independent of the choice of logarithm for $z$ outside $I$.

With this notion we can rewrite \eqref{eq:Kncontour3_appendix} as 
	\begin{multline} 
	\tilde K_{n,t}(x,y) = \frac{n}{(2\pi i)^2t}\int_{x_0-i\infty}^{x_0+i\infty}dz\int_\gamma dw \frac{1}{z-w}\frac{e^{\frac{n}{2t}\left( (z-x)^2+2t g_{\mu_n}(z)\right)}}{e^{\frac{n}{2t} \left((w-y)^2+2t g_{\mu_n}(w)\right)}}\\
	+\frac{1}{\pi (x-y)}\sin\left((x-y)s\frac{n}{t}\right)e^{\frac{n}{2t}(x^2-y^2)}e^{-(x-y)x_0\frac{n}{t}}.
	\end{multline}
Now,  $\tilde K_{n,t}$ and $K_{n,t}$ from \eqref{eq:Kncontour3} satisfy the relation 
\[\tilde K_{n,t}=K_{n,t}\frac{f(x)}{f(y)},\qquad f(x)=e^{\frac{n}{2t}(x^2-2xx_0)},\]
which means that they define the same determinantal point process.\\

{\bf Acknowledgements.} T.C.~and M.V.~were supported by the European Research Council under the European Union's Seventh Framework Programme (FP/2007/2013)/ ERC Grant Agreement 307074 and by the Belgian Interuniversity Attraction Pole P07/18. T.N.~is a Research Associate (charg\'e de recherches) of FRS-FNRS (Belgian Fund for Scientific Research).\\

\bibliographystyle{abbrv}
\bibliography{bibliography}

\begin{thebibliography}{10}

\bibitem{AdlerDelepinevanMoerbeke}
M.~Adler, J.~Del\'epine, and P.~van Moerbeke.
\newblock Dyson's nonintersecting {B}rownian motions with a few outliers.
\newblock {\em Comm. Pure Appl. Math.}, 62(3):334--395, 2009.

\bibitem{AdlerFerrarivanMoerbeke}
M.~Adler, P.~L. Ferrari, and P.~van Moerbeke.
\newblock Airy processes with wanderers and new universality classes.
\newblock {\em Ann. Probab.}, 38(2):714--769, 2010.

\bibitem{AdlerOrantinvanMoerbeke}
M.~Adler, N.~Orantin, and P.~van Moerbeke.
\newblock Universality for the {P}earcey process.
\newblock {\em Phys. D}, 239(12):924--941, 2010.

\bibitem{AEK}
O.~H. Ajanki, L.~Erd\H{o}s, and T.~Kr\"uger.
\newblock Local spectral statistics of {G}aussian matrices with correlated
  entries.
\newblock {\em J. Stat. Phys.}, 163(2):280--302, 2016.

\bibitem{AndersonGuionnetZeitouni}
G.~W. Anderson, A.~Guionnet, and O.~Zeitouni.
\newblock {\em An introduction to random matrices}, volume 118 of {\em
  Cambridge Studies in Advanced Mathematics}.
\newblock Cambridge University Press, Cambridge, 2010.

\bibitem{BenArousPeche}
G.~Ben~Arous and S.~P\'ech\'e.
\newblock Universality of local eigenvalue statistics for some sample
  covariance matrices.
\newblock {\em Comm. Pure Appl. Math.}, 58(10):1316--1357, 2005.

\bibitem{BertolaBuckinghamLeePierce1}
M.~Bertola, R.~Buckingham, S.~Y. Lee, and V.~Pierce.
\newblock Spectra of random {H}ermitian matrices with a small-rank external
  source: the critical and near-critical regimes.
\newblock {\em J. Stat. Phys.}, 146(3):475--518, 2012.

\bibitem{BertolaBuckinghamLeePierce2}
M.~Bertola, R.~Buckingham, S.~Y. Lee, and V.~Pierce.
\newblock Spectra of random {H}ermitian matrices with a small-rank external
  source: the supercritical and subcritical regimes.
\newblock {\em J. Stat. Phys.}, 153(4):654--697, 2013.

\bibitem{Biane}
P.~Biane.
\newblock On the free convolution with a semi-circular distribution.
\newblock {\em Indiana Univ. Math. J.}, 46(3):705--718, 1997.

\bibitem{BleherKuijlaars}
P.~M. Bleher and A.~B.~J. Kuijlaars.
\newblock Large {$n$} limit of {G}aussian random matrices with external source.
  {III}. {D}ouble scaling limit.
\newblock {\em Comm. Math. Phys.}, 270(2):481--517, 2007.

\bibitem{BEYY}
P.~Bourgade, L.~Erd\H{o}s, H.-T. Yau, and J.~Yin.
\newblock Fixed energy universality for generalized {W}igner matrices.
\newblock {\em Comm. Pure Appl. Math.}, 69(10):1815--1881, 2016.

\bibitem{BEYedge}
P.~Bourgade, L.~Erd\"os, and H.-T. Yau.
\newblock Edge universality of beta ensembles.
\newblock {\em Comm. Math. Phys.}, 332(1):261--353, 2014.

\bibitem{BrezinHikami}
E.~Br\'ezin and S.~Hikami.
\newblock Level spacing of random matrices in an external source.
\newblock {\em Phys. Rev. E (3)}, 58(6, part A):7176--7185, 1998.

\bibitem{BrezinHikami2}
E.~Br\'ezin and S.~Hikami.
\newblock Universal singularity at the closure of a gap in a random matrix
  theory.
\newblock {\em Phys. Rev. E (3)}, 57(4):4140--4149, 1998.

\bibitem{ClaeysKuijlaarsLiechtyWang}
T.~{Claeys}, A.~B.~J. {Kuijlaars}, K.~{Liechty}, and D.~{Wang}.
\newblock {Propagation of singular behavior for Gaussian perturbations of
  random matrices}.
\newblock {\em ArXiv e-prints}, Aug. 2016.

\bibitem{ClaeysKuijlaarsWang}
T.~Claeys, A.~B.~J. Kuijlaars, and D.~Wang.
\newblock Correlation kernels for sums and products of random matrices.
\newblock {\em Random Matrices Theory Appl.}, 4(4):1550017, 31, 2015.

\bibitem{ClaeysWang}
T.~Claeys and D.~Wang.
\newblock Random matrices with equispaced external source.
\newblock {\em Comm. Math. Phys.}, 328(3):1023--1077, 2014.

\bibitem{Deiftetal}
P.~Deift, T.~Kriecherbauer, K.~T.-R. McLaughlin, S.~Venakides, and X.~Zhou.
\newblock Uniform asymptotics for polynomials orthogonal with respect to
  varying exponential weights and applications to universality questions in
  random matrix theory.
\newblock {\em Comm. Pure Appl. Math.}, 52(11):1335--1425, 1999.

\bibitem{Dyson}
F.~J. Dyson.
\newblock A {B}rownian-motion model for the eigenvalues of a random matrix.
\newblock {\em J. Mathematical Phys.}, 3:1191--1198, 1962.

\bibitem{Erdosetal10}
L.~Erd\H{o}s, S.~P\'ech\'e, J.~A. Ram\'irez, B.~Schlein, and H.-T. Yau.
\newblock Bulk universality for {W}igner matrices.
\newblock {\em Comm. Pure Appl. Math.}, 63(7):895--925, 2010.

\bibitem{bookErdosYau}
L.~Erd\H{o}s and H.-T. Yau.
\newblock {\em A dynamical approach to random matrix theory}, volume~28 of {\em
  Courant Lecture Notes in Mathematics}.
\newblock Courant Institute of Mathematical Sciences, New York; American
  Mathematical Society, Providence, RI, 2017.

\bibitem{FG}
A.~Figalli and A.~Guionnet.
\newblock Universality in several-matrix models via approximate transport maps.
\newblock {\em Acta Math.}, 217(1):81--176, 2016.

\bibitem{GV14}
F.~G{\"o}tze and M.~Venker.
\newblock Local universality of repulsive particle systems and random matrices.
\newblock {\em Ann. Probab.}, 42(6):2207--2242, 2014.

\bibitem{HiaiPetz}
F.~Hiai and D.~Petz.
\newblock {\em The semicircle law, free random variables and entropy},
  volume~77 of {\em Mathematical Surveys and Monographs}.
\newblock American Mathematical Society, Providence, RI, 2000.

\bibitem{HLY}
J.~Huang, B.~Landon, and H.-T. Yau.
\newblock Bulk universality of sparse random matrices.
\newblock {\em J. Math. Phys.}, 56(12):123301, 19, 2015.

\bibitem{Johansson01}
K.~Johansson.
\newblock Universality of the local spacing distribution in certain ensembles
  of {H}ermitian {W}igner matrices.
\newblock {\em Comm. Math. Phys.}, 215(3):683--705, 2001.

\bibitem{Johansson2}
K.~Johansson.
\newblock Universality for certain {H}ermitian {W}igner matrices under weak
  moment conditions.
\newblock {\em Ann. Inst. Henri Poincar\'e Probab. Stat.}, 48(1):47--79, 2012.

\bibitem{Kuijlaars-Survey}
A.~B.~J. Kuijlaars.
\newblock Universality.
\newblock In {\em The {O}xford handbook of random matrix theory}, pages
  103--134. Oxford Univ. Press, Oxford, 2011.

\bibitem{LandonSosoeYau}
B.~{Landon}, P.~{Sosoe}, and H.-T. {Yau}.
\newblock {Fixed energy universality for Dyson Brownian motion}.
\newblock {\em ArXiv e-prints}, Sept. 2016.

\bibitem{LandonYau}
B.~Landon and H.-T. Yau.
\newblock Convergence of local statistics of {D}yson {B}rownian motion.
\newblock {\em Comm. Math. Phys.}, 355(3):949--1000, 2017.

\bibitem{LandonYauedge}
B.~{Landon} and H.-T. {Yau}.
\newblock {Edge statistics of Dyson Brownian motion}.
\newblock {\em ArXiv e-prints}, Dec. 2017.

\bibitem{LeeSchnelliStetlerYau}
J.~O. Lee, K.~Schnelli, B.~Stetler, and H.-T. Yau.
\newblock Bulk universality for deformed {W}igner matrices.
\newblock {\em Ann. Probab.}, 44(3):2349--2425, 2016.

\bibitem{LevinLubinsky}
E.~Levin and D.~S. Lubinsky.
\newblock Universality limits at the soft edge of the spectrum via classical
  complex analysis.
\newblock {\em Int. Math. Res. Not. IMRN}, (13):3006--3070, 2011.

\bibitem{Lubinsky-Survey}
D.~S. Lubinsky.
\newblock An update on local universality limits for correlation functions
  generated by unitary ensembles.
\newblock {\em SIGMA Symmetry Integrability Geom. Methods Appl.}, 12:Paper No.
  078, 36, 2016.

\bibitem{PS97}
L.~Pastur and M.~Shcherbina.
\newblock Universality of the local eigenvalue statistics for a class of
  unitary invariant random matrix ensembles.
\newblock {\em J. Statist. Phys.}, 86(1-2):109--147, 1997.

\bibitem{PS08}
L.~Pastur and M.~Shcherbina.
\newblock Bulk universality and related properties of {H}ermitian matrix
  models.
\newblock {\em J. Stat. Phys.}, 130(2):205--250, 2008.

\bibitem{SV15}
K.~Schubert and M.~Venker.
\newblock Empirical spacings of unfolded eigenvalues.
\newblock {\em Electron. J. Probab.}, 20:Paper No. 120, 37, 2015.

\bibitem{Shcherbina}
T.~Shcherbina.
\newblock On universality of bulk local regime of the deformed {G}aussian
  unitary ensemble.
\newblock {\em Zh. Mat. Fiz. Anal. Geom.}, 5(4):396--433, 440, 2009.

\bibitem{Shcherbina2}
T.~Shcherbina.
\newblock On universality of local edge regime for the deformed {G}aussian
  unitary ensemble.
\newblock {\em J. Stat. Phys.}, 143(3):455--481, 2011.

\bibitem{Simon}
B.~Simon.
\newblock {\em Trace ideals and their applications}, volume 120 of {\em
  Mathematical Surveys and Monographs}.
\newblock American Mathematical Society, Providence, RI, second edition, 2005.

\bibitem{Speicher}
R.~Speicher.
\newblock Free convolution and the random sum of matrices.
\newblock {\em Publ. Res. Inst. Math. Sci.}, 29(5):731--744, 1993.

\bibitem{TaoVu}
T.~Tao and V.~Vu.
\newblock Random matrices: universality of local eigenvalue statistics.
\newblock {\em Acta Math.}, 206(1):127--204, 2011.

\bibitem{TracyWidom06}
C.~A. Tracy and H.~Widom.
\newblock The {P}earcey process.
\newblock {\em Comm. Math. Phys.}, 263(2):381--400, 2006.

\bibitem{VoiculescuDykemaNica}
D.~V. Voiculescu, K.~J. Dykema, and A.~Nica.
\newblock {\em Free random variables}, volume~1 of {\em CRM Monograph Series}.
\newblock American Mathematical Society, Providence, RI, 1992.
\newblock A noncommutative probability approach to free products with
  applications to random matrices, operator algebras and harmonic analysis on
  free groups.

\end{thebibliography}

\end{document}